\documentclass[a4paper,UKenglish,cleveref, autoref, thm-restate,nolineno]{socg-lipics-v2021}

\usepackage{caption}
\usepackage{apxproof}

\usepackage[utf8]{inputenc}
\usepackage{lipsum}
\usepackage{multicol}
\usepackage{float}
\usepackage{hyperref}
\usepackage{stmaryrd}
\usepackage{stmaryrd}
\usepackage{amsmath}
\usepackage{amsfonts}
\usepackage{amssymb}
\usepackage{graphicx}
\usepackage[all]{xy}
\usepackage{mathtools}
\usepackage{url}
\usepackage[vlined,linesnumbered]{algorithm2e}

\usepackage{yfonts}
\usepackage{algpseudocode}

\usepackage{tikz}
\usepackage{tikz-cd}
\usepackage{tkz-euclide}
\usepackage{pgfplots}

\usepackage[titletoc]{appendix}

\definecolor{darkdgmcolor}{rgb}{0.0, 0.0, 0.8}
\definecolor{darkgreen}{rgb}{0.0, 0.8, 0.0}
\definecolor{purple}{RGB}{153,50,204}
\definecolor{dgmcolor}{RGB}{255,20,147}

\newcommand{\e}{\varepsilon}

\newcommand{\R}{\mathbb{R}}
\newcommand{\X}{\mathbb{X}}
\newcommand{\Y}{\mathbb{Y}}
\newcommand{\Z}{\mathbb{Z}}

\newcommand{\bbS}{\mathbb{S}}
\newcommand{\T}{\mathbb{T}}

\newcommand{\N}{\mathbb{N}}

\newcommand{\CC}{\mathcal{C}}

\newcommand{\DD}{\mathcal{D}}
\newcommand{\caB}{\mathcal{B}}

\newcommand{\Zfunc}{\mathbf{Z}}

\newcommand{\Ffunc}{\mathbf{F}}
\newcommand{\Gfunc}{\mathbf{G}}

\newcommand{\Jfunc}{\mathbf{J}}

\newcommand{\Hfunc}{\mathbf{H}}

\newcommand{\Xfunc}{\mathbf{X}}

\newcommand{\Vfunc}{\mathbf{V}}
\newcommand{\Wfunc}{\mathbf{W}}
\newcommand{\Yfunc}{\mathbf{Y}}
\newcommand{\Rfunc}{\mathbf{R}}

\newcommand{\im}{\mathbf{Im}}

\newcommand{\VR}{\mathbf{VR}}
\newcommand{\diam}{\operatorname{diam}}
\newcommand{\rank}{\operatorname{rank}}
\newcommand{\len}{\mathbf{len}}
\newcommand{\cupprod}{\mathbf{cup}}
\newcommand{\pcupprod}{\mathbf{cup}}

\newcommand{\di}{d_{\mathrm{I}}}

\newcommand{\bsigma}{\boldsymbol{\sigma}}
\newcommand{\btau}{\boldsymbol{\tau}}
\newcommand{\udim}{k}
\newcommand{\card}{\operatorname{card}}

\newcommand{\Vect}{\mathbf{Vec}}

\newcommand{\Ring}{\mathbf{GRing}}
\newcommand{\Alg}{\mathbf{GAlg}}
\newcommand{\Top}{\mathbf{Top}}

\newcommand{\Int}{\mathbf{Int}}

\newcommand{\ob}{\mathbf{Ob}}
\newcommand{\rk}{\mathbf{rank}}

\newcommand{\Coordinate}[2]%
{ \coordinate (#1) at (#2);
}
\usepackage{etoolbox}

\title{Persistent cup-length} 

\author{Marco Contessoto}{Department of Mathematics, S\~ao Paulo State University -  UNESP, Brazil \and \url{http://www.myhomepage.edu} }{marco.contessoto@unesp.br}{ }{MC was supported by FAPESP through grants 2016/24707-4, 2017/25675-1 and 2019/22023-9.}

\author{Facundo M\'emoli}{Department of Mathematics and Department of Computer Science and Engineering, The Ohio State University, US \and \url{http://facundo-memoli.org/} }{memoli@math.osu.edu}{ }{FM was partially supported by the NSF through grants RI-1901360, CCF-1740761, and CCF-1526513, and DMS-1723003.}

\author{Anastasios Stefanou}{Department of Mathematics and Computer Science, University of Bremen, Germany \and \url{https://sites.google.com/view/anastasiostefanou} }{stefanou@uni-bremen.de}{ }{AS was supported by NSF through grants CCF-1740761, DMS-1440386, RI-1901360, and the Dioscuri program initiated by the Max Planck Society, jointly managed with the National Science Centre (Poland), and mutually funded by the
Polish Ministry of Science and Higher Education and the German Federal
Ministry of Education and Research.}

\author{Ling Zhou}{Department of Mathematics, The Ohio State University, US \and \url{https://sites.google.com/view/lingzhou-math/home} }{zhou.2568@osu.edu}{https://orcid.org/0000-0001-6655-5162}{LZ was partially supported by the NSF through grants RI-1901360, CCF-1740761, and CCF-1526513, and DMS-1723003.}

\authorrunning{M. Contessoto, F. M\'emoli, A. Stefanou and L. Zhou}

\Copyright{Marco Contessoto, Facundo M\'emoli, Anastasios Stefanou and Ling Zhou}



\EventEditors{Xavier Goaoc and Michael Kerber}
\EventNoEds{2}
\EventLongTitle{38th International Symposium on Computational Geometry (SoCG 2022)}
\EventShortTitle{SoCG 2022}
\EventAcronym{SoCG}
\EventYear{2022}
\EventDate{June 7--10, 2022}
\EventLocation{Berlin, Germany}
\EventLogo{}
\SeriesVolume{224}
\ArticleNo{31}

\begin{document}


\nolinenumbers

\maketitle

\begin{abstract}
Cohomological ideas have recently been injected into persistent homology and have for example  been used for accelerating the calculation of persistence diagrams by the software Ripser.

The cup product operation which is available at cohomology level gives rise to a graded ring structure that extends the usual vector space structure and is therefore able to extract and encode additional rich information. The maximum number of cocycles having non-zero cup product yields an invariant, the cup-length, which is useful for discriminating spaces.

In this paper, we lift the cup-length into the persistent cup-length function for the purpose of capturing ring-theoretic information about the evolution of the cohomology (ring) structure across a filtration. We show that the persistent cup-length function can be computed from a family of representative cocycles and devise a polynomial time algorithm for its computation. We furthermore show that this invariant is stable under suitable interleaving-type distances. 
\end{abstract}

\section{Introduction}


\textit{Persistent Homology} \cite{frosini1990distance,frosini1992measuring,robins1999towards,zomorodian2005computing,cohen2007Stability,edelsbrunner2008Persistent,carlsson2009topology,carlsson2020persistent}, one of the main techniques in \textit{Topological Data Analysis (TDA)},  studies the evolution of homology classes across a filtration. This produces a collection of birth-death pairs which is called the \textit{barcode} or \textit{persistence diagram} of the filtration.

In the case of \textit{cohomology}, which is dual to that of homology, one studies linear maps from the vector space of simplicial chains into the field $K$, known as \textit{cochains}. 
Cochains are naturally endowed with a product operation, called the \textit{cup product}, which induces a bilinear operation on cohomology and is denoted by $\smile:\Hfunc^p(\X)\times\Hfunc^q(\X)\to \Hfunc^{p+q}(\X)$ for a space $\X$ and dimensions $p,q\geq 0$.
With the cup product operation, the collection of cohomology vector spaces can be given the structure of a \textit{graded ring}, called the \textit{cohomology ring}; see \cite[\textsection~48 and \textsection~68]{munkres1984elements} and \cite[Ch. 3, \textsection 3.D]{hatcher2000}. 
This makes cohomology a richer structure than homology. 

\textit{Persistent cohomology} has been studied in \cite{de2011dualities,de2011persistent,dlotko2014simplification,bauer2021ripser,kang2021evaluating},
without exploiting the ring structure induced by the cup product. Works which do attempt to exploit this ring structure include \cite{gonzalez2003hha,kaczynski2010computing} in the static case 
and 
\cite{huang2005cup,yarmola2010persistence,aubrey2011persistent,lupo2018persistence,herscovich2018higher,belchi2021a,polanco2021} at the persistent level.

In this paper, we continue this line of work and tackle the question of quantifying the \emph{evolution} of the cup product structure across a filtration through introducing a \emph{polynomial-time computable} invariant which is induced from the \emph{cup-length}: the maximal number of cocycles (in dimensions 1 and above) having non-zero cup product. We call this invariant the \textit{persistent cup-length function}, and identify a tool - the \textit{persistent cup-length diagram (associated to a family of representative cocycles $\bsigma$ of the barcode)} to compute it. (see Fig.~\ref{fig:klein-function-diagram}).

\begin{figure}[H]
\hspace{-1em}
    \begin{tikzcd}[column sep=0pt,row sep=0pt]
\begin{tikzpicture}[scale=0.25]
    \node[fill=red,circle,inner sep=.4mm] at (1.4,5)  {};
\end{tikzpicture}
	& 
	\begin{tikzpicture}[scale=0.25]
	   \filldraw[color=red!40, fill=none, thick](2.5,5) ellipse (1.2 and 0.5);
    \node[fill=red,circle,inner sep=.4mm] at (1.4,5)  {};
\end{tikzpicture}
	& 
\begin{tikzpicture}[scale=0.17]
	\filldraw[color=red!40, fill=none, thick](2.7,9) ellipse (1 and 0.5);
    \node[fill=red,circle,inner sep=.4mm] at (1.8,9)  {};
    \Coordinate{innerr}{4.5,6} 
    \Coordinate{innerl}{2.5,5}
    \Coordinate{intersectionr}{4.2,7} 
    \Coordinate{intersectionl}{2.1,6} 
    \Coordinate{bottoml}{2.2,3.3} 
    \Coordinate{bottomr}{7,4} 
    \Coordinate{left}{2,2} 
    \Coordinate{right}{9.9,4.2} 
        \draw (left) to[out=150,in=210,looseness=1.3] (intersectionl); 
        \draw[dashed] (intersectionl) to[out=30,in=200,looseness=0.6] (intersectionr);
    
        \draw (innerr) to[out=120,in=250,looseness=1] (4,10) to[out=70,in=180,looseness=1] (5.2,11) to[out=360,in=90,looseness=1] (7.5,9) to[out=270,in=20,looseness=1] (intersectionr);
        
        \draw  (innerl) to[out=110,in=80,looseness=4] (right);
        
        \draw (innerl) to[out=315,in=315,looseness=0.5] (innerr); 
        \draw[dashed] (innerl) to[out=135,in=135,looseness=0.5] node[below right,pos=0.5,black] {} (innerr); 
        \draw (left) to[out=330,in=260,looseness=1.1] (right);
        \draw[dashed] (innerl) to[out=300,in=60,looseness=1] (bottoml);
        \draw[dashed] (innerr) to[out=300,in=150,looseness=1] (bottomr);
        \draw[dashed] (1.6,3) to [out=330,in=240,looseness=.5] 
        (bottoml) to [out=30,in=160,looseness=1.4] 
        (bottomr) to [out=340,in=190,looseness=.5]  (7.7,3.8) ;
\end{tikzpicture}
& 
\hspace{-1em}
\begin{tikzpicture}[scale=0.17]
	\filldraw[color=red!40, fill=red!10, thick](2.7,9) ellipse (1 and 0.5);
    \node[fill=red,circle,inner sep=.4mm] at (1.8,9)  {};
    \Coordinate{innerr}{4.5,6} 
    \Coordinate{innerl}{2.5,5}
    \Coordinate{intersectionr}{4.2,7} 
    \Coordinate{intersectionl}{2.1,6} 
    \Coordinate{bottoml}{2.2,3.3} 
    \Coordinate{bottomr}{7,4} 
    \Coordinate{left}{2,2} 
    \Coordinate{right}{9.9,4.2} 
        \draw (left) to[out=150,in=210,looseness=1.3] (intersectionl); 
        \draw[dashed] (intersectionl) to[out=30,in=200,looseness=0.6] (intersectionr);
    
        \draw (innerr) to[out=120,in=250,looseness=1] (4,10) to[out=70,in=180,looseness=1] (5.2,11) to[out=360,in=90,looseness=1] (7.5,9) to[out=270,in=20,looseness=1] (intersectionr);
        
        \draw  (innerl) to[out=110,in=80,looseness=4] (right);
        
        \draw (innerl) to[out=315,in=315,looseness=0.5] (innerr); 
        \draw[dashed] (innerl) to[out=135,in=135,looseness=0.5] node[below right,pos=0.5,black] {} (innerr); 
        \draw (left) to[out=330,in=260,looseness=1.1] (right);
        \draw[dashed] (innerl) to[out=300,in=60,looseness=1] (bottoml);
        \draw[dashed] (innerr) to[out=300,in=150,looseness=1] (bottomr);
        \draw[dashed] (1.6,3) to [out=330,in=240,looseness=.5] 
        (bottoml) to [out=30,in=160,looseness=1.4] 
        (bottomr) to [out=340,in=190,looseness=.5]  (7.7,3.8) ;
\end{tikzpicture}
	\\
    \text{\small{$t\in[0,1)$}}  
    & \text{\small{$t\in[1,2)$}} 
    & \text{\small{$t\in [2,3)$}} 
    & \text{\small{$t\geq 3$}} 
  \end{tikzcd}  
    \begin{tikzcd}[column sep=tiny,row sep=tiny]
    &\\
      \vspace{-1em} 
      \begin{tikzpicture}[scale=0.4]
    \begin{axis} [ 
    ticklabel style = {font=\LARGE},
    axis y line=middle, 
    axis x line=middle,
    ytick={1,2,3,4,5.3},
    yticklabels={$1$,$2$,$3$,$4$,$\infty$},
    xticklabels={$1$,$2$,$3$,$4$,$\infty$},
    xtick={1,2,3,4,5.3},
    xmin=0, xmax=5.5,
    ymin=0, ymax=5.5,]
    \addplot [mark=none] coordinates {(0,0) (5.3,5.3)};
    \addplot [thick,color=dgmcolor!20!white,fill=dgmcolor!20!white, 
                    fill opacity=0.45]coordinates {
            (1,3) 
            (1,1)
            (2,2)  
            (2,3)
            (1,3)};
    \addplot [thick,color=dgmcolor!40!white,fill=dgmcolor!40!white, 
                    fill opacity=0.45]coordinates {
            (2,5.3) 
            (2,2)
            (2,2)
            (5.3,5.3)};
    \node[mark=none] at (axis cs:1.5,2.2){\LARGE\textsf{1}};
    \node[mark=none] at (axis cs:3.5,4.5){\LARGE\textsf{2}};
    \node[mark=none] at (axis cs:4,2){\huge$\cupprod(\Xfunc)$};
    \end{axis}
    \end{tikzpicture}
    \begin{tikzpicture}[scale=0.4]
    \begin{axis} [ 
    ticklabel style = {font=\LARGE},
    axis y line=middle, 
    axis x line=middle,
    ytick={1,2,3,4,5.3},
    yticklabels={$1$,$2$,$3$,$4$,$\infty$},
    xticklabels={$1$,$2$,$3$,$4$,$\infty$},
    xtick={1,2,3,4,5.3},
    xmin=0, xmax=5.5,
    ymin=0, ymax=5.5,]
    \addplot [mark=none] coordinates {(0,0) (5.3,5.3)};
  \addplot[dgmcolor!40!white,mark=*] (1,3) circle (3pt) node[below,black]{\LARGE\textsf{1}};
    \addplot[dgmcolor!100!white,mark=*] (2,3) circle (3pt) node[below,black]{\LARGE\textsf{2}};
  \addplot[dgmcolor!100!white,mark=*] (2,5.3) circle (3pt) node[below,black]{\LARGE\textsf{2}};
    \node[mark=none] at (axis cs:4,2){\huge$\mathbf{dgm}_{\bsigma}^{\smile}(\Xfunc)$};
    \end{axis}
    \end{tikzpicture}
	\end{tikzcd}
\caption{A filtration $\Xfunc$ of the pinched Klein bottle, its persistent cup-length function $\cupprod(\Xfunc)$ (see Ex.\ref{ex:cup_l_func_klein}) and its persistent cup-length diagram $\mathbf{dgm}_{\bsigma}^{\smile}(\Xfunc)$ (see Ex.~\ref{ex:example_dgm}). 
} 
\label{fig:klein-function-diagram}
\end{figure}

\subparagraph*{Some invariants related to the cup product.} In standard topology, an \textit{invariant} is a quantity associated to a given topological space which remains invariant under a certain class of maps. This invariance helps in discovering, studying and classifying properties of spaces. Beyond \emph{Betti numbers}, examples of classical invariants are: the \emph{Lusternik-Schnirelmann category (LS-category)} of a space $\X$, defined as the minimal integer $k\geq 1$ such that there is an open cover $\{U_i\}_{i=1}^k$ of $\X$ such that each inclusion map $U_i\hookrightarrow \X$ is null-homotopic, and the \emph{cup-length invariant},  which is the maximum number of positive-dimensional cocycles having non-zero cup product.
While being relatively more informative, the LS-category is difficult to compute \cite{cornea2003lusternik}, and with rational coefficients this computation is known to be NP-hard \cite{amann_2015}. The cup-length invariant, as a lower bound of the LS-category \cite{rudyak1999analytical,Rudyak1999ONCW}, serves as a computable estimate for the LS-category.
Another well known invariant which can be estimated through the cup-length is the so-called \textit{topological complexity} \cite{SMALE198781,farber2003topological,sarin2017cup}.


\paragraph*{Our contributions.}
Let $\Top$ denotes the category of (compactly generated weak Hausdorff) topological spaces.\footnote{We are following the convention from \cite{blumberg2017universality}.} Throughout the paper, by a (topological) space we refer to an object in $\Top$, and by a persistent space we mean a functor from the poset category $(\R,\leq )$ to $\Top$. A filtration (of spaces) is an example of a persistent space where the transition maps are given by inclusions. This paper considers only persistent spaces with a discrete set of critical values. In addition, all (co)homology groups are assumed to be taken over a field $K$.
We denote by $\Int_\omega$ the set of intervals of type $\omega$, where $\omega$ can be 
any one of the four types: open-open, open-closed, closed-open and closed-closed. The type $\omega$ will be omitted when the results apply to all four situations and intervals are written in the form of $\langle a,b \rangle.$ 

We introduce the invariant, the persistent cup-length function of general persistent spaces, by lifting the standard cup-length invariant into the persistent setting.
Let $\Xfunc:(\R,\leq)\to \Top$ be a persistent space with $t\mapsto \X_t$. The \textbf{persistent cup-length function} $\cupprod(\Xfunc):\Int\to\N$ of $\Xfunc$, see Defn.~\ref{def:cup_l_func},
is defined as the function 
from the set $\Int$ to the set $\N$ of non-negative integers, which assigns to each interval $\langle a,b\rangle$ the cup-length of the image ring\footnote{For $f:R\to S$ a graded ring morphism, we denote the graded ring $f(R)$ by $\mathbf{Im}(f)$.} $\mathbf{Im}\big(\Hfunc^*(\Xfunc)\langle a,b\rangle\big)$, which is the ring $\mathbf{Im}\big(\Hfunc^*(\X_b)\to\Hfunc^*(\X_a)\big)$ when $\langle a,b\rangle$ is a closed interval (in other cases, there is some subtlety, see Rmk.~\ref{rmk:image-ring}). Note that the persistent cup-length function is a generalization of the cup-length of spaces, since $\cupprod(\Xfunc)([a,a])$ reduces to the cup-length of the space $\X_a$. 

In the case when $\Xfunc$ is a filtration, we define a notion of a diagram to compute the persistent cup-length function (see Thm.~\ref{thm:tropical_mobius}): the \textbf{persistent cup-length diagram} $\mathbf{dgm}_{\bsigma}^{\smile}(\Xfunc):\Int\to\N$ (Defn.~\ref{def:cup_dgm}). We first assign a representative cocycle to every interval in the barcode of $\Xfunc$, and denote the family of representative cocycles by $\bsigma$. Then, the persistent cup-length diagram of an interval $\langle a,b\rangle$ is defined to be the maximum number of representative cocycles in $\Xfunc$ that have a nonzero cup product over $\langle a,b\rangle$.
It is worth noticing that the persistent cup-length diagram depends on the choice of representative cocycles; see Ex.~\ref{ex:dependence on sigma}. 

\begin{restatable}{theorem}{mobuis}
\label{thm:tropical_mobius}
Let $\Xfunc$ be a filtration, and let $\bsigma$ be a family of representative cocycles for the barcodes of $\Xfunc$.
The persistent cup-length function $\cupprod(\Xfunc)$ can be retrieved from the persistent cup-length diagram $\mathbf{dgm}_{\bsigma}^{\smile}(\Xfunc)$: for any $\langle a,b\rangle\in \Int$,
\begin{equation}\label{eq:tropical_mobuis}
    \cupprod(\Xfunc)(\langle a,b\rangle)=  \max_{\langle c,d\rangle\supseteq \langle a,b\rangle}\mathbf{dgm}_{\bsigma}^{\smile}(\Xfunc)(\langle c,d\rangle). 
\end{equation}
\end{restatable}
The persistent cup-length functions do not supersede the standard persistence diagrams, partly because they do not take $\Hfunc^0$ classes into account. However, it effectively augments the standard diagram in the sense that there are situations in which it can successfully capture information that standard persistence diagrams neglect (see Fig.~\ref{fig:doublefilt}). 
Our work therefore provides additional \emph{computable} persistence-like invariants enriching the TDA toolset which can be used in applications requiring discriminating between different hypotheses such as in shape classification or machine learning. For example, \cite{kang2021evaluating} mentions that cup product could provide additional evidence when recovering the structure of animal trajectories.

\subparagraph*{A polynomial time algorithm.}

We develop a poly-time algorithm (Alg.~\ref{alg:main}) to compute the persistent cup-length diagram of a filtration $\Xfunc$ of a simplicial complex $\X$ of dimension $(k+1)$. This algorithm is output sensitive, and it has complexity bounded above by $O( (m_k )^2\cdot q_1 \cdot q_{k-1}  \cdot \max\{c_k ,q_1\})\leq O( (m_\udim)^{k+3})$
(cf. Thm.~\ref{thm:comp-alg3}), with $m_{\udim}$ being the cardinality of $\X$, $q_1$ being the cardinality of the barcode, and
parameters $q_{k-1}$ ($\leq q_1^{k-1}$) and $c_\udim$ ($\leq m_{\udim}$) which we describe in \S \ref{sec:algorithms} on page~\pageref{para:parameter}.
In the case of the Vietoris-Rips filtration of an $n$-point metric space, this complexity is improved to $O( (m_k )^2\cdot q_1^2 \cdot q_{k-1})$, which can be upper bounded by $O(n^{\udim^2+5\udim+6})$.  

\subparagraph*{Gromov-Hausdorff stability and discriminating power.} 
In Thm.~\ref{thm:main-stability} we prove that the persistent cup-length function is stable to perturbations of the involved filtrations (in a suitable sense involving weak homotopy equivalences).  Below, $d_{\mathrm{E}},d_{\mathrm{HI}}$ and $d_{\mathrm{GH}}$ denote the erosion, homotopy-interleaving and Gromov-Hausdorff distances, respectively. See \S \ref{sec:stability}
for details.

In general, the Gromov-Hausdorff distance is NP-hard to compute \cite{Schmiedl17} whereas the erosion distance is computable in polynomial time (see \cite[Thm. 5.4]{kim2021spatiotemporal}) and thus, in combination with Thm. \ref{thm:main-stability}, provides a computable estimate for the Gromov-Hausdorff distance.

\begin{restatable}[Homotopical stability]{theorem}{homostab} 
\label{thm:main-stability}
For two persistent spaces $\Xfunc,\mathbf{Y}:(\R,\leq)\to\Top$, 
\begin{equation}\label{eq:dE-dHI}
    d_{\mathrm{E}}(\cupprod(\Xfunc),\cupprod(\mathbf{Y}))\leq d_{\mathrm{HI}}(\Xfunc,\mathbf{Y}).
\end{equation}
For the Vietoris-Rips filtrations $\VR(X)$ and $\VR(Y)$ of compact metric spaces $X$ and $Y$,
\begin{equation}\label{eq:dE-dGH}
d_{\mathrm{E}}\left(\cupprod\left(\mathbf{VR}(X)\right),\cupprod\left(\mathbf{VR}(Y)\right)\right)\leq 2\cdot d_{\mathrm{GH}}(X,Y).
\end{equation}
\end{restatable}

Through several examples, we show that the persistent cup-length function helps in discriminating filtrations when the persistent homology fails to or has a relatively weak performance in doing so.  Ex.~\ref{example:samePH} is a situation when two filtrations have identical persistent homology but induce different persistent cup-length functions. 
In addition, in 
Ex.~\ref{ex:T2-wedge} 
by specifying suitable metrics on the torus $\T^2$ and on the wedge sum $\bbS^1 \vee \bbS^2 \vee \bbS^1$, we compute the erosion distance between their persistent cup-length functions (see Fig.~\ref{fig:torus_sphere_VR}) and apply Thm.~\ref{thm:main-stability} to obtain a lower bound $\frac{\pi}{3}$ for the Gromov-Hausdorff distance between them $\T^2$ and $\bbS^1 \vee \bbS^2 \vee \bbS^1$ (see Prop.~\ref{prop:erosion-comp}):
\[\frac{\pi}{3}=d_{\mathrm{E}}\left(\cupprod(\VR(\T^2)),\cupprod(\VR(\bbS^1\vee\bbS^2\vee\bbS^1))\right)\leq 2\cdot d_{\mathrm{GH}}\left(\T^2,\bbS^1\vee\bbS^2\vee\bbS^1\right).\]
We also
verify that the interleaving distance between the persistent homology of these two spaces is at most $\frac{3}{5}$ of the bound obtained from persistent cup-length functions (a fact which we also establish). See 
Rmk.~\ref{rmk:inter-torus-wedge}. 

\begin{figure}[H]
\centering
    \begin{tikzpicture}[scale=0.65]
    \begin{axis} [ 
    ticklabel style = {font=\Large},
    axis y line=middle, 
    axis x line=middle,
    ytick={0.5,0.67,1},
    yticklabels={$\frac{\pi}{2}$, $\frac{2\pi}{3}$,$\pi$},
    xtick={0.5,0.67,1},
    xticklabels={$\frac{\pi}{2}$,$\frac{2\pi}{3}$, $\pi$},
    xmin=0, xmax=1.1,
    ymin=0, ymax=1.1,]
   \addplot [mark=none,color=dgmcolor!40!white] coordinates {(0,0) (1,1)};
    \addplot [thick,color=dgmcolor!40!white,fill=dgmcolor!40!white, 
                    fill opacity=0.45]coordinates {
            (0,.67) 
            (0,0)
            (.67,.67)
            (0,.67)};
    \addplot [thick,color=dgmcolor!40!white,fill=dgmcolor!40!white, 
                    fill opacity=0.45]coordinates {
            (0.67,.8) 
            (.67,.67)
            (.8,.8)
            (0.67,.8)};
    \addplot [thick,color=dgmcolor!40!white,fill=dgmcolor!40!white, 
                    fill opacity=0.45]coordinates {
            (0.67,.8) 
            (.67,.67)
            (.8,.8)
            (0.67,.8)};
    \addplot [thick,color=dgmcolor!40!white,fill=dgmcolor!40!white, 
                    fill opacity=0.45]coordinates {
            (0.8,.857) 
            (.8,.8)
            (.857,.857)
            (0.8,.857)};
    \addplot [thick,color=dgmcolor!40!white,fill=dgmcolor!40!white, 
                    fill opacity=0.45]coordinates {
            (0.857,.89) 
            (.857,.857)
            (.89,.89)
            (0.857,.89) };
    \addplot [thick,color=dgmcolor!40!white,fill=dgmcolor!40!white, 
                    fill opacity=0.45]coordinates {
            (0.89,.91) 
            (.89,.89)
            (.91,.91)
            (0.89,.91) };
    \node[mark=none] at (axis cs:.74,.76){\tiny{\textsf{2}}};
    \node[mark=none] at (axis cs:.25,.45){\textsf{2}};
    \end{axis}
    \end{tikzpicture}
    \hspace{1.5cm}
    \begin{tikzpicture}[scale=0.65]
    \begin{axis} [ 
    ticklabel style = {font=\Large},
    axis y line=middle, 
    axis x line=middle,
    ytick={0.5,0.67,1},
    yticklabels={$\frac{\pi}{2}$,$\frac{2\pi}{3}$,$\pi$},
    xtick={0.5,0.6,0.67,1},
    xticklabels={$\frac{\pi}{2}$,$\zeta$, $\frac{2\pi}{3}$, $\pi$},
    xmin=0, xmax=1.1,
    ymin=0, ymax=1.1,]
    \addplot [mark=none] coordinates {(0,0) (1,1)};
    \addplot [thick,color=dgmcolor!20!white,fill=dgmcolor!20!white, 
                    fill opacity=0.45]coordinates {
            (0,0.6)
            (0,0)
            (0.6,0.6)
            (0,0.6)};
    \node[mark=none] at (axis cs:.25,.45){\textsf{1}};
    \end{axis}
    \end{tikzpicture}
\caption{The persistent cup-length functions $\cupprod(\VR(\T^2))$ (left) and $\cupprod(\VR(\bbS^1\vee\bbS^2\vee\bbS^1))|_{(0,\zeta)}$ (right), respectively. Here, $\zeta=\arccos(-\tfrac{1}3)\approx 0.61\pi$.} 
\label{fig:torus_sphere_VR}
\end{figure}

\medskip
Proofs of all the theorems and results mentioned above are available in the appendix. 

\section{Persistent cup-length function}
\label{sec:Persistent-cup-length}
In the standard setting of persistent homology, one considers a \textit{filtration} of spaces, i.e.~a collection of spaces $\Xfunc=\{\X_t\}_{t\in\R}$ such that $\X_t\subset \X_s$ for all $t\leq s$, and studies the \textit{$p$-th persistent homology} for any given dimension $p$, defined as the functor $\Hfunc_p(\Xfunc):(\R,\leq)\to\Vect$ which sends each $t$ to the $p$-th homology $\Hfunc_p(\X_t)$ of $\X_t$, see \cite{edelsbrunner2008Persistent,carlsson2009topology}.
Here $\Vect$ denotes the category of vector spaces.
The $p$-th persistent homology encodes the lifespans, represented by intervals, of the $p$-dimensional holes ($p$-cycles that are not $p$-boundaries) in $\Xfunc$. 
The collection $\caB_p(\Xfunc)$ of these intervals is called \textit{the $p$-th barcode of $\Xfunc$}, and its elements are named \textit{bars}. 
The \emph{$p$-th persistent cohomology $\Hfunc^p(\Xfunc)$} and its corresponding barcode are defined dually. Since persistent homology and persistent cohomology have the same barcode \cite{de2011dualities}, we will denote both barcodes by $\caB_p(\Xfunc)$ for dimension $p$.
We call the disjoint union $\caB(\Xfunc):=\sqcup_{p\in \N}\caB_p(\Xfunc)$ \emph{the total barcode of $\Xfunc$}, and assume bars in $\caB(X)$ are of the same interval type\footnote{In TDA it is often the case that bars are of a fixed interval type, usually in closed-open form \cite{patel2018generalized}.}. 

By considering the \textit{cup product} operation on cocycles, 
the persistent cohomology is naturally enriched 
with the structure of a \emph{persistent graded ring}, which carries additional information and leads to invariants stronger than standard barcodes in cases like Ex.~\ref{example:samePH}.

In \S \ref{sec:cup_legnth} we recall the cup product operation, as well as the notion and properties of the cup-length invariant of cohomology rings. In \S \ref{sec:cup_l_func} we lift the cup-length invariant to a persistent invariant, called the persistent cup-length function, and  examine some examples that highlight its strength. 
Proofs and details are available in 
\textsection \ref{sec:supp_cup_l_func}.

\subsection{Cohomology rings and the cup-length invariant}\label{sec:cup_legnth}

For a topological space $\X$ and a dimension $p\in \N$, denote by $C_p(\X)$ and $C^p(\X)$ the spaces of singular $p$-chains and $p$-cochains, respectively. For a cocycle $\sigma$, denote by $[\sigma]$ the cohomology class of $\sigma$. If $\X$ is given by the geometric realization of some simplicial complex, then we consider its simplicial cohomology, by assuming an ordering on the vertex set of $\X$ and considering its simplices to be sets of ordered vertices.

Let $\Xfunc:=\{\X_t\}_{t\in\R}$ be a filtration of topological spaces, and let
$I=\langle b,d\rangle\in\caB_p(\Xfunc)$. If $I$ is closed at its right end $d$, we denote by $\sigma_I$ a cocycle in $C^p(\X_d)$; 
if not, we denote by $\sigma_I$ a cocycle in $C^p(\X_{d-\delta})$ for sufficiently small $\delta>0$. For any $t\leq d$, denote by $\sigma_I|_{C_p(\X_t)}$ the restriction of $\sigma_I$ to $C_p(\X_t)(\subset C_p(\X_d)).$
We introduce the notation $[\sigma_I]_t$ by defining $[\sigma_I]_t$ to be $\left[\sigma_I|_{C_p(\X_t)}\right]$ for $t\leq d$ and $0$ for $t>d$.

\begin{definition}[Representative cocycles]
Let $\bsigma^p:=\{\sigma_I\}_{I\in\caB_p(\Xfunc)}$ be a $\caB_p(\Xfunc)$-indexed collection of $p$-cocycles in $\Xfunc$. 
The collection $\bsigma^p$ is called \emph{a family of representative $p$-cocycles for $\Hfunc^p(\Xfunc)$}, if for any $t\in\R$, the set $\{[\sigma_I]_t\}_{t\in I\in \caB_p(\Xfunc)}$ forms a linear basis for $\Hfunc^p(\X_t)$. In this case, each $\sigma_I$ is called a \emph{representative cocycle associated to the interval $I$}.
The disjoint union $\bsigma:=\sqcup_{p\in\N}\bsigma^p$ is called \emph{a family of representative cocycles for $\Hfunc^*(\Xfunc)$}.
\end{definition}

The existence of a family of representative cocycles for $\Hfunc^*(\Xfunc)$ (assuming that the filtration $\Xfunc$ has finite critical values and finite-dimensional cohomology point-wise) is guaranteed by the interval decomposition theorem of point-wise finite dimensional persistence modules (see \cite{crawley2015decomposition}) and the axiom of choice.
Software programs are available to compute the total barcode and return a family of representative cocycles, such as Ripser (see \cite{bauer2021ripser}), Java-Plex (see \cite{adams2014javaplex}), Dionysus (see \cite{de2011persistent}), and Gudhi (see \cite{maria2014gudhi}).
These cocycles are naturally equipped with the cup product operation, which we recall as follows.

\subparagraph*{Cup product.} We recall the cup product operation in the setting of simplicial cohomology. 
Let $\X$ be a simplicial complex with an ordered vertex set $\{x_1<\dots<x_n\}$. 
For any non-negative integer $p$, we denote a $p$-simplex by $\alpha:=[\alpha_0,\dots,\alpha_{p}]$ where $\alpha_0<\dots<\alpha_{p}$ are ordered vertices in $\X$, and by $\alpha^*:C_p( \X)\to K$ the dual of $\alpha$, where $\alpha^*(\alpha)=1$ and $\alpha^*(\tau)=0$ for any $p$-simplex $\tau\neq \alpha$. Here $K$ is the base field as before, and $\alpha^*$ is also called a $p$-cosimplex.
Let $\beta:=[\beta_0,\dots,\beta_{q}]$ be a $q$-simplex for some integer $q\geq 0$. The \emph{cup product} $\alpha^*\smile \beta^*$ is defined as the linear map $C_{p+q}( \X)\to  K$ such that for any $ (p+q)$-simplex $\tau=[\tau_{0},\dots,\tau_{p+q}]$, 
\[\alpha^{*}\smile\beta^{*}(\tau):=\alpha^*([\tau_{0},\dots,\tau_{p}])\cdot \beta^*([\tau_{p},\dots,\tau_{p+q}]).\] 
Equivalently, 
we have that 
$\alpha^{*}\smile\beta^{*}$ is $[\alpha_0,\dots,\alpha_{p},\beta_1,\dots,\beta_{q}]^*$
if $\alpha_{p}=\beta_0$, and $0$ otherwise.
By a \emph{$p$-cochain} we mean a finite linear sum $\sigma=\sum_{j=1}^h \lambda_j \alpha^{j*}$, where each $\alpha^j$ is a $p$-simplex in $\X$ and $\lambda_j\in K$. The \emph{cup product} of a $p$-cochain $\sigma=\sum_{j=1}^h \lambda_j \alpha^{j*}$ and a $q$-cochain $\sigma'=\sum_{j'=1}^{h'}\mu_{j'}\beta^{j'*}$ is defined as $\sigma\smile\sigma':=\sum_{j,j'}\lambda_j\mu_{j'}\left(\alpha^{j*}\smile \beta^{j'*}\right).$ 

In our algorithms, $K$ is taken to be $\Z_2$ and every $p$-simplex $\alpha=[x_{i_0},\dots,x_{i_{p}}]$ is represented by the ordered  list 
$[i_0,\dots,i_{p}]$. We assume a total order (e.g. the order given in \cite{bauer2021ripser}) on the simplices in $\X$.
Since coefficients are either $0$ or $1$, a $p$-cochain can be written as $\sigma=\sum_{j=1}^h \alpha^{j*}$ for some $\alpha^j=[x_{i_0^j},\dots,x_{i_{p}^j}]$ and will be represented by the list $\left[ [i_0^1,\dots,i_{p}^1],\dots,[i_0^h,\dots,i_{p}^h] \right].$ We call $h$ the size of $\sigma$. Let $\X_p\subset \X$ be the set of $p$-simplices. Alg.~\ref{alg:cup_product} computes the cup product of two cochains over $\Z_2$.

\begin{algorithm}[H] 
\SetKwData{Left}{left}\SetKwData{This}{this}\SetKwData{Up}{up}
\SetKwInOut{Input}{Input}\SetKwInOut{Output}{Output}
\Input{Two cochains $\sigma_1$ and $\sigma_2$, and the simplicial complex $\X$.}
\Output{The cup product $\sigma=\sigma_1\smile \sigma_2$, at cochain level.}
\BlankLine
$\sigma\gets [\,]$\;
\If{$\dim(\sigma_1)+\dim(\sigma_2)\leq \dim(\X)$}
    {\For
    {$i \leq \mathrm{size}(\sigma_1)$ and $j \leq \mathrm{size}(\sigma_2)$ }{
            $a\gets \sigma_1(i)$ and $b\gets \sigma_2(j)$\;
            \If{$a[\mathrm{end}]== b[\mathrm{first}]$}{
                $c \gets a.\mathrm{append}(b[\mathrm{second}:\mathrm{end}])$\;
                \If
                {$c\in \X_{\dim(\sigma_1)+\dim(\sigma_2)}$ \label{line:last-if-cup}}{
                     Append $c$ to $\sigma$\;
                     }
                }
            }
    }
\Return $\sigma$.
\caption{$\mathrm{CupProduct}(\sigma_1,\sigma_2,\X)$ }
\label{alg:cup_product}
\end{algorithm}
 
\medskip

\begin{remark}[Complexity of Alg.~\ref{alg:cup_product}] \label{rmk:complexity_cup_prod}
Let $c$ be the complexity of checking whether a simplex is in the simplicial complex, and let $m:= \card(\X)$ be the number of simplices. For $\Z_2$-coefficients, cocycles are in one-to-one correspondence with the subsets of $\X$, so the size of a cocycle is at most $m$. 
Thus, the complexity for Alg.~\ref{alg:cup_product} is $O(\mathrm{size}(\sigma_1)\cdot\mathrm{size}(\sigma_2)\cdot c)\leq O( m^2\cdot c)$.
\end{remark}

\subparagraph*{Cohomology ring and cup-length} For a given space $\X$, the cup product yields a bilinear map $\smile:\Hfunc^p(\X)\times\Hfunc^q(\X)\to\Hfunc^{p+q}(\X)$ of vector spaces.
In particular, it turns the total cohomology vector space $\Hfunc^*(\X):=\bigoplus_{p\in\N}\Hfunc^p(\X)$ into a graded ring $(\Hfunc^*(\X),+,\smile)$ (see 
\textsection \ref{sec:supp_cup_l_func} 
for the explicit definition of a graded ring).
The \textit{cohomology ring map} $\X\mapsto \Hfunc^*(\X)$ defines a contravariant functor from the category of spaces, $\Top$, to the category of graded rings, $\Ring$
(see \cite[\textsection 3.2]{hatcher2000}). 
To avoid the difficulty of describing and comparing ring structures in a computer, we study a computable invariant of the graded cohomology ring, called the \textit{cup-length}. See
\textsection \ref{sec:eip-mono-inv} 
for the general notion of \emph{invariants}. 
For a category $\CC$, denote by $\ob(\CC)$ the set of objects in $\CC$.

\begin{restatable}[Length and cup-length]{definition}{deflength}
\label{def:length}
The \textbf{length} of a graded ring $R=\bigoplus_{p\in\N} R_p$ is the largest non-negative integer $\ell$ such that there exist positive-dimension homogeneous elements $\eta_1,\dots,\eta_\ell\in R$ (i.e. $\eta_1,\dots,\eta_\ell\in \bigcup_{p\geq 1} R_p$) with $\eta_1 \bullet \dots \bullet \eta_\ell \neq 0$. 
If $\bigcup_{p\geq 1} R_p=\emptyset$, then we define the length of $R$ to be zero.
We denote the length of a graded ring $R$ by $\len(R)$, and call the following map the \textbf{length invariant}:
\[\len:\ob(\Ring)\to\N,\text{ with }R\mapsto \len(R).\]

When $R=(\Hfunc^*(\X),+,\smile)$ for some space $\X$,
we denote $\cupprod(\X):=\len(\Hfunc^*(\X))$ and call it the \textbf{cup-length of $\X$}.
And we call the following map the \textbf{cup-length invariant}:
\[\cupprod:\ob(\Top)\to\N, \text{ with }X\mapsto \cupprod(\X).\]
\end{restatable}

\begin{remark} [About the strength of the cup-length invariant]
In some cases, cup-length captures more information than homology. One well-known example is given by the torus $\T^2$ v.s. the wedge sum $\bbS^1 \vee \bbS^2 \vee \bbS^1$, where despite having the same homology groups, these two spaces have \emph{different} cup-length. By specifying suitable metrics and considering the Vietoris-Rips filtrations of the two spaces, the strength of cup-length persists in the setting of persistence 
(see Ex.~\ref{ex:T2-wedge}).
It is also worth noticing that cup-length is not a complete invariant for graded cohomology rings. For instance, after taking the wedge sum of $\T^2$ and $\bbS^1 \vee \bbS^2 \vee \bbS^1$ with $\T^2$ respectively, the resulted spaces $\T^2 \vee \T^2$ and $\bbS^1 \vee \bbS^2 \vee \bbS^1 \vee \T^2$ still have different ring structures, but they have the same cup-length (since cup length takes the `maximum').
\end{remark}

An important fact about the cup-length is that it can be computed using a linear basis for the cohomology vector space. In 
Prop.~\ref{prop:len_of_basis} 
we show that if $B_p$ is a linear basis for $\Hfunc^p(\X)$ for each $p\geq 1$ and $B:=\bigcup_{p\geq 1} B_p$, then $\cupprod(\X)=\sup\left\{\ell\geq 1\mid B^{\smile \ell}\neq \{0\}\right\}$. 

\subsection{Persistent cohomology rings and persistent cup-length functions}\label{sec:cup_l_func}
We study the persistent cohomology ring of a filtration and 
the associated notion of persistent cup-length invariant. 
We examine several examples of this persistent invariant and establish a way to visualize it in the half-plane above the diagonal. See 
\textsection \ref{supp:sec_cup_l_func} 
for the proofs of our results in this section. 

Filtrations of spaces are special cases of \textit{persistent spaces}. In general, for any category $\CC$, one can define the notion of a \emph{persistent object} in $\CC$, as a functor from the poset  $(\R,\leq)$ (viewed as a category) to the category $\CC$. 
For instance, a functor $\Rfunc:(\R,\leq)\to\Ring$ is called a \textbf{persistent graded ring}. 
Recall the contravariant cohomology ring functor $\Hfunc^*:\Top\to\Ring$. 
Given a persistent space $\Xfunc:(\R,\leq)\to\Top$, 
the composition $\Hfunc^*(\Xfunc):(\R,\leq)\to\Ring$ is called the \textbf{persistent cohomology ring of $\Xfunc$}. Due to the contravariance of $\Hfunc^*$, we consider only contravariant persistent graded rings in this paper.

\begin{definition}\label{def:cup_l_func}
We define the \textbf{persistent cup-length function} of a persistent space $\Xfunc$ as the function $\cupprod(\Xfunc):\Int\to\N$ given by $\langle t,s \rangle\mapsto\len\left(\mathbf{Im}(\Hfunc^*(\Xfunc)(\langle t,s\rangle)\right).$
\end{definition}

\begin{remark}[Notation for image ring]
\label{rmk:image-ring} 
 $\mathbf{Im}(\Hfunc^*(\Xfunc)(\langle t,s\rangle)$ is defined as
the image ring $\mathbf{Im}(\Hfunc^*(\Xfunc)([t-\delta,s+\delta]))=\mathbf{Im}(\Hfunc^*(\X_{s+\delta})\to\Hfunc^*(\X_{t-\delta}))$ for sufficiently small $\delta>0$, when $\langle t,s\rangle=(t,s)$, 
and is defined similarly for the cases when $\langle t,s\rangle=(t,s]$ or $[t,s).$

\end{remark}

\begin{remark}
It follows from 
Prop.~\ref{prop:injective-surjective-length} 
that the cup-length invariant $\cupprod$ is non-increasing under surjective morphisms and non-decreasing under injective morphisms, which we call an \emph{inj-surj invariant}. 
As a consequence 
(see Prop.~\ref{thm:Functoriality of persistent invariants}), 
for any persistent space $\Xfunc$, the persistent cup-length function $\cupprod(\Xfunc)$ defines a \emph{functor} from $(\Int,\leq)$ to $(\N,\geq)$.
\end{remark}

Prop.~\ref{prop:generators-of-ring} below allows us to compute the cohomology images of a persistent cohomology ring from representative cocycles, which will be applied to compute persistent cup-length functions in Ex.~\ref{ex:cup_l_func_klein} and prove Thm.~\ref{thm:tropical_mobius},
see page~\pageref{pf:tropical_mobius}. 
Prop.~\ref{prop:prod-coprod-pers-cup} allows us to simplify the calculation of persistent cup-length functions in certain cases, such as the Vietoris-Rips filtration of products or wedge sums of metric spaces, e.g. Ex.~\ref{ex:T2-wedge}. 

\begin{restatable}[Persistent image ring]{proposition}{propimagering}
\label{prop:generators-of-ring}
Let $\Xfunc=\{\X_t\}_{t\in\R}$ be a filtration, together with a family of representative cocycles $\bsigma=\{\sigma_I\}_{I\in \caB(\Xfunc)}$ for $\Hfunc^*(\Xfunc)$. Let $t\leq s$ in $\R$. Then
$\mathbf{Im}(\Hfunc^*(\X_s)\to \Hfunc^*(\X_t))=\langle [\sigma_I]_t: [t,s]\subset I\in\caB(\Xfunc) \rangle,$
generated as a graded ring. 
\end{restatable}

\begin{restatable}{proposition}{propproperty}
\label{prop:prod-coprod-pers-cup}
Let $\Xfunc,\Yfunc:(\R,\leq)\to\Top$ be two persistent spaces. Then:
\begin{itemize}
    \item $\cupprod\left(\Xfunc\times\Yfunc\right)=\cupprod(\Xfunc)+\cupprod(\Yfunc),\text{  }$
    \item $\cupprod\left(\Xfunc\amalg\Yfunc\right)=\max\{\cupprod(\Xfunc),\cupprod(\Yfunc)\},\text{ and }$
    \item $\cupprod\left(\Xfunc\vee\Yfunc\right)=\max\{\cupprod(\Xfunc),\cupprod(\Yfunc)\}$.
\end{itemize}
Here $\times,\amalg$ and $\vee$ denote point-wise product, disjoint union, and wedge sum, respectively.
\end{restatable}


\subparagraph*{Examples and visualization}

Each interval $\langle a,b\rangle$ in $\Int$ is visualized as a point $(a,b)$ in the half-plane above the diagonal (see Fig.~\ref{figure:interval-triangle}).
To visualize the persistent cup-length function of a filtration $\Xfunc$, we assign to each point $(a,b)$ the integer value $\cupprod(\Xfunc)(\langle a,b\rangle)$, if it is positive. If $\cupprod(\Xfunc)(\langle a,b\rangle)=0$ we do not assign any value. We present an example to demonstrate how persistent cup-length functions are visualized in the upper-diagonal plane (see Fig.~\ref{fig:klein-function-diagram}). 

\begin{figure}[H]
    \begin{tikzpicture}[scale=0.5]
    \begin{axis} [ 
    axis y line=middle, 
    axis x line=middle,
    ticks=none,
    xmin=-.5, xmax=3.5,
    ymin=-.5, ymax=3.5,]
    \addplot [mark=none] coordinates {(-.5,-.5) (3.5,3.5)};
    \addplot [mark=none,dashed] coordinates {(1,0) (1,1)};
    \addplot [mark=none,dashed] coordinates {(3,0) (3,3)};
    \addplot [thick,color=red!60!white,fill=red!50!white, 
                    fill opacity=0.45]coordinates {
            (1,3)
            (1,1)
            (3,3)  
            (1,3)};
    \node[mark=none] at (axis cs:1,-0.3){\huge $a$};
    \node[mark=none] at (axis cs:3,-0.3){\huge $b$};
    \node[mark=none] at (axis cs:1,3.3){\huge $(a,b)$};
    \end{axis}
    \end{tikzpicture}
 \caption{The interval $\langle a,b\rangle$ in $\Int$ corresponds to the point $(a,b)$ in $\R^2$. }
 \label{figure:interval-triangle}
 \end{figure}
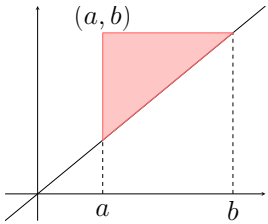

\begin{example}[Visualization of $\cupprod(\cdot)$]
\label{ex:cup_l_func_klein} 
Recall the filtration $\mathbf{X}=\{\mathbb{X}_t\}_{t\geq0}$ of a Klein bottle with a 2-cell attached, defined in Fig.~\ref{fig:klein-function-diagram}. Consider the persistent cohomology $\Hfunc^*(\Xfunc)$ in $\Z_2$-coefficients. 
Let $v$ be the $0$-cocycle born at $t=0$, let $\alpha$ be the $1$-cocycle born at $t=1$ and died at $t=3$, and let $\beta$ be the $1$-cocycle 
born at time $t=2$. 
Let $\gamma:=\beta\smile\beta$, which is then a non-trivial $2$-cocycle born at time $t=2$, like $\beta$.
Then the barcodes of $\Xfunc$ are: $\caB_0(\Xfunc)=\{[0,\infty)\}$, $\caB_1(\Xfunc)=\{[1,3),[2,\infty)\}$, and $\caB_2(\Xfunc)=\{[2,\infty)\}$. See Fig.~\ref{fig:kleinfilt}. 

\begin{figure}[H]
\begin{tikzcd}[column sep=tiny,row sep=0pt]
	\begin{tikzpicture}[scale=0.25]
    \node[fill=red,circle,inner sep=.4mm] at (1.4,5)  {};
\end{tikzpicture}
	& 
	\begin{tikzpicture}[scale=0.25]
	   \filldraw[color=red!40, fill=none, thick](2.5,5) ellipse (1.2 and 0.5);
    \node[fill=red,circle,inner sep=.4mm] at (1.4,5)  {};
\end{tikzpicture}
	& 
\begin{tikzpicture}[scale=0.18]
	\filldraw[color=red!40, fill=none, thick](2.7,9) ellipse (1 and 0.5);
    \node[fill=red,circle,inner sep=.4mm] at (1.8,9)  {};
    \Coordinate{innerr}{4.5,6} 
    \Coordinate{innerl}{2.5,5}
    \Coordinate{intersectionr}{4.2,7} 
    \Coordinate{intersectionl}{2.1,6} 
    \Coordinate{bottoml}{2.2,3.3} 
    \Coordinate{bottomr}{7,4} 
    \Coordinate{left}{2,2} 
    \Coordinate{right}{9.9,4.2} 
        \draw (left) to[out=150,in=210,looseness=1.3] (intersectionl); 
        \draw[dashed] (intersectionl) to[out=30,in=200,looseness=0.6] (intersectionr);
    
        \draw (innerr) to[out=120,in=250,looseness=1] (4,10) to[out=70,in=180,looseness=1] (5.2,11) to[out=360,in=90,looseness=1] (7.5,9) to[out=270,in=20,looseness=1] (intersectionr);
        
        \draw  (innerl) to[out=110,in=80,looseness=4] (right);
        \draw (innerl) to[out=315,in=315,looseness=0.5] (innerr); 
        \draw[dashed] (innerl) to[out=135,in=135,looseness=0.5] node[below right,pos=0.5,black] {} (innerr); 
        \draw (left) to[out=330,in=260,looseness=1.1] (right);
        \draw[dashed] (innerl) to[out=300,in=60,looseness=1] (bottoml);
        \draw[dashed] (innerr) to[out=300,in=150,looseness=1] (bottomr);
        \draw[dashed] (1.6,3) to [out=330,in=240,looseness=.5] 
        (bottoml) to [out=30,in=160,looseness=1.4] 
        (bottomr) to [out=340,in=190,looseness=.5]  (7.7,3.8) ;
\end{tikzpicture}
& \begin{tikzpicture}[scale=0.18]
	\filldraw[color=red!40, fill=red!10, thick](2.7,9) ellipse (1 and 0.5);
    \node[fill=red,circle,inner sep=.4mm] at (1.8,9)  {};
    \Coordinate{innerr}{4.5,6} 
    \Coordinate{innerl}{2.5,5}
    \Coordinate{intersectionr}{4.2,7} 
    \Coordinate{intersectionl}{2.1,6} 
    \Coordinate{bottoml}{2.2,3.3} 
    \Coordinate{bottomr}{7,4} 
    \Coordinate{left}{2,2} 
    \Coordinate{right}{9.9,4.2} 
        \draw (left) to[out=150,in=210,looseness=1.3] (intersectionl); 
        \draw[dashed] (intersectionl) to[out=30,in=200,looseness=0.6] (intersectionr);
    
        \draw (innerr) to[out=120,in=250,looseness=1] (4,10) to[out=70,in=180,looseness=1] (5.2,11) to[out=360,in=90,looseness=1] (7.5,9) to[out=270,in=20,looseness=1] (intersectionr);
        
        \draw  (innerl) to[out=110,in=80,looseness=4] (right);
        
        \draw (innerl) to[out=315,in=315,looseness=0.5] (innerr); 
        \draw[dashed] (innerl) to[out=135,in=135,looseness=0.5] node[below right,pos=0.5,black] {} (innerr); 
        \draw (left) to[out=330,in=260,looseness=1.1] (right);
        \draw[dashed] (innerl) to[out=300,in=60,looseness=1] (bottoml);
        \draw[dashed] (innerr) to[out=300,in=150,looseness=1] (bottomr);
        \draw[dashed] (1.6,3) to [out=330,in=240,looseness=.5] 
        (bottoml) to [out=30,in=160,looseness=1.4] 
        (bottomr) to [out=340,in=190,looseness=.5]  (7.7,3.8) ;
\end{tikzpicture}
&
\begin{tikzpicture}[scale=.7] 
    \begin{axis} [ 
    height=4.5cm,
    width=7cm,
    hide y axis,
    axis x line*=bottom,
    xtick={1,2,3,4,5},
    xticklabels={$0$, $1$, $2$, $3$, $4$},
    xmin=.5, xmax=6.5,
    ymin=0, ymax=1.6,]
    \addplot [mark=none,thick] coordinates {(1.05,.2) (5.95,.2)};
    \addplot [mark=none,thick] coordinates {(2.05,.4) (3.95,.4)};
    \addplot [mark=none,thick] coordinates {(3.05,.6) (5.95,.6)};
    \addplot [mark=none,thick] coordinates {(3.05,.8) (5.95,.8)};
    \addplot [mark=*] coordinates {(1,.2)}  node[left] {$v$};
    \addplot [mark=*] coordinates {(3,.6)}  node[left] {$\beta$};
    \addplot [mark=*] coordinates {(2,.4)}  node[left] {$\alpha$};
    \addplot [mark=*] coordinates {(3,.8)}  node[left] {$\gamma$};
    \addplot [mark=o] coordinates {(4,.4)};
    \node[mark=none] at (axis cs:3.5,1.4){\large$\caB(\Xfunc)$};
    \end{axis}
    \end{tikzpicture}
\\
\text{\small{$t\in [0,1)$}} 
& \text{\small{$t\in [1,2)$}} 
& \text{\small{$t\in [2,3)$}} 
& \text{\small{$t\geq 3$}}
&
	\end{tikzcd}
\caption{The filtration $\Xfunc$ given in Fig.~\ref{fig:klein-function-diagram} and its barcode $\caB(\Xfunc)$, see Ex.~\ref{ex:cup_l_func_klein}.}
\label{fig:kleinfilt}
\end{figure}

Using the formula in Prop.~\ref{prop:generators-of-ring}, for any $t\leq s$, we have 
\begin{align*}
     \mathbf{Im}(\Hfunc^*(\mathbb{X}_s)\to\Hfunc^*(\mathbb{X}_t)) =\,&  
     \begin{cases}
    \langle[v]_t,[\beta]_t,[\gamma]_t\rangle , &\mbox{if $2\leq t< 3$ and $s\geq 3$}\\
    \langle[v]_t,[\alpha]_t,[\beta]_t,[\gamma]_t\rangle , &\mbox{if $2\leq t\leq s<3$}\\
    \langle[v]_t,[\alpha]_t\rangle , &\mbox{if $1\leq t<2$ and $s< 3$}\\
    \langle [v]_t\rangle , &\mbox{otherwise.} 
     \end{cases}
\end{align*}

The persistent cup-length function of $\mathbf{X}$ is computed as follows and visualized in Fig.~\ref{fig:klein-function-diagram}.
\[
  \pcupprod(\mathbf{X})([t,s]) = 
      \begin{cases}
       2 , & \text{if }t\geq 2 \\
       1 , & \text {if } 1\leq t<2 \text{ and } s<3\\
       0 , & otherwise.
      \end{cases}
\] 
\end{example}

We end this section by presenting an example, Ex.~\ref{example:samePH}, where the persistent cup-length function distinguishes a pair of filtrations which the total barcode is not able to. A similar example is available in 
\textsection \ref{sec:erosion_T2_wedge}, 
where we will also give a quantitative measure
via the erosion distance on the difference between persistent cup-length functions of different filtrations. 

\begin{example}[$\cupprod(\cdot)$ better than standard barcode]
\label{example:samePH}
Consider the filtration $\Xfunc=\{\X_t\}_{t\geq0}$ of a $2$-torus $\T^2$ and the filtration $\Yfunc=\{\Y_t\}_{t\geq0}$ of the space $\mathbb{S}^1\vee \mathbb{S}^2\vee \mathbb{S}^1$ as shown in Fig.~\ref{fig:doublefilt}.
Knowing that $\X_3=\T^2$ and $\Y_3=\mathbb{S}^1\vee \mathbb{S}^2\vee \mathbb{S}^1$ have the same (co)homology vector spaces in all dimensions, one can directly check that the persistent (co)homology vector spaces associated to $\Xfunc$ and $\Yfunc$ are the same. 
However, the cohomology ring structure of $\X_3$ is different from that of $\Y_3$: there are two $1$-cocycles in $\X_3$ with a non-zero product (indeed the product is a $2$-cocycle), whereas all $1$-cocycles in $\Y_3$ have zero product.
This difference between the cohomology ring structures of these two filtration is quantified by their persistent cup-length functions, see Fig.~\ref{fig:doublefilt}. Also, see
Ex.~\ref{ex:T2-wedge}
for a more geometric example, which considers the Vietoris-Rips filtrations of $\T^2$ and $\bbS^1\vee\bbS^2\vee\bbS^1$. 

\begin{figure}[H]
    \begin{tikzcd}[column sep=tiny,row sep=tiny]
    \begin{tikzpicture}[scale=0.45]
    \node[fill=red,circle,inner sep=.4mm] at (3,1.5)  {};
    \end{tikzpicture}
	& 
	\begin{tikzpicture}[scale=0.45]
	   \filldraw[color=red!40, fill=none, thick](3,1.75) ellipse (0.35 and .66);
    \node[fill=red,circle,inner sep=.4mm] at (3,1)  {};
    \end{tikzpicture}
	& 
	\begin{tikzpicture}[scale=0.35]
	   \filldraw[color=red!40, fill=none, thick](3,.95) ellipse (0.35 and .66);
    \node[fill=red,circle,inner sep=.4mm] at (3,0.25)  {};
    \Coordinate{left}{0,2}
    \Coordinate{right}{6,2}
    \Coordinate{leftinner}{1.5,2}
    \Coordinate{rightinner}{4.5,2}
        \draw (left) to[out=90,in=90,looseness=1] 
        (right) to [out=270,in=270,looseness=1] (left);
        \draw (leftinner) to[out=30,in=150,looseness=1] 
        (rightinner); 
        \draw (1.2,2.15) to[out=330,in=210,looseness=1] 
        (4.8,2.15); 
    \end{tikzpicture}
	\\
    \text{\small{$t\in[0,1)$}} 
    & \text{\small{$t\in [1,2)$}} 
    & \text{\small{$t\geq 2$}} 
    \\
    \begin{tikzpicture}[scale=0.45]
    \node[fill=red,circle,inner sep=.4mm] at (-3.5,0)  {};
    \end{tikzpicture}
	& 
	\begin{tikzpicture}[scale=0.35]
	   \filldraw[color=red!40, fill=none, thick](-2.5,0) circle (1);
    \node[fill=red,circle,inner sep=.4mm] at (-3.5,0)  {};
    \end{tikzpicture}
	& 
	\begin{tikzpicture}[scale=0.35]
	   \filldraw[color=red!40, fill=none, thick](-2.5,0) circle (1);
    \node[fill=red,circle,inner sep=.4mm] at (-3.5,0)  {};
    \Coordinate{left}{0,2}
    \Coordinate{right}{6,2}
    \Coordinate{leftinner}{1.5,2}
    \Coordinate{rightinner}{4.5,2}
    \filldraw[fill=none](0,0) circle (1.5);
    \filldraw[fill=none](2.5,0) circle (1);
    \draw [dashed](-1.5,0) to[out=90,in=90,looseness=.5] (1.5,0);
    \draw (1.5,0) to [out=270,in=270,looseness=.5] (-1.5,0);
    \end{tikzpicture}
    \end{tikzcd}
    \begin{tikzcd}[column sep=tiny,row sep=tiny]
    \begin{tikzpicture}[scale=0.45]
    \begin{axis} [ 
    axis y line=middle, 
    axis x line=middle,
    ytick={1,2,3,4,5.3},
    yticklabels={$1$,$2$,$3$,$4$,$\infty$},
    xticklabels={$1$,$2$,$3$,$4$,$\infty$},
    xtick={1,2,3,4,5.3},
    xmin=0, xmax=5.5,
    ymin=0, ymax=5.5,]
    \addplot [mark=none] coordinates {(0,0) (5.3,5.3)};
    \addplot [thick,color=dgmcolor!20!white,fill=dgmcolor!20!white, 
                    fill opacity=0.45]coordinates {
            (1,5.3) 
            (1,1)
            (2,2)  
            (2,5.3)};
    \addplot [thick,color=dgmcolor!40!white,fill=dgmcolor!40!white, 
                    fill opacity=0.45]coordinates {
            (2,5.3) 
            (2,2)  
            (5.3,5.3)};
    \node[mark=none] at (axis cs:1.5,3.2){\huge\textsf{1}};
    \node[mark=none] at (axis cs:3.5,4){\huge\textsf{2}};
    \node[mark=none] at (axis cs:4,2){\huge$\cupprod(\Xfunc)$};
    \end{axis}
    \end{tikzpicture} 
    \\
    \begin{tikzpicture}[scale=0.45]
    \begin{axis} [ 
    axis y line=middle, 
    axis x line=middle,
    ytick={1,2,3,4,5.3},
    yticklabels={$1$,$2$,$3$,$4$,$\infty$},
    xticklabels={$1$,$2$,$3$,$4$,$\infty$},
    xtick={1,2,3,4,5.3},
    xmin=0, xmax=5.5,
    ymin=0, ymax=5.5,]
    \addplot [mark=none] coordinates {(0,0) (5.3,5.3)};
    \addplot [thick,color=dgmcolor!20!white,fill=dgmcolor!20!white, 
                    fill opacity=0.45]coordinates {
            (1,5.3) 
            (1,1)  
            (5.3,5.3)};
    \node[mark=none] at (axis cs:2.5,3.5){\huge\textsf{1}};
    \node[mark=none] at (axis cs:4,2){\huge$\cupprod(\Yfunc)$};
    \end{axis}
    \end{tikzpicture}
	\end{tikzcd}
\caption{Top: A filtration $\Xfunc$ of  $\mathbb{T}^2$ and its persistent cup-length function $\cupprod(\Xfunc)$. Bottom: A filtration $\Yfunc$ of the wedge sum $\mathbb{S}^1\vee \mathbb{S}^2\vee \mathbb{S}^1$ and its persistent cup-length function $\cupprod(\Yfunc)$. See Ex.~\ref{example:samePH}.} 
\label{fig:doublefilt}
\end{figure}
\end{example}

\section{The persistent cup-length diagram}
\label{sec:comp_cup_l}
In this section, we introduce the notion of the 
\textit{persistent cup-length diagram of a filtration}, by using the cup product operation on  cocycles. In \textsection\ref{sec:cup_dgm}, we show how the persistent cup-length diagram is used to compute the persistent cup-length function (see Thm.~\ref{thm:tropical_mobius}). In \textsection \ref{sec:algorithms}, we develop an algorithm (see Alg.~\ref{alg:main}) to compute the persistent cup-length diagram, and study its complexity.
Proofs, details and extra examples are available in 
\textsection \ref{sec:details_persistent cup-length diagram}.

\subsection{Persistent cup-length diagram} 
\label{sec:cup_dgm}

We define the persistent cup-length diagram using a family of representative cocycles. 
In Thm.~\ref{thm:tropical_mobius} we show that the persistent cup-length function can be retrieved from the persistent cup-length diagram. 
See 
\textsection \ref{sec:proof_cup_dgm} 
for the proofs of Thm.~\ref{thm:tropical_mobius} and details for this section.

\begin{definition} [$\ell$-fold $*_{\bsigma}$-product]
\label{dfn:ell-product}
Let $\bsigma$ be a family of representative cocycles for $\Hfunc^*(\Xfunc)$. Let $\ell\in\N^*$ and let $I_1,\dots,I_\ell$ be a sequence of elements in $\caB(\Xfunc)$ with representative cocycles $\sigma_{I_1},\dots,\sigma_{I_\ell}\in\bsigma$, respectively. 
We define the \textbf{$\ell$-fold $*_{\bsigma}$-product of $I_1,\cdots,I_\ell$} to be  
\begin{equation}\label{eq:support}
I_1 *_{\bsigma}\dots*_{\bsigma} I_\ell := \{t\in \R\mid [ \sigma_{I_1}]_t \smile\dots\smile [\sigma_{I_\ell}]_t \neq [0]_t\},
\end{equation}
associated with the formal representative cocycle $ \sigma_{I_1} \smile\dots\smile \sigma_{I_\ell}$.
We also call the right-hand side of Eq.~(\ref{eq:support}) the \textbf{support} of $\sigma_{I_1} \smile\dots\smile \sigma_{I_\ell}$, and denote it by $\operatorname{supp}(\sigma_{I_1} \smile\dots\smile \sigma_{I_\ell})$.
\end{definition}

The support of a product of representative cocycles is always an interval: 
\begin{restatable}[Support is an interval]
{proposition}{endsofsupport}\label{prop:support}
With the same assumption and notation in Defn.~\ref{dfn:ell-product}, let $I:=\operatorname{supp}(\sigma_{I_1} \smile\dots\smile \sigma_{I_\ell})$. If $I\neq \emptyset$, then $I$ is an interval $\langle b,d\rangle$, where $b\leq d$ are such that $d$ is the right end of $\cap_{1\leq i\leq \ell}I_i$ and $b$ is the left end of some $J\in \caB(\Xfunc)$ ($J$ is not necessarily one of the $I_i$).
\end{restatable}

The $*_{\bsigma}$-product is associative and invariant 
under permutations. 
Let $\caB_{\geq 1}(\Xfunc):=\sqcup_{p\geq 1} \caB_p(\Xfunc)$.
Let $\caB_{\geq 1}(\Xfunc)^{*_{\bsigma}\ell}$ be the set of $I_1*_{\bsigma}\dots *_{\bsigma}I_\ell$ where each $I_i\in \caB_{\geq 1}(\Xfunc)$. For the simplicity of notation, we often write $\caB_{\geq 1}(\Xfunc)^{*_{\bsigma}\ell}$ as $\caB(\Xfunc)^{*_{\bsigma}\ell}$.


\begin{definition}[persistent cup-length diagram]
\label{def:cup_dgm}
Let $\Xfunc$ be a filtration and let $\caB_{\geq1}(\Xfunc)$ be its barcode over positive dimensions.
Let $\bsigma=\{\sigma_I\}_{I\in\caB_{\geq1}(\Xfunc)}$ be a family of representative cocycles for $\Hfunc^{\geq1}(\Xfunc)$.
The \textbf{persistent cup-length diagram of $\Xfunc$ (associated to $\bsigma$)} is defined to be the map $\mathbf{dgm}_{\bsigma}^{\smile}(\Xfunc):\Int\to\N$, given by:
\[\mathbf{dgm}_{\bsigma}^{\smile}(\Xfunc)(I):= 
    \max\{\ell\in\N^*\mid I=I_1*_{\bsigma}\dots *_{\bsigma}I_\ell\text{, where each }I_i\in\caB_{\geq1}(\Xfunc)\},\]
with the convention that $\max\emptyset=0.$
\end{definition}

Recall Thm.~\ref{thm:tropical_mobius}, which states the relation between the persistent cup-length function $\cupprod(X)$ and the persistent cup-length diagram $\mathbf{dgm}_{\bsigma}^{\smile}(\Xfunc)$: for any interval $\langle a,b\rangle$, the $\cupprod(\Xfunc)(\langle a,b\rangle)$ attains the \textbf{maximum} value of $\mathbf{dgm}_{\bsigma}^{\smile}(\Xfunc)(\langle c,d\rangle)$ over all intervals $\langle c,d\rangle\supseteq\langle a,b\rangle$.
This is in the same spirit as in \cite{patel2018generalized} where the rank function can be reconstructed from  the persistence diagram by replacing `max' operation with the \textbf{sum} operation.

\begin{example} [Example of $\mathbf{dgm}_{\bsigma}^{\smile}(\cdot)$ and Thm.~\ref{thm:tropical_mobius}] 
\label{ex:example_dgm}
Recall the filtration $\mathbf{X}=\{\mathbb{X}_t\}_{t\geq0}$ of the pinched Klein bottle defined in Fig.~\ref{fig:klein-function-diagram}, and its persistent cup-length function and the representative cocycles $\{\alpha,\beta,\gamma\}=:\bsigma$ from Ex.~\ref{ex:cup_l_func_klein}. 
Because $\Hfunc^*(\Xfunc)$ is non-trivial up to dimension $2$, $\mathbf{dgm}_{\bsigma}^{\smile}(\Xfunc)(I)\leq 2$ for any $I$. It follows from $[\alpha\smile\alpha]=0$ that $\mathbf{dgm}_{\bsigma}^{\smile}(\Xfunc)([1,3))=1$, and from $[\alpha\smile\beta]=[\gamma]$ that $[2,3)=[1,3)*_{\bsigma}[2,\infty)$, implying $\mathbf{dgm}_{\bsigma}^{\smile}(\Xfunc)([2,3))=2$. A similar argument holds for $[2,\infty)$, using the fact that $[\beta\smile\beta]=[\gamma]$. 
Thus, we obtain the persistent cup-length diagram $\mathbf{dgm}_{\bsigma}^{\smile}(\mathbf{X})$ as below (see the right-most figure in Fig.~\ref{fig:klein-function-diagram} for its visualization):
\[
\mathbf{dgm}_{\bsigma}^{\smile}(\Xfunc)(I)= 
      \begin{cases}
    1, & \mbox{ if } I=[1,3) \\
    2, & \mbox{ if } I=[2,3) \text { or } I=[2,\infty)\\
    0, & \mbox{otherwise.} 
     \end{cases}
\]
Applying Thm.~\ref{thm:tropical_mobius}, we obtain the persistent cup-length function $\cupprod(\Xfunc)$ shown in the middle figure of Fig.~\ref{fig:klein-function-diagram}.
\end{example}

See 
\textsection \ref{sec:proof_cup_dgm} 
for the proof of Thm.~\ref{thm:tropical_mobius} and more examples of persistent cup-length diagrams.
It is worth noticing that the persistent cup-length diagram depends on the choice of the family of representative cocycles $\bsigma$, see Ex.~\ref{ex:dependence on sigma} below. 

\begin{example}[$\mathbf{dgm}_{\bsigma}^{\smile}(\Xfunc)$ depends on $\bsigma$]
\label{ex:dependence on sigma}
Let $\mathbb{RP}^2$ be the real projective plane. Consider the filtration $\Xfunc$ of the $2$-skeleton $S_2(\mathbb{RP}^2\times \mathbb{RP}^2)$ of the product space $\mathbb{RP}^2\times \mathbb{RP}^2$, given by:
\[\begin{tikzcd}[row sep=0pt]
\Xfunc:
& 
\color{red}\bullet \ar[r,hook]
&
{\color{red}\mathbb{RP}^2} \ar[r,hook]
	& 
{\color{red}\mathbb{RP}^2}\vee {\color{blue}\mathbb{RP}^2} \ar[r,hook]
	& 
S_2({\color{red}\mathbb{RP}^2}\times {\color{blue}\mathbb{RP}^2})
\\
&\text{\small{$t\in [0,1)$}} 
& \text{\small{$t\in [1,2)$}} 
& \text{\small{$t\in [2,3)$}}
& \text{\small{$t\geq 3$}}
	\end{tikzcd}\]
Let $\alpha$ be the $1$-cocycle born at $t=1$, and $\beta$ be the $1$-cocycles born at $t=2$ when the second copy of $\mathbb{RP}^2$ appears. 
See Fig.~\ref{fig:rp2} for two choices of representative cocycles $\bsigma$ and $\btau$ for $\caB_{\geq1}(\Xfunc)$, where these two choices only differ by the first dimensional cocycles associated with the bar $[1,\infty)$. 
For a detailed explanation of the cohomology rings of the above spaces, see 
\S \ref{app:computation for RP2}.

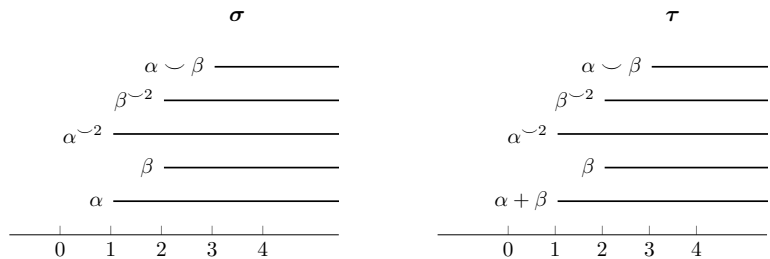
\begin{figure}[H]
\centering
\begin{tikzpicture}[scale=.8] 
    \begin{axis} [ 
    width=7cm,
    hide y axis,
    axis x line*=bottom,
    xtick={0,1,2,3,4},
    xticklabels={$0$, $1$, $2$, $3$, $4$},
    xmin=-1, xmax=5.5,
    ymin=0, ymax=1.6,]
    \addplot [mark=none] coordinates {(1,.2)}  node[left] {$\alpha$};
    \addplot [mark=none,thick] coordinates {(1.05,.2) (5.95,.2)};
    
    \addplot [mark=none] coordinates {(2,.4)}  node[left] {$\beta$};
    \addplot [mark=none,thick] coordinates {(2.05,.4) (5.95,.4)};
    
    \addplot [mark=none] coordinates {(1,.6)}  node[left] {$\alpha^{\smile 2}$};    
    \addplot [mark=none,thick] coordinates {(1.05,.6) (5.95,.6)};
    
    \addplot [mark=none] coordinates {(2,.8)}  node[left] {$\beta^{\smile 2}$};    
    \addplot [mark=none,thick] coordinates {(2.05,.8) (5.95,.8)};
    
    \addplot [mark=none] coordinates {(3,1)}  node[left] {$\alpha\smile \beta$};    
    \addplot [mark=none,thick] coordinates {(3.05,1) (5.95,1)};
    
    \node[mark=none] at (axis cs:3.5,1.3){\large$\bsigma$};
    \end{axis}
    \end{tikzpicture}
    \hspace{3em}
\begin{tikzpicture}[scale=.8] 
    \begin{axis} [ 
    width=7cm,
    hide y axis,
    axis x line*=bottom,
    xtick={0,1,2,3,4},
    xticklabels={$0$, $1$, $2$, $3$, $4$},
    xmin=-1.5, xmax=5.5,
    ymin=0, ymax=1.6,]
    \addplot [mark=none] coordinates {(1,.2)}  node[left] {$\alpha+\beta$};
    \addplot [mark=none,thick] coordinates {(1.05,.2) (5.95,.2)};
    
    \addplot [mark=none] coordinates {(2,.4)}  node[left] {$\beta$};
    \addplot [mark=none,thick] coordinates {(2.05,.4) (5.95,.4)};
    
    \addplot [mark=none] coordinates {(1,.6)}  node[left] {$\alpha^{\smile 2}$};    
    \addplot [mark=none,thick] coordinates {(1.05,.6) (5.95,.6)};
    
    \addplot [mark=none] coordinates {(2,.8)}  node[left] {$\beta^{\smile 2}$};    
    \addplot [mark=none,thick] coordinates {(2.05,.8) (5.95,.8)};
    
    \addplot [mark=none] coordinates {(3,1)}  node[left] {$\alpha\smile \beta$};    
    \addplot [mark=none,thick] coordinates {(3.05,1) (5.95,1)};
    
    \node[mark=none] at (axis cs:3.5,1.3){\large$\btau$};
    \end{axis}
    \end{tikzpicture}
\caption{Two choices of representative cocycles for the filtration $\Xfunc$ given by Ex.~\ref{ex:dependence on sigma}.} 
\label{fig:rp2}
\end{figure}

To obtain the cup-length diagram, we first compute $\caB(\Xfunc)^{*\bsigma 2}$ and $\caB(\Xfunc)^{*\btau 2}$:
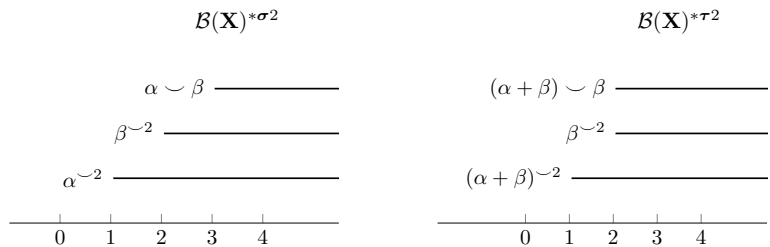
\begin{figure}[H]
\centering
\begin{tikzpicture}[scale=.8] 
    \begin{axis} [ 
    width=7cm,
    hide y axis,
    axis x line*=bottom,
    xtick={0,1,2,3,4},
    xticklabels={$0$, $1$, $2$, $3$, $4$},
    xmin=-1, xmax=5.5,
    ymin=0.4, ymax=1.6,]
    
    \addplot [mark=none] coordinates {(1,.6)}  node[left] {$\alpha^{\smile 2}$};    
    \addplot [mark=none,thick] coordinates {(1.05,.6) (5.95,.6)};
    
    \addplot [mark=none] coordinates {(2,.8)}  node[left] {$\beta^{\smile 2}$};    
    \addplot [mark=none,thick] coordinates {(2.05,.8) (5.95,.8)};
    
    \addplot [mark=none] coordinates {(3,1)}  node[left] {$\alpha\smile \beta$};    
    \addplot [mark=none,thick] coordinates {(3.05,1) (5.95,1)};
    
    \node[mark=none] at (axis cs:3.5,1.3){\large$\caB(\Xfunc)^{*\bsigma 2}$};
    \end{axis}
    \end{tikzpicture}
    \hspace{3em}
\begin{tikzpicture}[scale=.8] 
    \begin{axis} [ 
    width=7cm,
    hide y axis,
    axis x line*=bottom,
    xtick={0,1,2,3,4},
    xticklabels={$0$, $1$, $2$, $3$, $4$},
    xmin=-2, xmax=5.5,
    ymin=0.4, ymax=1.6,]
    
    \addplot [mark=none] coordinates {(1,.6)}  node[left] {$(\alpha+\beta)^{\smile 2}$};    
    \addplot [mark=none,thick] coordinates {(1.05,.6) (5.95,.6)};
    
    \addplot [mark=none] coordinates {(2,.8)}  node[left] {$\beta^{\smile 2}$};    
    \addplot [mark=none,thick] coordinates {(2.05,.8) (5.95,.8)};
    
    \addplot [mark=none] coordinates {(2,1)}  node[left] {$(\alpha+\beta)\smile \beta$};    
    \addplot [mark=none,thick] coordinates {(2.05,1) (5.95,1)};
    
    \node[mark=none] at (axis cs:3.5,1.3){\large$\caB(\Xfunc)^{*\btau 2}$};
    \end{axis}
    \end{tikzpicture}
\end{figure}

By Defn.~\ref{def:cup_dgm}, the persistent cup-length diagram associated to $\bsigma$ and $\btau$ are (see Fig.~\ref{fig:two-diagram}):
    \begin{figure}[H]
    \centering
    \begin{tikzpicture}[scale=0.65]
    \begin{axis} [ 
    axis y line=middle, 
    axis x line=middle,
    ytick={1,2,3,4,5.3},
    yticklabels={$1$,$2$,$3$,$4$,$\infty$},
    xticklabels={$1$,$2$,$3$,$4$,$\infty$},
    xtick={1,2,3,4,5.3},
    xmin=0, xmax=5.5,
    ymin=0, ymax=5.5,]
    \addplot [mark=none] coordinates {(0,0) (5.3,5.3)};
   \addplot[dgmcolor!60!white,mark=*] (1,5.3) circle (2pt) node[below,black]{\textsf{2}};
   \addplot[dgmcolor!100!white,mark=*] (2,5.3) circle (2pt) node[below,black]{\textsf{2}};
    \addplot[dgmcolor!100!white,mark=*] (3,5.3) circle (2pt) node[below,black]{\textsf{2}};
    \node[mark=none] at (axis cs:4,2){\Large $\mathbf{dgm}_{\bsigma}^{\smile}(\Xfunc)$};
    \end{axis}
    \end{tikzpicture}
    \hspace{1.5cm}
    \begin{tikzpicture}[scale=0.65]
    \begin{axis} [ 
    axis y line=middle, 
    axis x line=middle,
    ytick={1,2,3,4,5.3},
    yticklabels={$1$,$2$,$3$,$4$,$\infty$},
    xticklabels={$1$,$2$,$3$,$4$,$\infty$},
    xtick={1,2,3,4,5.3},
    xmin=0, xmax=5.5,
    ymin=0, ymax=5.5,]
    \addplot [mark=none] coordinates {(0,0) (5.3,5.3)};
    \addplot[dgmcolor!60!white,mark=*] (1,5.3) circle (2pt) node[below,black]{\textsf{2}};
   \addplot[dgmcolor!100!white,mark=*] (2,5.3) circle (2pt) node[below,black]{\textsf{2}};
   \node[mark=none] at (axis cs:4,2){\Large $\mathbf{dgm}_{\btau}^{\smile}(\Xfunc)$};
    \end{axis}
    \end{tikzpicture}
    \caption{The persistent cup-length diagrams $\mathbf{dgm}_{\bsigma}^{\smile}(\Xfunc)$ (left) and $\mathbf{dgm}_{\btau}^{\smile}(\Xfunc)$ (right), see Ex.~\ref{ex:dependence on sigma}.}
    \label{fig:two-diagram}
     \end{figure}
\end{example}

\medskip

See 
page~\pageref{Supplement-examp} in \textsection\ref{Supplement-examp} 
for more examples of persistent cup-length diagrams.  
In the next section, we develop an algorithm for computing the persistent cup-length diagram $\mathbf{dgm}_{\bsigma}^{\smile}(\Xfunc)$, which can be used to compute the persistent cup-length function $\cupprod(\Xfunc)$ due to Thm.~\ref{thm:tropical_mobius}.

\subsection{An algorithm for computing the persistent cup-length diagram over $\Z_2$}
\label{sec:algorithms}

Let $\Xfunc:\X_1\hookrightarrow\dots\hookrightarrow \X_N(=\X)$ be a finite filtration of a finite simplicial complex $\X$. Suppose that the barcode over positive dimensions $\caB:=\caB_{\geq 1}(\Xfunc)$ and a family of representative cocycles $\bsigma:=\{\sigma_I\}_{I\in \caB}$ are given. Because a finite filtration has only finitely many critical values, we assume that all intervals in the barcode are closed at the right end. If not, we replace the right end of each such interval with its closest critical value
to the left. Since each interval is considered together with a representative cocycle in this section, we will abuse the notation and write $\caB$ for the set $\{(I,\sigma_I)\}_{I\in \caB}$ as well. Let $\caB^{*_{\bsigma}\ell}$ be the set of $I_1*_{\bsigma}\dots *_{\bsigma}I_\ell$ where each $I_i\in \caB$. We compute $\left\{ \caB^{*_{\bsigma}\ell} \right \}_{\ell\geq 1}$ using: \label{para:overall_idea}

\RestyleAlgo{plain}
\begin{algorithm}[H]
\label{alg:main-simple}
\While{$\caB^{*_{\bsigma}\ell}\neq \emptyset$}{
    \For{$(I_1, \sigma_1)\in \caB$ and $(I_2, \sigma_2)\in \caB^{*_{\bsigma}\ell}$}{
        \If{$I_1*_{\bsigma} I_2\neq \emptyset$ \label{line:while-if-simple}}{
    Append $(I_1*_{\bsigma} I_2,\sigma_1\smile\sigma_2)$ to $\caB^{*_{\bsigma}(\ell+1)}$
    }}}
\end{algorithm}

The \textbf{line \ref{line:while-if-simple}} involves the computation of $\operatorname{supp}(\sigma_1\smile\sigma_2)$ which is some interval $[b_\sigma,d_\sigma]$ for $1\leq b_\sigma\leq d_\sigma\leq m$ such that $d_\sigma$ is simply the right end of $I_1\cap I_2$ and $b_\sigma$ is the left end of some $I\in \caB$, by Prop.~\ref{prop:support}. The computation of $b_\sigma$ is broken down in two steps: (1) compute the cup product (at cochain level) $\sigma:=\sigma_1\smile \sigma_2$, and (2) find $b_\sigma$ as the smallest $i\leq d_\sigma$ such that $\sigma|_{C_*(\X_i)}$ is not a coboundary.
Step (1) is already addressed by Alg.~\ref{alg:cup_product} on page~\pageref{alg:cup_product}. Let us now introduce an algorithm to address Step (2).

\subsection{Checking whether a cochain is a coboundary} 
As before, we assume a total order (e.g. the order given in \cite{bauer2021ripser}) on the simplices, and denote the ordered simplices by $S=\{\alpha_1<\dots<\alpha_m\}$, where $m$ is the number of simplices. We adopt the reverse ordering for the set of cosimplices $S^*:=\{\alpha_m^*<\dots<\alpha_1^*\}.$ Notice that 
$S^*$ forms a basis for the linear space of cochains. A $p$-cochain $\sigma$ is written as a linear sum of elements in $S^*$ uniquely. If $\alpha_j^*$ appears as a summand for $\sigma$, we denote $\alpha_j\in \sigma$.

Let $A$ be the coboundary matrix associated to the ordered basis $S^*$. 
Assume that \fbox{$R=AV$} is the reduced matrix of $A$ obtained from left-to-right column operations, given by the upper triangular matrix $V$. As a consequence, the pivots $\mathrm{Pivots}(R )$ of columns of $R$ are unique. Using all $i$-th row for $i\in \mathrm{Pivots}(R )$, we do bottom-to-top row reduction on $R$: \fbox{$U  R =\Lambda$}, such that $\Lambda $ has at most one non-zero element in each row and column, and $U $ is an upper triangular matrix. See 
Alg.~\ref{alg:row_reduction} on page~\pageref{sec:supp_algorithm} 
for the row reduction algorithm  (with complexity $O(m^2)$), which outputs the matrix $U$.
The following proposition allows us to use the row reduction matrix $U$ and $\mathrm{Pivots}(R)$ to check whether a cochain is a coboundary. 
For the proof of Prop.~\ref{prop:is_coboundary}, see 
\textsection \ref{sec:supp_algorithm}. 

\begin{restatable}{proposition}{iscoboundary} \label{prop:is_coboundary}
Given a $p$-cochain $\sigma$, let $y\in \Z_2^{m}$ be such that $\sigma=S^*\cdot y$. Let $R$ be the column reduced coboundary matrix, $U$ be the row reduction matrix of $R$. Then $\sigma$ is a coboundary, iff $\{i:\text{the i-th row of }(U\cdot y)\neq 0\}\subset   \mathrm{Pivots}(R)$.
\end{restatable}

\subparagraph*{Using the boundary matrix.} We can also use the boundary matrix to check whether a cocycle is a coboundary, see 
page~\pageref{para:boundary_reduction}. 
Using the boundary matrix, only column reduction is needed, while for the coboundary matrix both column reduction and row reduction are performed. However, it has been justified in \cite{bauer2021ripser} that reducing the coboundary matrix is more efficient than reducing the boundary matrix. Combined with the fact that the row reduction step does not increase the computation complexity of computing the persistent cup-length diagram, we will use the coboundary matrix in this paper.

\subsection{Main algorithm and its complexity}\label{sec:main_algo}

Motivated by the goal of obtaining a practical algorithm, and in order to control the complexity, we consider truncating filtrations up to a user specified dimension bound.

\subparagraph*{Truncation of a filtration.} \label{para:truncation} Fix a positive integer $\udim$. Given a filtration $\Xfunc:\X_1\hookrightarrow\dots\hookrightarrow \X_N(=\X)$ of a finite simplicial complex $\X$, let
$\X_i^{\udim+1}$ be the $(\udim+1)$-skeleton of $\X_i$ for each $i$. 
The \emph{$(\udim+1)$-dimensional truncation} of $\Xfunc$ is the filtration $\Xfunc^{\udim+1}:\X_1^{\udim+1}\hookrightarrow\dots\hookrightarrow \X_N^{\udim+1}$. 
Since $\Hfunc^{\leq \udim}(\Y)\cong\Hfunc^{\leq \udim}(\Y^{\udim+1})$ (as vector spaces) for any simplicial complex $\Y$, we conclude that $\Hfunc^{\leq \udim}(\Xfunc)\cong\Hfunc^{\leq \udim}\left(\Xfunc^{\udim+1}\right)$ as persistent vector spaces. Thus, the barcode $\caB(\Xfunc^{\udim+1})$ of $\Xfunc^{\udim+1}$ is equal to the barcode $\caB_{\leq \udim}(\Xfunc):=\sqcup_{1\leq p\leq k} \caB_p(\Xfunc)$. 

Let $\caB_{[1,k]}(\Xfunc):=\sqcup_{1\leq p\leq k} \caB_p(\Xfunc)$.
We introduce Alg.~\ref{alg:main} to compute the persistent cup-length diagram for the $(\udim+1)$-dimensional truncation of $\Xfunc$ over $\Z_2$, whose time complexity is described in terms of the variables below: 
\begin{itemize}
    \item $\udim$ is a dimension bound used to truncate the filtration;\label{para:parameter}
    \item $ m_\udim $ is the number of simplices with positive dimension in the $(\udim+1)$-skeleton $\X^{\udim+1}$ of $\X$;
    \item $c_\udim$ is the complexity of checking whether a simplex is alive at a given filtration parameter;
    \item $q_{k-1} :=\max_{1\leq \ell\leq k -1}\card\left(\left(\caB_{[1,k]}(\Xfunc)\right)^{*_{\bsigma}\ell}\right)$ (see Defn.~\ref{dfn:ell-product}). In particular, $q_1=\card\left(\caB_{[1,k]}(\Xfunc)\right).$
\end{itemize}

\begin{algorithm}
\SetKwData{Left}{left}\SetKwData{This}{this}\SetKwData{Up}{up}
\SetKwInOut{Input}{Input}\SetKwInOut{Output}{Output}
\Input{A dimension bound $\udim$, the ordered list of cosimplices $S^*$ from dimension $1$ to $\udim+1$, the column reduced coboundary matrix $R$ from dimension $1$ to $\udim+1$, and barcodes (annotated by representative cocycles) from dimension $1$ to $\udim$: $\mathcal{B}_{[1,\udim]}=\{(b_\sigma,d_\sigma,\sigma)\}_{\sigma\in \bsigma}$, where each $ \sigma$ is a representative cocycle for the bar $(b_\sigma,d_\sigma)$ and $\{ \sigma_1,\dots, \sigma_{q_1}\}$ is ordered first in the increasing order of the death time and then in the increasing order of the birth time.} 
\Output{A matrix representation $A_\ell$ of persistent cup-length diagram, and the lists of distinct birth times $\mathrm{b\_time}$ and death times $\mathrm{d\_time}$.} 
\BlankLine
$\mathrm{b\_time},\mathrm{d\_time}\gets \mathrm{unique}(\{b_\sigma\}_{ \sigma\in \bsigma}),\, \mathrm{unique}(\{d_\sigma\}_{ \sigma\in \bsigma})$\;
$ m_\udim,\, \ell,\, B_1 \gets \mathrm{card}(S^*),\, 1,\,\mathcal{B}_{[1,\udim]}$\;
$A_0=A_1\gets \mathrm{zeros}(\mathrm{card}(\mathrm{b\_time}), \mathrm{card}(\mathrm{d\_time}))$\;
$U\gets \mathrm{RowReduce}(R)$ \tcp*{$O( m_\udim ^2)$, 
Alg.~\ref{alg:row_reduction}
}
\For(\tcp*[f]{$O( m_\udim )$}){$(b_i,b_j)\in B_1$ \label{line:for}}{
$A_1(i,j)\gets 1$\;}
\While(\tcp*[f]{$O(\udim)$}){$A_{\ell-1}\neq A_{\ell}$ and $l\leq k-1$ \label{line:while}}{
    $B_{\ell+1}=\{\}$\; 
    \For(\tcp*[f]{$O( q_1\cdot q_{k-1})$}){$(b_{i_1},d_{j_1}, \sigma_1)\in B_1$ and $(b_{i_2},d_{j_2}, \sigma_2)\in B_{\ell}$ \label{line:while-for}
    }{
    $\sigma\gets \mathrm{CupProduct}( \sigma_1, \sigma_2,S^*)$ \tcp*{$O( m_\udim ^2\cdot c_\udim)$, Alg.~\ref{alg:cup_product}}
    $y\gets$ the vector representation of $\sigma$ in $S^*$\;
    $i\gets \max\{i': b_{i'}\leq d_{\min\{j_1,j_2\}}\}$\;
    $s_i\gets $ number of simplices alive at $b_i$\;
    \While(\tcp*[f]{$O((m_\udim)^2\cdot q_1)$}){$\{u:(U_{m_k+1-s_i:\,m_k,\, m_k+1-s_i:\,m_k}\cdot y_{m_k+1-s_i:\,m_k})(u)\neq 0\}\subset \mathrm{Pivots}(R)$ \label{line:while-while}}{
        $i\gets i-1$\;
        $s_i\gets $ number of simplices alive at $b_i$\;
        }
    \If{$b_i<d_{\min\{j_1,j_2\}}$}{
        Append $(b_i,d_{\min\{j_1,j_2\}},\sigma)$ to $B_{\ell+1}$\;
        $A_{\ell+1}(i,\min\{j_1,j_2\})\gets \ell$
        }
    }
    $\ell\gets \ell+1$.
    }
\Return $A_\ell,\mathrm{b\_time},\mathrm{d\_time}$.
\caption{Main algorithm: compute persistent cup-length diagram.}
\label{alg:main}
\end{algorithm}

\subparagraph*{Time complexity.}  

In Alg.~\ref{alg:main}, \textbf{line \ref{line:while-for}} runs no more than $q_1 \cdot q_{k-1} $ times, due to the definition of $q_1$ and $q_{k-1} $. 
The while loop in \textbf{line \ref{line:while-while}} runs no more than $\card(\operatorname{b\_time})\leq q_1$ times, and the condition of this while loop involves a matrix multiplication whose complexity is at most $O((m_k)^2)$.
Combined with other comments in Alg.~\ref{alg:main} and the fact that $k $ is a fixed constant,
the total complexity is upper bounded by 
\[O(k )\cdot O( q_1 \cdot q_{k-1} )\cdot O( (m_k )^2\cdot \max\{c_k ,q_1\}) \leq O( (m_k )^2\cdot q_1 \cdot q_{k-1}  \cdot \max\{c_k ,q_1\}).\]

Next, we estimate $q_{k-1} $ and $c_k $ using $q_1$, $m_k$ and $k $. Since each $B_\ell$ consists of $\ell$-fold $*_{\bsigma}$-products of elements in $B_1$, we have $ q_{k-1} = \max_{1\leq \ell\leq k-1}\card(B_\ell)\leq  (q_1 ) ^{k-1},$ which turns out to be a very coarse bound 
(see Rmk.~\ref{rmk:torus_complexity}). 
On the other hand, $c_k $ as the cost of checking whether a simplex is alive at a given parameter, is at most $m_k $ the number of simplices. Hence, the complexity of Alg.~\ref{alg:main} is upper bounded by $O( (m_k  )^{3}\cdot q_1^k)$. In addition, we have $q_1\leq m_k$, because in the matrix reduction algorithm for computing barcodes, bars are obtained from the pivots of the column reduced coboundary matrix and each column provides at most one pivot. Thus,
$O( (m_k )^2\cdot q_1 \cdot q_{k-1}  \cdot \max\{c_k ,q_1\})\leq O( (m_k  )^{3}\cdot q_1^k)\leq O( (m_k  )^{k +3}).$

Consider the Vietoris-Rips filtration arising from a metric space of $n$ points with the distance matrix $D$. Then \textbf{line \ref{line:last-if-cup}} of Alg.~\ref{alg:cup_product}, checking whether a simplex $a$ 
(represented by a set of at most $\udim+1$ indices into $[n]$)
is alive at the filtration parameter value $t$, can be done by checking whether $\max( D[a,a])\leq t$, with the constant time complexity $c_\udim=O(\udim^2)$. In summary, we have the following theorem. 
 
\label{para:complexity_VR}

\begin{restatable}[Complexity of Alg.~\ref{alg:main}]{theorem}{complexity}
\label{thm:comp-alg3} 
For an arbitrary finite filtration truncated up to dimension $(\udim+1)$, computing its persistent cup-length diagram via Alg.~\ref{alg:main} has complexity at most $O( (m_k )^2\cdot q_1 \cdot q_{k-1}  \cdot \max\{c_k ,q_1\})$. In terms of just $m_\udim$, the complexity of Alg.~\ref{alg:main} is at most $O( (m_\udim )^{\udim+3})$, since
$c_\udim\leq m_\udim$ and $q_{k-1}\leq (m_\udim)^{\udim-1}$. 

For the $(\udim+1)$-dimensional truncation of the Vietoris-Rips filtration arising from a metric space of $n$ points, the complexity of Alg.~\ref{alg:main} is improved to $O( (m_k )^2\cdot q_1^2 \cdot q_{k-1})$, which is at most $ O((m_\udim)^{\udim+3})\leq O(n^{\udim^2+5\udim+6})$. 
\end{restatable}

Notice that when $\udim=1$, 
the persistent cup-length diagram simply evaluates $1$ at each bar in the standard barcode, and $0$ elsewhere. When $\udim\geq 2$, the resulting persistent cup-length diagram becomes more informative and  captures certain topological features that the standard persistence diagram is not able to detect. This is  reflected in Ex.~\ref{ex:sigma-diag-klein}. 

Although the algorithm has not been tested on datasets yet, it is a practical algorithm, given that there are available software programs, such as Ripser (see \cite{bauerGithub}), which computes the barcode and extracts representative cocycles for Vietoris-Rips filtrations.  Note that, according to \cite{bauer2021ripser}, the implementation ideas `are also applicable to persistence computations for other filtrations as well'. 

\begin{remark}[Estimating the parameter $q_{k-1}$] 
\label{rmk:q_{k-1}}
The inequality $q_{k-1}\leq (m_\udim)^{\udim-1}$ is quite coarse in general. 
Consider a filtration consisting of contractible spaces, where $q_{k-1}$ is always $0$ but $m_k$ can be arbitrarily large. Even in the case when there is a reasonable number of cohomology classes with non-trivial cup products, $q_{k-1}$ can be much smaller than $(m_k)^{k-1}$. See 
Rmk.~\ref{rmk:torus_complexity}.
\end{remark}

\begin{remark}[Reducing the time complexity] Because cup products cannot live longer than their factors, discarding short bars will not result into loss of important information. In our algorithm, an extra parameter $\epsilon\geq 0$ can be added to discard all the bars in the barcode $B_1$ with length less than $\varepsilon$.
By doing so, since the cardinality of $B_1$ is decreased, one expects the runtime of Alg.~\ref{alg:main} (in particular inside the loop in \textbf{line 9}) to be significantly reduced. A similar trimming strategy can also be applied in the construction of the subsequent $B_\ell$s. 
\end{remark}

\subparagraph*{Correctness of the algorithm.} Checking whether a cocycle is a coboundary requires local matrix reduction for the given filtration parameter $d_{\min\{j_1,j_2\}}$, but a global matrix reduction is performed in the algorithm. The reason is that the coboundary matrix $A$, the column reduction matrix $V$ and the row reduction matrix $U$ are all upper-diagonal. Therefore, reducing the ambient matrix $A$ and then taking the bottom-right submatrix to get $\bar{U}$, is equivalent to reducing the submatrix of $A$ directly.

\bibliography{refs-arxiv} 

\begin{thebibliography}{10}

\bibitem{adamaszek2017vietoris}
Micha{\l} Adamaszek and Henry Adams.
\newblock The {V}ietoris--{R}ips complexes of a circle.
\newblock {\em Pacific Journal of Mathematics}, 290(1):1--40, 2017.
\newblock \href {https://doi.org/10.2140/pjm.2017.290.1}
  {\path{doi:10.2140/pjm.2017.290.1}}.

\bibitem{adamaszek2020homotopy}
Micha{\l} Adamaszek, Henry Adams, Ellen Gasparovic, Maria Gommel, Emilie
  Purvine, Radmila Sazdanovic, Bei Wang, Yusu Wang, and Lori Ziegelmeier.
\newblock On homotopy types of {V}ietoris--{R}ips complexes of metric gluings.
\newblock {\em Journal of Applied and Computational Topology}, 4(3):425--454,
  2020.
\newblock \href {https://doi.org/10.1007/s41468-020-00054-y}
  {\path{doi:10.1007/s41468-020-00054-y}}.

\bibitem{adams2014javaplex}
Henry Adams, Andrew Tausz, and Mikael Vejdemo-Johansson.
\newblock javaplex: A research software package for persistent (co)homology.
\newblock In {\em Mathematical Software, ICMS 2014 - 4th International
  Congress, Proceedings}, Lecture Notes in Computer Science (including
  subseries Lecture Notes in Artificial Intelligence and Lecture Notes in
  Bioinformatics), pages 129--136. Springer Verlag, 2014.
\newblock 4th International Congress on Mathematical Software, ICMS 2014 ;
  Conference date: 05-08-2014 Through 09-08-2014.
\newblock \href {https://doi.org/10.1007/978-3-662-44199-2_23}
  {\path{doi:10.1007/978-3-662-44199-2_23}}.

\bibitem{amann_2015}
Manuel Amann.
\newblock Computational complexity of topological invariants.
\newblock {\em Proceedings of the Edinburgh Mathematical Society},
  58(1):27–32, 2015.
\newblock \href {https://doi.org/10.1017/S0013091514000455}
  {\path{doi:10.1017/S0013091514000455}}.

\bibitem{aubrey2011persistent}
HB~Aubrey.
\newblock {\em Persistent cohomology operations}.
\newblock Ph{D} thesis, Duke University, 2011.

\bibitem{awodey2010category}
Steve Awodey.
\newblock {\em Category theory}.
\newblock Oxford university press, 2010.

\bibitem{bauerGithub}
Ulrich Bauer.
\newblock Ripser.
\newblock \url{https://github.com/Ripser/ripser}, 2016.

\bibitem{bauer2021ripser}
Ulrich Bauer.
\newblock Ripser: efficient computation of {V}ietoris--{R}ips persistence
  barcodes.
\newblock {\em Journal of Applied and Computational Topology}, pages 1--33,
  2021.
\newblock \href {https://doi.org/10.1007/s41468-021-00071-5}
  {\path{doi:10.1007/s41468-021-00071-5}}.

\bibitem{belchi2021a}
Francisco Belch{\'\i} and Anastasios Stefanou.
\newblock A-infinity persistent homology estimates detailed topology from point
  cloud datasets.
\newblock {\em Discrete \& Computational Geometry}, pages 1--24, 2021.
\newblock \href {https://doi.org/10.1007/s00454-021-00319-y}
  {\path{doi:10.1007/s00454-021-00319-y}}.

\bibitem{blumberg2017universality}
Andrew~J Blumberg and Michael Lesnick.
\newblock Universality of the homotopy interleaving distance.
\newblock {\em arXiv preprint arXiv:1705.01690}, 2017.

\bibitem{bubenik2015metrics}
Peter Bubenik, Vin De~Silva, and Jonathan Scott.
\newblock Metrics for generalized persistence modules.
\newblock {\em Foundations of Computational Mathematics}, 15(6):1501--1531,
  2015.
\newblock \href {https://doi.org/10.1007/s10208-014-9229-5}
  {\path{doi:10.1007/s10208-014-9229-5}}.

\bibitem{carlsson2009topology}
Gunnar Carlsson.
\newblock Topology and data.
\newblock {\em Bulletin of the American Mathematical Society}, 46(2):255--308,
  2009.

\bibitem{carlsson2020persistent}
Gunnar Carlsson.
\newblock Persistent homology and applied homotopy theory.
\newblock In {\em Handbook of Homotopy Theory}, pages 297--329. Chapman and
  Hall/CRC, 2020.

\bibitem{carlsson2009theory}
Gunnar Carlsson and Afra Zomorodian.
\newblock The theory of multidimensional persistence.
\newblock {\em Discrete \& Computational Geometry}, 42(1):71--93, 2009.
\newblock \href {https://doi.org/10.1007/s00454-009-9176-0}
  {\path{doi:10.1007/s00454-009-9176-0}}.

\bibitem{cohen2007Stability}
David Cohen-Steiner, Herbert Edelsbrunner, and John Harer.
\newblock Stability of persistence diagrams.
\newblock {\em Discrete \& Computational Geometry}, 37(1):103--120, 2007.
\newblock \href {https://doi.org/10.1007/s00454-006-1276-5}
  {\path{doi:10.1007/s00454-006-1276-5}}.

\bibitem{polanco2021}
Luis~Polanco Contreras and Jose Perea.
\newblock Persistent cup product for quasi periodicity detection.
\newblock
  \url{https://4c0aa4c9-c4b2-450c-a81a-c4a8e2d3f528.filesusr.com/ugd/58704f_dcd2001732bb4b3ab91900f99955241c.pdf},
  2021.
\newblock Second Graduate Student Conference: Geometry and Topology meet Data
  Analysis and Machine Learning (GTDAML2021).

\bibitem{cornea2003lusternik}
Octavian Cornea, Gregory Lupton, John Oprea, Daniel Tanr{\'e}, et~al.
\newblock {\em Lusternik-{S}chnirelmann category}.
\newblock Number 103 in Mathematical Surveys and Monographs. American
  Mathematical Society, 2003.

\bibitem{crawley2015decomposition}
William Crawley-Boevey.
\newblock Decomposition of pointwise finite-dimensional persistence modules.
\newblock {\em Journal of Algebra and its Applications}, 14(05):1550066, 2015.
\newblock \href {https://doi.org/10.1142/S0219498815500668}
  {\path{doi:10.1142/S0219498815500668}}.

\bibitem{de2011dualities}
Vin de~Silva, Dmitriy Morozov, and Mikael Vejdemo-Johansson.
\newblock Dualities in persistent (co)homology.
\newblock {\em Inverse Problems}, 27(12):124003, nov 2011.
\newblock \href {https://doi.org/10.1088/0266-5611/27/12/124003}
  {\path{doi:10.1088/0266-5611/27/12/124003}}.

\bibitem{de2011persistent}
Vin De~Silva, Dmitriy Morozov, and Mikael Vejdemo-Johansson.
\newblock Persistent cohomology and circular coordinates.
\newblock {\em Discrete \& Computational Geometry}, 45(4):737--759, 2011.
\newblock \href {https://doi.org/10.1007/s00454-011-9344-x}
  {\path{doi:10.1007/s00454-011-9344-x}}.

\bibitem{dlotko2014simplification}
Pawe{\l} D{\l}otko and Hubert Wagner.
\newblock Simplification of complexes for persistent homology computations.
\newblock {\em Homology, Homotopy and Applications}, 16(1):49--63, 2014.
\newblock \href {https://doi.org/10.4310/HHA.2014.v16.n1.a3}
  {\path{doi:10.4310/HHA.2014.v16.n1.a3}}.

\bibitem{edelsbrunner2008Persistent}
Herbert Edelsbrunner and John Harer.
\newblock Persistent homology-a survey.
\newblock {\em Contemporary mathematics}, 453:257--282, 2008.

\bibitem{farber2003topological}
Michael Farber.
\newblock Topological complexity of motion planning.
\newblock {\em Discrete and Computational Geometry}, 29(2):211--221, 2003.
\newblock \href {https://doi.org/10.1007/s00454-002-0760-9}
  {\path{doi:10.1007/s00454-002-0760-9}}.

\bibitem{frosini1990distance}
Patrizio Frosini.
\newblock A distance for similarity classes of submanifolds of a euclidean
  space.
\newblock {\em Bulletin of the Australian Mathematical Society},
  42(3):407--415, 1990.
\newblock \href {https://doi.org/10.1017/S0004972700028574}
  {\path{doi:10.1017/S0004972700028574}}.

\bibitem{frosini1992measuring}
Patrizio Frosini.
\newblock Measuring shapes by size functions.
\newblock In {\em Intelligent Robots and Computer Vision X: Algorithms and
  Techniques}, volume 1607, pages 122--133. International Society for Optics
  and Photonics, 1992.
\newblock \href {https://doi.org/10.1117/12.57059}
  {\path{doi:10.1117/12.57059}}.

\bibitem{ginot2019multiplicative}
Gr{\'e}gory Ginot and Johan Leray.
\newblock Multiplicative persistent distances.
\newblock {\em arXiv preprint arXiv:1905.12307}, 2019.

\bibitem{gonzalez2003hha}
Roc{\'\i}o Gonz{\'a}lez~D{\'\i}az and Pedro Real~Jurado.
\newblock Computation of cohomology operations of finite simplicial complexes.
\newblock {\em Homology, Homotopy and Applications (HHA), 5 (2), 83-93.}, 2003.

\bibitem{hatcher2000}
Allen Hatcher.
\newblock {\em {Algebraic topology}}.
\newblock Cambridge Univ. Press, Cambridge, 2000.
\newblock URL: \url{https://cds.cern.ch/record/478079}.

\bibitem{herscovich2018higher}
Estanislao Herscovich.
\newblock A higher homotopic extension of persistent (co)homology.
\newblock {\em Journal of Homotopy and Related Structures}, 13(3):599--633,
  2018.
\newblock \href {https://doi.org/10.1007/s40062-017-0195-x}
  {\path{doi:10.1007/s40062-017-0195-x}}.

\bibitem{huang2005cup}
Jonathan Huang.
\newblock Cup products in computational topology, 2005.
\newblock URL:
  \url{http://jonathan-huang.org/research/old/computationalcupproduct.pdf}.

\bibitem{kaczynski2010computing}
Tomasz Kaczynski, Pawe{\l} D{\l}otko, and Marian Mrozek.
\newblock Computing the cubical cohomology ring.
\newblock {\em Image-A: Applicable Mathematics in Image Engineering, 1 (3),
  137-142}, 2010.
\newblock URL: \url{http://hdl.handle.net/11441/26211}.

\bibitem{kang2021evaluating}
Louis Kang, Boyan Xu, and Dmitriy Morozov.
\newblock Evaluating state space discovery by persistent cohomology in the
  spatial representation system.
\newblock {\em Frontiers in Computational Neuroscience}, 15, 2021.
\newblock URL:
  \url{https://www.frontiersin.org/article/10.3389/fncom.2021.616748}, \href
  {https://doi.org/10.3389/fncom.2021.616748}
  {\path{doi:10.3389/fncom.2021.616748}}.

\bibitem{kim2021generalized}
Woojin Kim and Facundo M{\'e}moli.
\newblock Generalized persistence diagrams for persistence modules over posets.
\newblock {\em Journal of Applied and Computational Topology}, 5(4):533--581,
  2021.
\newblock \href {https://doi.org/10.1007/s41468-021-00075-1}
  {\path{doi:10.1007/s41468-021-00075-1}}.

\bibitem{kim2021spatiotemporal}
Woojin Kim and Facundo M{\'e}moli.
\newblock Spatiotemporal persistent homology for dynamic metric spaces.
\newblock {\em Discrete \& Computational Geometry}, 66(3):831--875, 2021.
\newblock \href {https://doi.org/10.1007/s00454-019-00168-w}
  {\path{doi:10.1007/s00454-019-00168-w}}.

\bibitem{leinster2014basic}
Tom Leinster.
\newblock {\em Basic category theory}, volume 143.
\newblock Cambridge University Press, 2014.

\bibitem{lim2020vietoris}
Sunhyuk Lim, Facundo M\'emoli, and Osman~Berat Okutan.
\newblock {V}ietoris-{R}ips persistent homology, injective metric spaces, and
  the filling radius.
\newblock {\em arXiv preprint arXiv:2001.07588}, 2020.

\bibitem{lupo2018persistence}
Umberto Lupo, Anibal~M. Medina-Mardones, and Guillaume Tauzin.
\newblock Persistence {S}teenrod modules.
\newblock {\em arXiv preprint arXiv:1812.05031}, pages arXiv--1812, 2018.

\bibitem{mac2013categories}
Saunders Mac~Lane.
\newblock {\em Categories for the working mathematician}, volume~5.
\newblock Springer Science \& Business Media, 2013.

\bibitem{maria2014gudhi}
Cl{\'e}ment Maria, Jean-Daniel Boissonnat, Marc Glisse, and Mariette Yvinec.
\newblock The {G}udhi library: Simplicial complexes and persistent homology.
\newblock In Hoon Hong and Chee Yap, editors, {\em Mathematical Software --
  ICMS 2014}, pages 167--174, Berlin, Heidelberg, 2014. Springer Berlin
  Heidelberg.
\newblock \href {https://doi.org/10.1007/978-3-662-44199-2_28}
  {\path{doi:10.1007/978-3-662-44199-2_28}}.

\bibitem{munkres1984elements}
James~R. Munkres.
\newblock {\em Elements of algebraic topology}.
\newblock Addison-Wesley, Menlo Park, CA, 1984.
\newblock \href {https://doi.org/10.1201/9780429493911}
  {\path{doi:10.1201/9780429493911}}.

\bibitem{patel2018generalized}
Amit Patel.
\newblock Generalized persistence diagrams.
\newblock {\em Journal of Applied and Computational Topology}, 1(3):397--419,
  2018.
\newblock \href {https://doi.org/10.1007/s41468-018-0012-6}
  {\path{doi:10.1007/s41468-018-0012-6}}.

\bibitem{robins1999towards}
Vanessa Robins.
\newblock Towards computing homology from finite approximations.
\newblock In {\em Topology proceedings}, volume~24, pages 503--532, 1999.

\bibitem{rudyak1999analytical}
Yuli~B. Rudyak.
\newblock On analytical applications of stable homotopy (the {A}rnold
  conjecture, critical points).
\newblock {\em Mathematische Zeitschrift}, 230(4):659--672, 1999.
\newblock \href {https://doi.org/10.1007/PL00004708}
  {\path{doi:10.1007/PL00004708}}.

\bibitem{Rudyak1999ONCW}
Yuli~B Rudyak.
\newblock On category weight and its applications.
\newblock {\em Topology}, 38(1):37--55, 1999.
\newblock URL:
  \url{https://www.sciencedirect.com/science/article/pii/S0040938397001018},
  \href {https://doi.org/https://doi.org/10.1016/S0040-9383(97)00101-8}
  {\path{doi:https://doi.org/10.1016/S0040-9383(97)00101-8}}.

\bibitem{sarin2017cup}
Parth Sarin.
\newblock Cup length as a bound on topological complexity.
\newblock {\em arXiv preprint arXiv:1710.06502}, 2017.

\bibitem{Schmiedl17}
Felix Schmiedl.
\newblock Computational aspects of the {G}romov---{H}ausdorff distance and its
  application in non-rigid shape matching.
\newblock {\em Discrete Comput. Geom.}, 57(4):854--880, June 2017.
\newblock \href {https://doi.org/10.1007/s00454-017-9889-4}
  {\path{doi:10.1007/s00454-017-9889-4}}.

\bibitem{SMALE198781}
Steve Smale.
\newblock On the topology of algorithms, {I}.
\newblock {\em Journal of Complexity}, 3(2):81--89, 1987.
\newblock URL:
  \url{https://www.sciencedirect.com/science/article/pii/0885064X87900215},
  \href {https://doi.org/https://doi.org/10.1016/0885-064X(87)90021-5}
  {\path{doi:https://doi.org/10.1016/0885-064X(87)90021-5}}.

\bibitem{yarmola2010persistence}
Andrew Yarmola.
\newblock {\em Persistence and computation of the cup product}.
\newblock Undergraduate honors thesis, Stanford University, 2010.

\bibitem{zomorodian2005computing}
Afra Zomorodian and Gunnar Carlsson.
\newblock Computing persistent homology.
\newblock {\em Discrete \& Computational Geometry}, 33(2):249--274, 2005.
\newblock \href {https://doi.org/10.1007/s00454-004-1146-y}
  {\path{doi:10.1007/s00454-004-1146-y}}.

\end{thebibliography}


\begin{thebibliography}{}

\end{thebibliography}

\begin{appendices}
Below, we provide a table of contents in order to facilitate reading the paper.  The \textsection 1, \textsection 2, and \textsection 3 constitute the main body of the paper. The other sections serve as appendices, containing all necessary details which do not fit in the first 3 sections. In particular, in \textsection\ref{sec:stability} we establish the stability of the persistent cup-length function and apply it to estimate the Gromov-Hausdorff distance $d_{\mathrm{GH}}$ between metric spaces.

The proof of Thm.~\ref{thm:tropical_mobius} is on page~\pageref{pf:tropical_mobius} in \textsection~\ref{sec:proof_cup_dgm}, and the proof of Thm.~\ref{thm:main-stability} is on page~\pageref{pf:main-stability} in \textsection~\ref{sec:cont-pers-inv}.


\tableofcontents

\section{Persistent invariants}
\label{sec:Persistent-invariants}
In this section, we define invariants and persistent invariants in a general setting. In classical topology, an \textit{invariant} is a
numerical quantity associated to a given topological space that remains invariant under a homeomorphism. In linear algebra, an \textit{invariant} is a numerical quantity that remains invariant under a linear isomorphism of vector spaces. Extending these notions to a general `persistence' setting in TDA, leads to the study of \textit{persistent invariants}, which are designed to
extract and quantify important information about the TDA structures, such as 
the \textit{rank invariant} for persistent vector spaces \cite{carlsson2009topology,carlsson2009theory}.

\subsection{Persistence theory}\label{sec:pers-theo}
We recall the notions of persistent objects and their morphisms.
For general definitions and results in category theory, we refer to \cite{awodey2010category,leinster2014basic,mac2013categories}.

\begin{definition}\label{defi:pers-object}
Let $\CC$ be a category.
We call any functor $\Ffunc:(\R,\leq)\to\CC$ a \textbf{persistent object (in $\CC$)}.
Specifically, a persistent object $\Ffunc:(\R,\leq)\to\CC$ consists of 
\begin{itemize}
    \item for each $t\in\R$, an object $\Ffunc_t$ of $\CC$,
    \item for each inequality $t\leq s$ in $\R$, a morphism $f_t^s:\Ffunc_t\to\Ffunc_s$, such that 
    \begin{itemize}
        \item $f_t^t=\mathbf{id}_{\Ffunc_t}$
        \item $f_s^r\circ f_t^s=f_t^r$, for all $t\leq s\leq r$.
    \end{itemize}
     
\end{itemize}
\end{definition}

\begin{definition}
Let $\Ffunc,\Gfunc:(\R,\leq)\to\CC$ be two persistent objects in $\CC$.
A \textbf{natural transformation from $\Ffunc$ to $\Gfunc$}, denoted by $\varphi:\Ffunc\Rightarrow \Gfunc$, consists of an $\R$-indexed family $(\varphi_t:\Ffunc_t\to\Gfunc_t)_{t\in\R}$ of morphisms in $\CC$, such that the diagram
\[
\xymatrix
{
\mathbf{F}_t\ar[rr]^{\varphi_t}\ar[d]_{f_t^s}  &  &
\mathbf{G}_t \ar[d]^{g_t^s} \\
	  \mathbf{F}_s\ar[rr]^{\varphi_b}& & \mathbf{G}_s\\
}
\]
commutes for all $t\leq s$.
\end{definition}

\begin{example} \label{ex:VR}
\begin{itemize}
    \item 
Let $X$ be a finite metric space and let $\VR_t(X)$ denote the \textit{Vietoris-Rips complex of $X$} at the scale parameter $t$, which is a simplicial complex defined as $\VR_t(X):=\{\alpha\subset X:\diam(\alpha)\leq t\}.$ Let $|\cdot|$ denote the geometric realization of a simplicial complex, and we denote   
 \[
     \X_t :=  
      \begin{cases}
      |\VR_t(X)| , & \text { if } t\geq 0;\\
      |\VR_0(X)| , &  \text{ otherwise.}
     \end{cases}
\]
For each inequality $t\leq s$ in $\R$, we have the inclusion $\iota_t^s :\X_t\hookrightarrow\X_s$ giving rise to a persistent space $\Xfunc:(\R,\leq)\to \Top$.

\item Applying the $p$-th homology functor to a persistent topological space $\Xfunc$, for each $t\in\R$ we obtain the vector space $\Hfunc_p(\X_t)$ and for each pair of parameters $t\leq s$ in $\R$, we have the linear map in (co)homology induced by the inclusion $\X_t\hookrightarrow\X_s$. This is another example of a persistent object, namely a \textit{persistent vector space} $\Hfunc_p(\Xfunc):(\R,\leq)\to \Vect$. Dually, by applying the $p$-th cohomology functor we obtain a persistent vector space $\Hfunc^p(\Xfunc):(\R,\leq)\to \Vect$ which is a contravariant functor.
\end{itemize}
\end{example}

\subsection{Injective-surjective invariants}\label{sec:eip-mono-inv}
We define the notions of (iso-)invariants and inj-surj invariants in a general categorical setting.

\begin{definition}[Invariant]\label{defi:invariant}
Let $\CC$ be any category.
An \textbf{iso-invariant} or simply an \textbf{invariant} in $\CC$ is any map $J:\ob(\CC)\to\N$ such that: if $X\cong Y$ in $\CC$, then $J(X)=J(Y)$. 
\end{definition}
Examples of iso-invariants include the \emph{cardinality invariant} in the category of sets, the \emph{dimension invariant} in the category of vector spaces and the \emph{number of connected components} in the category of topological spaces.



\begin{definition}[Inj-surj invariant]\label{defi:injective-surjective-inv}
Let $\CC$ be any category.
An \textbf{injective-surjective invariant (or inj-surj invariant)} in $\CC$ is any map $J:\ob(\CC)\to\mathbb{N}$ such that: 
\begin{itemize}
    \item if there is an injective morphism $X\hookrightarrow Y$, then $J(X)\leq J(Y)$;
    \item if there is a surjective morphism $ X\twoheadrightarrow Y$, then $J(X)\geq J(Y).$
\end{itemize}
\end{definition}

Any inj-surj invariant is clearly an iso-invariant.

\begin{example}
Let $\Vect$ be the category of finite dimensional vector spaces over the field $K$ whose morphisms are $K$-linear maps. The dimension invariant $\mathbf{dim}:\ob(\Vect)\to\N$, that assigns to each vector space its dimension, is an example of an inj-surj invariant.
\end{example}

A contravariant functor from $\DD$ to $\CC$ is equivalent to a covariant functor from $\DD$ to the opposite category\footnote{The opposite category $\CC^{op}$ of a category $\CC$ is the category whose objects are the same as $\CC$, but whose arrows are the arrows of $\CC$ with the reverse direction.} $\CC^{op}$ of $\CC$. 

\begin{remark}
\label{rem:self-duality of inj-surj invariants}
Any inj-surj invariant $J:\ob(\CC)\to\N$ of $\CC$ is also an inj-surj invariant in the opposite category $\CC^{op}$ of $\CC$. 
The injective morphisms $X\hookrightarrow Y$ in $\CC^{op}$ are exactly the surjective morphisms $Y\twoheadrightarrow X$ in $\CC$ and the surjective morphisms $X\twoheadrightarrow Y$ in $\CC^{op}$ are exactly the injective morphisms $Y\hookrightarrow X$ in $\CC$.
Thus, we have
\[\left(X\hookrightarrow Y\text{ in }\CC^{op}\right)\Rightarrow \left(Y\twoheadrightarrow X\text{ in }\CC\right)\Rightarrow \left(J(X)\leq J(Y)\right),\]
\[\left(X\twoheadrightarrow Y\text{ in }\CC^{op}\right)\Rightarrow \left(Y\hookrightarrow X\text{ in }\CC\right)\Rightarrow \left(J(X)\geq J(Y)\right).\]
\end{remark}

\subsection{Persistent invariants}\label{sec:pers-inv}
We define the notion of a persistent invariant in a general persistence setting.

\begin{definition}[Persistent invariant]\label{def:p_invariant}
Let $\CC$ be a category with images and let $J:\ob(\CC)\to\N$ be an inj-surj invariant. 
For any given persistent object $\Ffunc:(\R,\leq)\to\CC$, we associate the map
\begin{align*}
    J(\Ffunc):\Int&\to\N\\
    [a,b]&\mapsto J(\mathbf{Im}(f_a^b)).
\end{align*}
$J(\Ffunc)$ is called the \text{persistent invariant associated to $\Ffunc$}.
\end{definition}

For brevity in notation, for any morphism $f:X\to Y$ in $\CC$, we will denote $J(f(X))$ and $J(\mathbf{Im} (f))$ simply by $J(f)$.

\begin{lemma}
\label{lem:gf}
Let $\CC$ be a category with images and let $J$ be an inj-surj invariant.
Then, for any two morphisms $f:X\to Y$ and $g:Y\to Z$ we have
\[J(g\circ f)\leq\min\{J(f),J(g)\}.\]
\end{lemma}
\begin{proof}
We claim that $J(g\circ f)\leq J(f)$ and $J(g\circ f)\leq J(g)$.
Indeed:
\begin{itemize}
    \item The canonical surjective morphism $f(R)\twoheadrightarrow g(f(R))$ yields the inequality $ J(f(R))\geq J(g(f(R)))$, or equivalently $J(g\circ f)\leq J(f)$.
    \item  The canonical injective morphism  $g(f(R))\hookrightarrow  g(S)$, yields the inequality $J(g(f(R)))\leq J(g(S))$, or equivalently $J(g\circ f)\leq J(g)$.
\end{itemize}
\vspace{-1em}
\end{proof}

\begin{proposition}[Functoriality of persistent invariants]
\label{thm:Functoriality of persistent invariants}
Let $\CC$ be a category with images and let $J:\ob(\CC)\to\N$ be an inj-surj invariant in $\CC$.
Let $\mathbf{F}:(\mathbb{R},\leq)\to \CC$ be a persistent object in $\CC$. 
Then, the associated persistent invariant $J(\Ffunc)$ forms a functor $J(\Ffunc):(\Int,\subset)\to(\N,\geq)$.
Namely: 
\[[a,b]\subseteq [c,d]\Rightarrow J(f_a^b)\geq J(f_{c}^{d}).\]
\end{proposition}
\begin{proof}
Suppose we have $[a,b]\subseteq [c,d]$. 
Then, by applying  Lem.~\ref{lem:gf}, we obtain:
\begin{align*}
    J(f_{c}^{d})&=J(f_{b}^{d}\circ f_{a}^{b}\circ f_{c}^{a})\\
    &\leq \min\{J(f_b^{d}),J(f_{a}^{b}\circ f_{c}^{a})\}\\
    &\leq J(f_{a}^{b}\circ f_{c}^{a})\\
  &\leq \min\{J(f_{a}^{b}),J(f_{c}^{a})\}\\
 &\leq J(f_{a}^{b}).
 \end{align*}
 \vspace{-1em}
\end{proof}

\begin{example}
\label{ex:pers-rank}
Recall that $\mathbf{dim}:\ob(\Vect)\to\N$ is an inj-surj invariant associated to the category $\Vect$ of finite dimensional vector spaces over $K$. Let us consider $\mathbf{V}:(\R,\leq)\to\Vect$, a persistent $K$-vector space (also called a persistent module).
Then, we can consider the functor $\rk(
\mathbf{V}):(\Int,\subset)\to(\N,\geq)$ that associates each interval $[a,b]$ to the rank $\rank(v_a^b)$ of the transition map $v_a^b:\Vfunc_a\to \Vfunc_b$. The map $\rk(
\mathbf{V})$ is known in TDA as the \textbf{rank invariant} \cite{kim2021generalized}.
Note that $\rk(v_a^b)=\mathbf{dim}(\mathbf{Im}(v_a^b))$, thus the rank invariant is an example of a persistent invariant. 
\end{example}


\section{Proofs and details for \textsection \ref{sec:Persistent-cup-length} (persistent cup-length function)}
\label{sec:supp_cup_l_func}

\subsection{Supplementary materials for \textsection \ref{sec:cup_legnth}}


\begin{definition}
A \textbf{ring} $(R,+,\bullet)$ is a set
$R$ equipped with an addition operation $+$ and a multiplication operation $\bullet$, such that (i) $(R,+)$ is an abelian group, (ii) $(R,\bullet)$ is a monoid, and  (iii) multiplication is distributive with respect to addition, i.e. $a\bullet (b+c)=(a\bullet b)+(a\bullet c)$ and $(a+b)\bullet c=(a\bullet c)+(b\bullet c)$.
\end{definition}

\begin{definition}
A ring $(R,+,\bullet)$ is called \textbf{graded ring} if there exists a family of subgroups $\lbrace R_{n}\rbrace _{n\in \N}$ of $R$ such that $R=\bigoplus_{n\in \N} R_n$ (as abelian groups), and $R_a\bullet R_b\subseteq R_{a+b}$ for all $a,b\in \N$.
Let $R$ and $S$ be two graded rings. 
A ring homomorphism $\varphi:R\to  S$ is called a \textbf{graded homomorphism} if it preserves the grading, i.e.~$\varphi(R_n)\subset S_n$, for all $n\in\N$.
\end{definition}
Recall: 
\deflength*

Here are some properties of the (cup-)length invariant that we will use.
\begin{restatable}[Compute cup-length using basis]
{proposition}{propcuplengthbasis}
\label{prop:len_of_basis}
Let $R$ be a graded ring. Suppose $B=\bigcup_{p\geq 1} B_p$, where each $B_p$ generates $R_p$ as a monoid under addition. Then $\len(R)=\sup\left\{\ell\geq 1\mid B^{\smile \ell}\neq \{0\}\right\}.$ In the case of cohomology ring, for each $p\geq 1$ let $B_p$ be a set that linearly spans $\Hfunc^p(\X)$, and let $B:=\bigcup_{p\geq 1} B_p$. Then $\cupprod(\X)=\sup\left\{\ell\geq 1\mid B^{\smile \ell}\neq \{0\}\right\}$. 
\end{restatable}

\begin{proof}[Proof of Prop. \ref{prop:len_of_basis}] 
It follows from the definition that $\len(R)=\sup\big\{\ell\geq 1\mid(\bigcup_{n\geq 1} R_n )^\ell\neq \{0\}\big\}.$ We claim that $\left(\bigcup_{n\geq 1} R_n \right)^\ell\neq \{0\}$ iff $B^{\smile \ell}\neq \{0\}$. Indeed, whenever $\eta_1\bullet\dots\bullet\eta_\ell\neq 0$, where each $\eta_i\in \bigcup_{n\geq 1} R_n$, we can write every $\eta_i$ as a linear sum of elements in $B$. Thus, $\eta$ can be written as a linear sum of elements in the form of $r_1\bullet\dots\bullet r_\ell$, where each $r_j\in B$. Because $\eta\neq 0$, there must be a summand $r_1\bullet\dots\bullet r_\ell\neq 0$. Therefore, $B^{\smile \ell}\neq \{0\}$. 
\end{proof}

\begin{definition}
We define the \textbf{persistent length function} of a persistent graded ring $\Rfunc$ as the function $\len(\Rfunc):\Int\to\N$ given by $\langle t,s \rangle\mapsto\len\left(\mathbf{Im}(\Rfunc(\langle t,s\rangle)\right).$ 
Here the notation $\mathbf{Im}(\Rfunc(\langle t,s\rangle)$ is defined in a similar convention to Rmk.~\ref{rmk:image-ring}. In particular, $\mathbf{Im}(\Rfunc([t,s]))=\mathbf{Im}(\Rfunc_s\to \Rfunc_t)$.
\end{definition}

\begin{restatable}{proposition}{propmonotone}
\label{prop:injective-surjective-length}
The length invariant $\len:\ob(\Ring)\to\N$ is an inj-surj invariant in $\Ring$. Namely: 
\begin{itemize}
    \item If $f:R\hookrightarrow S$ is an injective morphism, then
$\len(R)\leq \len(S)$.
    \item If $f:R\twoheadrightarrow S$ is a surjective morphism, then $\len(R)\geq \len(S)$.
\end{itemize}
\end{restatable}

\begin{proof}[Proof of Prop. \ref{prop:injective-surjective-length}]
Let $f:R\hookrightarrow S$ be a graded ring injective morphism, and suppose $\len(R)=\ell$. Then there exist $\eta_1, \dots,\eta_\ell\in \bigcup_{p\geq 1}R_p$, such that
$\eta:=\eta_1 \bullet\dots\bullet\eta_\ell\neq0 \in R$. 
Then, $\len(S)\geq \ell$, because 
\begin{align*}
    f(\eta_1) \bullet \dots \bullet f(\eta_\ell)&=f(\eta_1 \bullet\dots\bullet\eta_\ell)\\
    &=f(\eta)\neq0\hspace{1em}\text{ (since }\eta\neq0\text{ and }f\text{ is an injective morphism).}
\end{align*}

Let $f:R\twoheadrightarrow S$ be a graded ring surjective morphism, and suppose $\len(S)=\ell$. Because $f$ is surjective, 
$f(\alpha_1)\bullet\dots \bullet f(\alpha_\ell)\neq 0\in S$.
Since $f$ is a ring homomorphism, then  $f(\alpha_1\bullet \dots \bullet\alpha_\ell)\neq 0\in S$.
For the same reason, we have $\alpha_1\bullet \dots \bullet \alpha_\ell\neq 0\in R$.
Thus, $\len(R)\geq \ell$.
\end{proof}

\begin{proposition}
\label{prop:tensor-direct product}
Let $f:R\to R'$, $g:S\to S'$ be morphisms in $\Ring^{op}$, and $\X,\Y$ be two spaces. We define the length of $f$ to be the length $\len(\im f)$ of its image. Then:
\begin{align*}
    \len(R\otimes S)&=\len(R)+\len(S) & \len(R\times S)&=\max\{\len(R),\len(S)\}\\
    \len(f\otimes g)&=\len(f)+\len(g)&\len(f\times g)&=\max\{\len(f),\len(g)\}\\
    \cupprod(\X\times \Y)&=\cupprod(\X)+\len(\Y) & \cupprod(\X\amalg \Y)&=\max\{\cupprod(\X),\cupprod(\Y)\}.
\end{align*}
\end{proposition}

\begin{proof}[Proof of Prop. \ref{prop:tensor-direct product}] 
We first prove that $\len(R\otimes S)= \len(R)+\len(S)$. The inequality $\len(R\otimes S)\leq \len(R)+\len(S)$ is trivial. For `$\geq$', let $B_p:=\bigcup_{j=0}^p R_j\otimes S_{p-j}$ for $p\geq 1$, and let $B:=\bigcup_{p\geq 1}B_p$. Since each $B_p$ generates $(R\otimes S)_p$ under addition, we apply Proposition \ref{prop:len_of_basis} to obtain that $\len(R\otimes S)=\sup\{\ell\geq 1\mid B^{\smile \ell}\neq 0\}$. Assume that $B^{\smile \ell}\neq 0$. Then there exists elements $r_1,\dots,r_\ell\in R$ and $s_1,\dots,s_\ell\in S$ such that $\deg(r_i)+\deg(s_i)\geq 1$ for each $i$, and $(r_1\otimes s_1)\bullet\dots\bullet(r_\ell\otimes s_\ell)\neq 0$. It follows that $r_1\bullet\dots\bullet r_\ell\neq 0$ and $s_1\bullet\dots\bullet s_\ell\neq 0$. For every $i$ we have $\deg(r_i)\geq 1$ or $\deg(s_i)\geq 1$, and thus there are at least $\ell$ elements in the set $\{r_i\}_{i=1}^\ell\cup \{s_i\}_{i=1}^\ell$ with positive dimension. It follows that $\len(R)+\len(S)\geq \ell.$
 
The equality $\len(R\times S)=\max\{\len(R),\len(S)\}$ is straightforward, because each $\eta\in R\times S$ is in the form $\eta=(r,s)$ for $r\in R$ and $s\in S$. Indeed, given $\eta_i=(r_i,s_i)$ for $i=1,\dots,\ell$, we have $\eta_1 \bullet \dots \bullet \eta_\ell\neq 0$, if and only if, either $r_1\bullet\dots \bullet r_\ell\neq 0$ or $s_1\bullet\dots \bullet s_\ell\neq 0$.

The other cases follow directly by recalling the facts $(f\otimes g)(R\otimes S)=f(R)\otimes g(S)$, $(f\times g)(R\times S)=f(R)\times g(S)$, $\Hfunc^*(\X\amalg \Y)\cong\Hfunc^*(\X)\times\Hfunc^*(\Y)$ and $\Hfunc^*(\X\times \Y)\cong\Hfunc^*(\X)\otimes\Hfunc^*(\Y)$, where the last one follows from \cite[Thm.~3.15]{hatcher2000} and the fact that we are using field coefficients.
\end{proof}

\subsection{Supplementary materials for \textsection \ref{sec:cup_l_func}}
\label{supp:sec_cup_l_func}

\begin{remark}\label{rmk:rep_cocycle} 
Let $\sigma_I$ be a representative cocycle associated to an interval $ I\in \caB(\Xfunc)$. Then, it follows from the definition of representative cocycles that 
for any $t\leq s$, $\Hfunc^*(\iota_t^s )([\sigma_I]_s)=\left[\sigma_I|_{C_p(\X_t)}\right]\neq 0\iff [t,s]\subset I$.
\end{remark}


\propimagering*
\begin{proof}[Proof of Prop.~\ref{prop:generators-of-ring}] 
Given a space $\X$, the cohomology ring $\Hfunc^*(\X)\in \Ring$ is a graded ring generated by the graded cohomology vector space $\Hfunc^*(\X)\in \Vect$, under the operation of cup products. It is clear that any linear basis for $\Hfunc^*(\X)$ also generates the ring $\Hfunc^*(\X)$, under the cup product. Given an inclusion of spaces $\X\xhookrightarrow{\iota}\Y$, let $f:\Hfunc^*(\Y)\to \Hfunc^*(\X)$ denotes the induced cohomology ring morphism. Let $A$ be a linear basis for $\Hfunc^*(\Y)$. Since $A$ also generates $\Hfunc^*(\Y)$ as a ring, the image $f(A)$ generates $f(\Hfunc^*(\Y))$ as a ring. 

Now, let $\Hfunc^*(\iota_t^s ):\Hfunc^*(\X_s)\to \Hfunc^*(\X_t)$ denote the cohomology map induced by the inclusion $\iota_t^s:\X_t\hookrightarrow \X_s$. The set $A:=\{[\sigma_I]_s: s\in I\in\caB(\Xfunc)\}$ forms a linear basis for $\Hfunc^*(\X_s)$, and thus $\Hfunc^*(\iota_t^s )(A)$ generates $\mathbf{Im}(\Hfunc^*(\iota_t^s ))$ as a ring. On the other hand, it follows from Remark \ref{rmk:rep_cocycle} that 
$\Hfunc^*(\iota_t^s )(A)=\left\{\Hfunc^*(\iota_t^s  )\left([\sigma_I]_s\right): [t,s]\subset I\in\caB(\Xfunc)\right\}=\left\{[\sigma_I]_t: [t,s]\subset I\in\caB(\Xfunc)\right\}.$
\end{proof}

\propproperty*
\begin{proof}[Proof of Prop. \ref{prop:prod-coprod-pers-cup}]
By functoriality of products, disjoint unions, and wedge sums, we can define the persistent spaces: $\Xfunc\times\Yfunc:=(\{\X_t\times \Y_t\}_{t\in\R},\{f_{t}^{s}\times g_{t}^{s}\})$, $\Xfunc\amalg\Yfunc:=(\{\X_t\amalg \Y_t\}_{t\in\R},\{f_{t}^{s}\amalg g_{t}^{s}\})$, and $\Xfunc\vee\Yfunc:=(\{\X_t\vee \Y_t\}_{t\in\R},\{f_{t}^{s}\vee g_{t}^{s}\})$. 
We prove this proposition for the case when $\Int$ is the set of closed intervals. For other types of intervals, a similar proof follows.

Let $[a,b]\in \Int$. Using the contravariance property of the cohomology ring functor $\Hfunc^*$, we obtain:
\begin{align*}
    \cupprod\left(\Xfunc\times\Yfunc \right)([a,b])&=\len\left(\Hfunc^*(f_a^b\times g_{a}^{b})\right)\\
    &=\len\left(\Hfunc^*(f_a^b)\otimes \Hfunc^*(g_a^b)\right)\\
    &=\len\left(\Hfunc^*(f_{a}^{b})\right)+\len\left(\Hfunc^*(g_a^b)\right)\\
    &=\cupprod(\Xfunc)([a,b])+\cupprod(\Yfunc)([a,b]),
\end{align*}
\begin{align*}
    \cupprod\left(\Xfunc\amalg \Yfunc\right)([a,b])&=\len\left(\Hfunc^*(f_{a}^{b}\amalg g_{a}^{b})\right)\\
    &=\len\left(\Hfunc^*(f_{a}^{b})\times \Hfunc^*(g_{a}^{b})\right)\\
    &=\max\left\{\len\left(\Hfunc^*(f_{a}^{b})\right),\len\left(\Hfunc^*(g_{a}^{b})\right)\right\}\\
    &=\max\left\{\cupprod(\Xfunc)([a,b]),\cupprod(\Yfunc)([a,b])\right\}\text{, and}
    \\
        \cupprod\left(\Xfunc\vee \Yfunc\right)([a,b])&=\len\left(\Hfunc^*(f_a^b\vee g_{a}^{b})\right)\\
    &=\len\left(\Hfunc^*(f_{a}^{b})\times \Hfunc^*(g_{a}^{b})\right)\\
    &=\max\left\{\len\left(\Hfunc^*(f_{a}^{b})\right),\len\left(\Hfunc^*(g_{a}^{b})\right)\right\}\\
    &=\max\left\{\cupprod(\Xfunc)([a,b]),\cupprod(\Yfunc)([a,b])\right\}.
\end{align*}
\end{proof}


\section{Proofs and details for \textsection \ref{sec:comp_cup_l} (persistent cup-length diagram)}
\label{sec:details_persistent cup-length diagram}

This section is a supplement of \textsection \ref{sec:comp_cup_l}, including the proofs for results in \textsection \ref{sec:cup_dgm}, examples of persistent cup-length diagrams and details about the main algorithm.

\subsection{Proof of Thm.~\ref{thm:tropical_mobius} and other results in \textsection \ref{sec:cup_dgm}}
\label{sec:proof_cup_dgm}

\begin{proposition} \label{prop:propertiy_of_*_sigma}
Let $\langle b_i,d_i\rangle$, $i=1,\dots,\ell$, be as in Defn.~\ref{dfn:ell-product}. 
The $\ell$-fold $*_{\bsigma}$-product of $\langle b_1,d_1\rangle,\dots,\langle b_\ell,d_\ell\rangle$  is symmetric, i.e.~for any permutation $\rho$ of $\{1,2,\dots,\ell\}$, 
\[\langle b_1,d_1\rangle*_{\bsigma}\dots*_{\bsigma}\langle b_\ell,d_\ell\rangle=\langle b_{\rho(1)},d_{\rho(1)}\rangle*_{\bsigma}\dots*_{\bsigma}\langle b_{\rho(\ell)},d_{\rho(\ell)}\rangle.\]
\end{proposition}

\begin{proof}[Proof of Prop.~\ref{prop:propertiy_of_*_sigma}]
The proof follows by (i) the symmetry of $\max$ and $\min$, and (ii) by applying the property of cup product that: for any pair $\alpha,\beta$ of cochains, \[\alpha\smile\beta=(-1)^s\beta\smile\alpha\text{, for some integer }s,\]
which in turn, implies that for sufficiently small $\delta>0$,
\[[\sigma_{I_1}\smile\dots\smile\sigma_{I_\ell}]_{\min\{d_1,\dots,d_\ell\}-\delta}=[0]\Leftrightarrow [\sigma_{I_{\rho(1)}}\smile\dots\smile\sigma_{I_{\rho(\ell)}}]_{\min\{d_1,\dots,d_\ell\}-\delta}=[0],\]
for any permutation $\tau$ of $\{1,2,\dots,\ell\}$. \end{proof}

Recall:
\mobuis*
\begin{proof}[Proof of Thm.~\ref{thm:tropical_mobius}]
\label{pf:tropical_mobius} 
Without loss of generality, the theorem is proved in the case when $\Int$ is the set of closed intervals. For other cases, a similar discussion follows.
\medskip

Let $I:=[a,b]$ be a closed interval. If $\cupprod(\Xfunc)([a,b])=0$, then the image ring $\mathbf{Im}\left(\Hfunc^*(\X_b)\to \Hfunc^*(\X_a)\right)$ is trivial in positive dimensions. We claim that for any $[c,d]\supset[a,b]$, $\mathbf{dgm}_{\bsigma}^{\smile}(\Xfunc)([c,d])=0$. Assume not, then $\mathbf{dgm}_{\bsigma}^{\smile}(\Xfunc)([c,d])>0$ for some $[c,d]\supset[a,b]$, which necessarily means that there is a bar associated with a positive-dimensional cocycle that contains $[a,b]$. This contradicts with $\mathbf{Im}\left(\Hfunc^{\geq1}(\X_b)\to \Hfunc^{\geq1}(\X_a)\right)=0.$
\medskip

We now assume $\cupprod(\Xfunc)([a,b])\neq0$ and define 
\[B:=\{[\sigma_J]_a\mid \caB_{\geq 1}(\Xfunc)\ni J\supseteq [a,b]\}.\]
Recall that for $J=[c,d]$ in the barcode, $\sigma_J$ is a cocycle in $\X_d$ and $[\sigma_J]_a$ is the cohomology class of the restriction $\sigma_J|_{C_p(\X_a)}$, if the dimension of $\sigma_J$ is $p$.
Then,
\begin{align}
\cupprod(\Xfunc)([a,b]) 
=\,& \len\left(\mathbf{Im}\left(\Hfunc^*(\X_b)\to \Hfunc^*(\X_a)\right)\right) \label{eq1}  \\
=\,& \len\left(\langle B \rangle\right) \label{eq2}  \\
=\,& \max \left\{\ell\in\N^*\mid B^{\smile\ell}\neq \{0\}\right\} \label{eq3} 
\end{align}
Eqn.~(\ref{eq1}) follows from the definition of the persistent cup-length function, and Eqn.~(\ref{eq2}) is a direct application of Prop.~\ref{prop:generators-of-ring}, where $\langle\cdot\rangle$ denotes the generating set of a ring. 
Because $B$ linearly spans the image $\mathbf{Im}(\Hfunc^{\geq1}(\X_b)\to \Hfunc^{\geq1}(\X_a))$ in each dimension, the assumption of Prop.~\ref{prop:len_of_basis} is satisfied and thus Eqn.~(\ref{eq3}) follows. 

\medskip

Given $J_1,\dots,J_\ell\in \caB_{\geq1}(\Xfunc)$ such that $J_i\supseteq[a,b]$ for each $i$, we claim that
\[[\sigma_{J_1} ]_a\smile\dots\smile[\sigma_{J_\ell} ]_a\neq 0 \iff \operatorname{supp}(\sigma_{J_1} \smile\dots\smile \sigma_{J_\ell})\supset [a,b].\]
The `$\Leftarrow$' is trivial. As for `$\Rightarrow$', recall from Prop.~\ref{prop:support} that in this case the support is a non-empty interval with its right end equal to the right end of $\cap_i J_i\supset[a,b]$. It follows that the support, as an interval, contains both $a$ and $b$, and thus containing $[a,b]$.

Therefore, we have Eqn.~(\ref{eq4}) below:
\begin{align}
&\cupprod(\Xfunc)([a,b]) \nonumber\\
=\, & \max \left\{\ell\in\N^*\mid B^{\smile\ell}\neq \{0\}\right\} \nonumber\\
=\,& \max \left\{\ell\in\N^*\mid [\sigma_{J_1} ]_a\smile\dots\smile[\sigma_{J_\ell} ]_a\neq 0,\,J_i\supseteq[a,b],\, J_i\in \caB_{\geq1}(\Xfunc),\, \forall i=1,\dots,\ell \right\} \nonumber\\
=\,& \max \left\{\ell\in\N^*\mid \operatorname{supp}(\sigma_{J_1}\smile\dots\smile\sigma_{J_\ell})\supseteq[a,b],\,J_i\in\caB_{\geq1}(\Xfunc),\, \forall i=1,\dots,\ell \right\} \label{eq4}\\
 =\,& \max_{[c,d]\supseteq[a,b]} \left\{\max\left\{\ell\in\N^*\mid [c,d]=\operatorname{supp}(\sigma_{J_1}\smile\dots\smile\sigma_{J_\ell}) \text{, where }J_i\in\caB_{\geq1}(\Xfunc)\right\}
 \right\} \nonumber\\
 =\,& \max_{[c,d]\supseteq[a,b]} \left\{\max\left\{\ell\in\N^*\mid [c,d]=J_1*_{\bsigma}\dots *_{\bsigma} J_\ell \text{, where }J_i\in\caB_{\geq1}(\Xfunc)\right\}
 \right\} \label{eq-second-last}\\
 =\,& \max_{[c,d]\supseteq[a,b]}\mathbf{dgm}_{\bsigma}^{\smile}(\Xfunc)([c,d]). \label{eq-last}
\end{align}
Here Eqn.~(\ref{eq-second-last}) and Eqn.~(\ref{eq-last}) follow from the definition of the $*_{\bsigma}$ operation (Defn.~\ref{dfn:ell-product}) and the definition of the persistent cup-length diagram (Defn.~\ref{def:cup_dgm}), respectively.
\end{proof} 

\endsofsupport*
\begin{proof} 
\label{pf:support}
We prove in the case of closed intervals. For the other types of intervals, the statement follows from a similar discussion. 

Let $d$ be the right end of $\cap_{1\leq i\leq \ell}I_i$. Clearly, any $t>d$ is not in $I$, because there is some $I_i$ such that $[\sigma_{I_i}]_t=[0]_t$. To show $d$ is the right end of $I$, it suffices to show that $d$ is in $I$. 
If $d\notin I$, then it follows from $[ \sigma_{I_1}]_d \smile\dots\smile [\sigma_{I_\ell}]_d = [0]_d$ that $[ \sigma_{I_1}]_t \smile\dots\smile [\sigma_{I_\ell}]_t = [0]_t$ for all $t\leq d$. Thus, $I=\emptyset$, which gives a contradiction. Therefore, $d$ is the right end of $I$.

We show that $I$ in an interval, i.e. for any $t\in I$ and $s\in [t,d]$, we have $s\in I$. This is true because $[ \sigma_{I_1}]_s \smile\dots\smile [\sigma_{I_\ell}]_s$, as the preimage of a non-zero element $[ \sigma_{I_1}]_t \smile\dots\smile [\sigma_{I_\ell}]_t$, cannot be zero.

Assume the left end of $I$ is $b$. Then $[ \sigma_{I_1}]_b \smile\dots\smile [\sigma_{I_\ell}]_b\neq 0$ but $[ \sigma_{I_1}]_{b-\epsilon} \smile\dots\smile [\sigma_{I_\ell}]_{b-\epsilon}=0$ for any $\epsilon>0$. Notice that we can write the cup product $[\sigma_{I_1}]_b \smile\dots\smile [\sigma_{I_\ell}]_b=\sum \lambda_J [\sigma_J]_b$ for some coefficients $\lambda_J$ and distinct representative cocycles $\sigma_J$ with $[\sigma_J]_b\neq 0$, where $J\in \caB(\Xfunc)$. For any $\epsilon>0$, it follows from $[ \sigma_{I_1}]_{b-\epsilon} \smile\dots\smile [\sigma_{I_\ell}]_{b-\epsilon}=0$ and the linear independence of $[\sigma_J]_{b-\epsilon}$ that $[\sigma_J]_{b-\epsilon}=0$ for every $J$. Thus, these $J$ are bars with left end equal to $b$. 
\end{proof}

\paragraph*{The cohomology ring of $\mathbb{RP}^2$} \label{app:computation for RP2}

Consider the cell complex structure on $\mathbb{RP}^2$, given by one $0$-cell $e^0$, one $1$-cell $e^1$ and one $2$-cell $e^2$, and the attaching maps: $\varphi_1:\partial e^1\to e^0$ is a constant map and $\varphi_2:\partial e^2\to e^1$ identifies the antipodal points. 
By \cite[Thm.~3.19]{hatcher2000}, $\Hfunc^*(\mathbb{RP}^2;\Z_2)\cong \Z_2[\alpha]/(\alpha^3)$.
For another copy of $\mathbb{RP}^2$, denote its cells by $\bar{e}^0$, $\bar{e}^1$ and $\bar{e}^2$ and attaching maps by $\bar{\varphi}_1$ and $\bar{\varphi}_2$. Then $\Hfunc^*(\mathbb{RP}^2;\Z_2)\cong \Z_2[\beta]/(\beta^3)$. 

Define 
\begin{itemize}
    \item the wedge sum $\mathbb{RP}^2\vee \mathbb{RP}^2$ to be the union $\mathbb{RP}^2\cup \mathbb{RP}^2$ quotient by the relation $e^0\sim \bar{e}^0$;
    \item the product $\mathbb{RP}^2\vee \mathbb{RP}^2$ to be a cell complex with cells the products $e^i\times e^j$ and attaching maps $\varphi_i\times \bar{\varphi}_j$. 
\end{itemize}
Notice that there is an inclusion $\mathbb{RP}^2\vee \mathbb{RP}^2\hookrightarrow \mathbb{RP}^2\vee \mathbb{RP}^2$ given by $e^0=\bar{e}^0\mapsto e^0\times e^0$, $e^i\mapsto e^i\times \bar{e}^j$ and $\bar{e}^j\mapsto e^0\times\bar{e}^j.$

For the wedge sum, we have $\Hfunc^*(\mathbb{RP}^2\vee \mathbb{RP}^2;\Z_2)\cong \Z_2[\alpha,\beta]/(\alpha^3,\alpha\beta,\beta^3)$.
For the $2$-skeleton of the product space, we claim that $\Hfunc^*(S_2(\mathbb{RP}^2\times \mathbb{RP}^2);\Z_2)$ is a linear space over $\Z_2$ with basis $\{ 1,\alpha,\beta,\alpha\smile\alpha,\alpha\smile\beta,\beta\smile\beta\}$.
It is clear that $\Hfunc^ {\leq 1}(S_2(\mathbb{RP}^2\times \mathbb{RP}^2);\Z_2)=\Hfunc^ {\leq 1}(\mathbb{RP}^2\times \mathbb{RP}^2;\Z_2)$.
By \cite[Thm.~3.15]{hatcher2000}, the cohomology ring $\Hfunc^*(\mathbb{RP}^2\times \mathbb{RP}^2;\Z_2)$ is isomorphic to $\Z_2[\alpha]/(\alpha^3)\otimes\Z_2[\beta]/(\beta^3)$, It follows that $\{\alpha\smile\alpha,\alpha\smile\beta,\beta\smile\beta\}$ is a linearly independent set of $2$-cocycles, which remains true when the space is truncated up to dimension $2$. 
On the other hand, the 2-skeleton of $\mathbb{RP}^2\times \mathbb{RP}^2$ has only three $2$-cells, implying that the dimension of the $\Hfunc^ {2}(S_2(\mathbb{RP}^2\times \mathbb{RP}^2);\Z_2)$ cannot go over $3$. Thus, $\{\alpha\smile\alpha,\alpha\smile\beta,\beta\smile\beta\}$ forms a basis for $\Hfunc^ {2}(S_2(\mathbb{RP}^2\times \mathbb{RP}^2);\Z_2)$.

\subsection{More examples of persistent cup-length diagrams}\label{Supplement-examp}
\begin{example}[$\mathbf{dgm}_{\bsigma}^{\smile}(\cdot)$ more informative than standard barcode]
\label{ex:sigma-diag-klein}
Recall from Fig.~\ref{fig:klein-function-diagram} the filtration $\mathbf{X}=\{\mathbb{X}_t\}_{t\geq0}$ of the pinched Klein bottle. 
In Fig.~\ref{fig:dgm_sigma-dgm-klein}, we compare the persistent cup-length diagram  $\mathbf{dgm}_{\bsigma}^{\smile}(\Xfunc)$ with the standard persistence diagram $\mathbf{dgm}^{\geq 1}(\Xfunc)$, and see that the former contains more information than the latter. The two diagrams differs only by their values on $[2,3)$. Whereas $\mathbf{dgm}^{\geq 1}(\Xfunc)([2,3))=0$, the value $\mathbf{dgm}_{\bsigma}^{\smile}(\Xfunc)([2,3))$ is non-zero.
    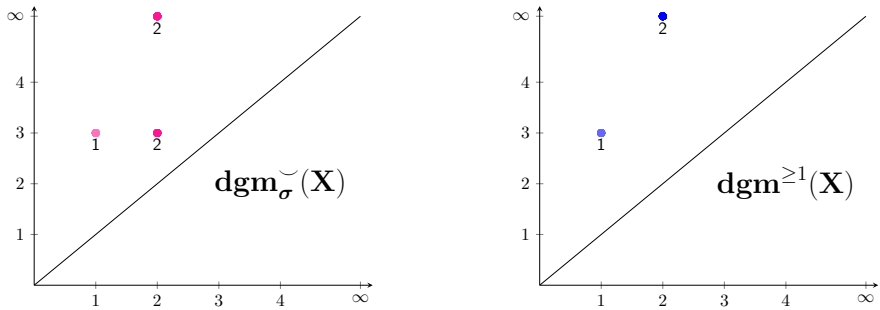
\begin{figure}[H]
    \begin{tikzpicture}[scale=0.65]
    \begin{axis} [ 
    axis y line=middle, 
    axis x line=middle,
    ytick={1,2,3,4,5.3},
    yticklabels={$1$,$2$,$3$,$4$,$\infty$},
    xticklabels={$1$,$2$,$3$,$4$,$\infty$},
    xtick={1,2,3,4,5.3},
    xmin=0, xmax=5.5,
    ymin=0, ymax=5.5,]
    \addplot [mark=none] coordinates {(0,0) (5.3,5.3)};
  \addplot[dgmcolor!60!white,mark=*] (1,3) circle (2pt) node[below,black]{\textsf{1}};
    \addplot[dgmcolor!100!white,mark=*] (2,3) circle (2pt) node[below,black]{\textsf{2}};
  \addplot[dgmcolor!100!white,mark=*] (2,5.3) circle (2pt) node[below,black]{\textsf{2}};
    \node[mark=none] at (axis cs:4,2){\LARGE$\mathbf{dgm}_{\bsigma}^{\smile}(\Xfunc)$};
    \end{axis}
    \end{tikzpicture}
    \hspace{1.5cm}
    \begin{tikzpicture}[scale=0.65]
    \begin{axis} [ 
    axis y line=middle, 
    axis x line=middle,
    ytick={1,2,3,4,5.3},
    yticklabels={$1$,$2$,$3$,$4$,$\infty$},
    xticklabels={$1$,$2$,$3$,$4$,$\infty$},
    xtick={1,2,3,4,5.3},
    xmin=0, xmax=5.5,
    ymin=0, ymax=5.5,]
    \addplot [mark=none] coordinates {(0,0) (5.3,5.3)};
  \addplot[blue!60!white,mark=*] (1,3) circle (2pt) node[below,black]{\textsf{1}};
  \addplot[blue!100!white,mark=*] (2,5.3) circle (2pt) node[below,black]{\textsf{2}};
  \node[mark=none] at (axis cs:4,2){\LARGE$\mathbf{dgm}^{\geq 1}(\Xfunc)$};
    \end{axis}
    \end{tikzpicture}
    \caption{The persistent cup-length diagram $\mathbf{dgm}_{\bsigma}^{\smile}(\Xfunc)$ (left) and the positive-dimension persistence diagrams $\mathbf{dgm}^{\geq 1}(\Xfunc)$ (right), for $\mathbf{X}$ the filtration of a pinched Klein bottle. See Ex.~\ref{ex:sigma-diag-klein}.}
    \label{fig:dgm_sigma-dgm-klein}
    \end{figure}
\end{example}

\begin{example}[$\cupprod(\cdot)$ v.s. $\mathbf{dgm}_{\bsigma}^{\smile}(\cdot)$]
\label{ex:sigma-diag-disks}
Let $\mathbf{X}=\{\X_t\}_{t\geq0}$ be a filtration of the disjoint union of two $2$-disks, as shown in Fig.~\ref{fig:dcircle-filt}.  
Note that the barcodes of $\Xfunc$ in positive dimensions are: $\caB_1(\Xfunc)=\{[0,2),[1,3)\}$ and let $\bsigma:=\{\alpha,\beta\}$ denote the representative cocycles with $\alpha$ and $\beta$ corresponding to the top and bottom circle in Fig.~\ref{fig:dcircle-filt}, respectively. Then we compute 
\[
\mathbf{dgm}_{\bsigma}^{\smile}(\Xfunc)(I):= 
      \begin{cases}
    1, & \mbox{ if } I=[0,2) \text{ or } I=[1,3)\\
    0, & \mbox{otherwise.} 
     \end{cases}
\]
We compute persistent cup-length function $\cupprod(\Xfunc)$ using Thm.~\ref{thm:tropical_mobius}, and compare it with the persistent cup-length diagram $\mathbf{dgm}_{\bsigma}^{\smile}(\Xfunc)$ in Fig.~\ref{fig:ex-dgm-disks}.

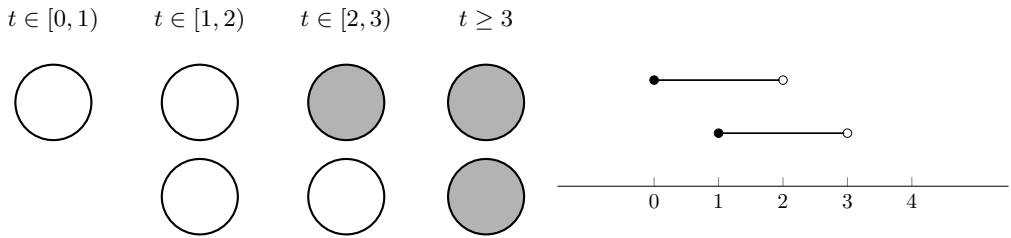
\begin{figure}[H]
\begin{tikzcd}[column sep=small,row sep=tiny]
    \text{\small{$t\in [0,1)$}} 
    & \text{\small{$t\in [1,2)$}} 
    & \text{\small{$t\in [2,3)$}} 
    & \text{\small{$t\geq 3$}} 
    \\
\begin{tikzpicture}[scale=0.5]
	   \filldraw[fill=none,thick](0,0) circle (1);
    \end{tikzpicture}
	& 
	\begin{tikzpicture}[scale=0.5]
	   \filldraw[fill=none,thick](0,0) circle (1);
	   \filldraw[fill=none,thick](0,-2.5) circle (1);
    \end{tikzpicture}
	& 
	\begin{tikzpicture}[scale=0.5]
	   \filldraw[fill=black!30,thick](0,0) circle (1);
	   \filldraw[fill=none,thick](0, -2.5) circle (1);
    \end{tikzpicture}
    & 
	\begin{tikzpicture}[scale=0.5]
	   \filldraw[fill=black!30,thick](0,0) circle (1);
	   \filldraw[fill=black!30,thick](0,-2.5) circle (1);
    \end{tikzpicture}
\end{tikzcd}
\begin{tikzcd}[column sep=small,row sep=tiny]
\\
  \begin{tikzpicture}[scale=0.8]
    \begin{axis} [ 
    height=4cm,
    width=9cm,
    hide y axis,
    axis x line*=bottom,
    xtick={0,1,2,3,4},
    xticklabels={$0$, $1$, $2$, $3$, $4$}, 
    xmin=-1.5, xmax=5.5,
    ymin=0, ymax=1.1,]
        \addplot [mark=none,thick] coordinates {(0.05,.8) (1.95,.8)};
    \addplot [mark=none,thick] coordinates {(1.05,.4) (2.95,.4)};
    \addplot [mark=*] coordinates {(0,.8)};
    \addplot [mark=*] coordinates {(1,.4)};
    \addplot [mark=o] coordinates {(3,.4)};
    \addplot [mark=o] coordinates {(2,.8)};
    \end{axis}
    \end{tikzpicture}
\end{tikzcd}
    \caption{A filtration of the disjoint union of two $2$-disks, and its barcode in positive dimensions, see Ex.~\ref{ex:sigma-diag-disks}.}
\label{fig:dcircle-filt}
\end{figure}
\begin{figure}
    \begin{tikzpicture}[scale=0.65]
    \begin{axis} [ 
    axis y line=middle, 
    axis x line=middle,
    ytick={1,2,3,4,5.3},
    yticklabels={$1$,$2$,$3$,$4$,$\infty$},
    xticklabels={$1$,$2$,$3$,$4$,$\infty$},
    xtick={1,2,3,4,5.3},
    xmin=0, xmax=5.5,
    ymin=0, ymax=5.5,]
    \addplot [mark=none] coordinates {(0,0) (5.3,5.3)};
  \addplot[dgmcolor!60!white,mark=*] (1,3) circle (2pt) node[right,black]{\textsf{1}};
    \addplot[dgmcolor!60!white,mark=*] (0,2) circle (2pt) node[right,black]{\textsf{1}};
    \end{axis}
    \end{tikzpicture}\hspace{1.5cm}
\begin{tikzpicture}[scale=0.65]
    \begin{axis} [ 
    axis y line=middle, 
    axis x line=middle,
    ytick={1,2,3,4,5.3},
    xtick={1,2,3,4,5.3},
    yticklabels={$1$,$2$,$3$,$4$,$\infty$},
    xticklabels={$1$,$2$,$3$,$4$,$\infty$},
    xmin=0, xmax=5.5,
    ymin=0, ymax=5.5,]
    \addplot [mark=none] coordinates {(0,0) (5.3,5.3)};
    \addplot [thick,color=dgmcolor!20!white,fill=dgmcolor!20!white, 
                    fill opacity=0.45]coordinates {
            (1,3) 
            (1,2)
            (0,2)  
            (0,0)
            (3,3)
            (1,3)};
    \node[mark=none] at (axis cs:1.2,1.8){\textsf{1}};
    \end{axis}
    \end{tikzpicture}
    \caption{The persistent cup-length diagram $\mathbf{dgm}_{\bsigma}^{\smile}(\Xfunc)$ (left) and the persistent cup-length function $\cupprod(\Xfunc)$ (right), for $\mathbf{X}$ the filtration of the disjoint union of two ($2$-dim) disks. See Ex.~\ref{ex:sigma-diag-disks}.}
\label{fig:ex-dgm-disks}
\end{figure}
\end{example}

\begin{example} \label{ex:n_torus}
Let $\Xfunc=\{\X_t\}_{t\geq 0}$ be a filtration given by
$\X_t:= \underbrace{\bbS^1\times \bbS^1\times\dots\times \bbS^1}_{\lfloor t\rfloor \text{times}},$ for $t\geq 0$. For any $t\leq s$, $\X_t\hookrightarrow \X_s$ is given by the natural inclusion.
When $t<1$, $\Hfunc^*(\X_t)$=0. When $t\geq 1$, there are $\lfloor t\rfloor$ linearly independent $1$-cocycles $\eta_1,\dots,\eta_{\lfloor t\rfloor}$. 
Notice that the set $\bsigma:=
\{\eta_{i_1}\smile\dots\smile\eta_{i_j}\}_{1\leq i_1\leq \dots\leq i_j\leq \lfloor t\rfloor}$ forms a family of representative cocycles for $\Hfunc^*(\Xfunc)$ in positive dimensions, with which we compute the persistent cup-length diagram $\mathbf{dgm}_{\bsigma}^{\smile}(\Xfunc)$ and present it in left figure of Fig.~\ref{fig:ex_dgm}. Then we apply Thm.~\ref{thm:tropical_mobius} to obtain that $\cupprod(\Xfunc)([t,s]) =\cupprod(\X_t)= \lfloor t \rfloor$, if $1\leq t\leq s$ and $0$ otherwise, and we plot the persistent cup-length function $\cupprod(\Xfunc)$ in the right figure of Fig.~\ref{fig:ex_dgm}.
\begin{figure}
    \begin{tikzpicture}[scale=0.65]
    \begin{axis} [ 
    axis y line=middle, 
    axis x line=middle,
    ytick={1,2,3,4,5,6.3},
    yticklabels={$1$,$2$,$3$,$4$,$5$,$\infty$},
    xticklabels={$1$,$2$,$3$,$4$,$5$,$\infty$},
    xtick={1,2,3,4,5,6.3},
    xmin=0, xmax=6.5,
    ymin=0, ymax=6.5,]
    \addplot [mark=none] coordinates {(0,0) (6.3,6.3)};
    \addplot[dgmcolor!100!white,mark=*] (5,6.3) circle (2pt) node[below,black]{\textsf{5}};
    \addplot[dgmcolor!80!white,mark=*] (4,6.3) circle (2pt) node[below,black]{\textsf{4}};
    \addplot[dgmcolor!60!white,mark=*] (3,6.3) circle (2pt) node[below,black]{\textsf{3}};
    \addplot[dgmcolor!40!white,mark=*] (2,6.3) circle (2pt) node[below,black]{\textsf{2}};
    \addplot[dgmcolor!20!white,mark=*] (1,6.3) circle (2pt) node[below,black]{\textsf{1}};
    \addplot[mark=none] (5.7,6.3) node[color=dgmcolor!100!white]{$\dots$};
    \end{axis}
    \end{tikzpicture}
    \hspace{1.5cm}
    \begin{tikzpicture}[scale=0.65]
    \begin{axis} [ 
    axis y line=middle, 
    axis x line=middle,
    ytick={1,2,3,4,5,6.3},
    yticklabels={$1$,$2$,$3$,$4$,$5$,$\infty$},
    xticklabels={$1$,$2$,$3$,$4$,$5$,$\infty$},
    xtick={1,2,3,4,5,6.3},
    xmin=0, xmax=6.5,
    ymin=0, ymax=6.5,]
    \addplot [mark=none] coordinates {(0,0) (6.3,6.3)};
    \addplot [thick,color=dgmcolor!20!white,fill=dgmcolor!20!white, 
                    fill opacity=0.45]coordinates {
            (1,6.3) 
            (1,1)
            (2,2)  
            (2,6.3)};
    \addplot [thick,color=dgmcolor!40!white,fill=dgmcolor!40!white, 
                    fill opacity=0.45]coordinates {
            (2,6.3) 
            (2,2)
            (3,3)  
            (3,6.3)};
    \addplot [thick,color=dgmcolor!60!white,fill=dgmcolor!60!white, 
                    fill opacity=0.45]coordinates {
            (3,6.3) 
            (3,3)
            (4,4)
            (4,6.3)};
    \addplot [thick,color=dgmcolor!80!white,fill=dgmcolor!80!white, 
                    fill opacity=0.45]coordinates {
            (4,6.3) 
            (4,4)
            (5,5)
            (5,6.3)};
    \addplot [thick,color=dgmcolor!100!white,fill=dgmcolor!100!white,
                    fill opacity=0.45]coordinates {
            (5,6.3) 
            (5,5)
            (6,6)
            (6,6.3)};
    \addplot [thick,color=dgmcolor!100!white,fill=dgmcolor!100!white,
                    fill opacity=0.45]coordinates {
            (6,6.3) 
            (6,6)
            (6.3,6.3)};
    
    \node[mark=none] at (axis cs:5.5,6){\textsf{5}};
    \node[mark=none] at (axis cs:4.5,5){\textsf{4}};
    \node[mark=none] at (axis cs:3.5,4){\textsf{3}};
    \node[mark=none] at (axis cs:2.5,3){\textsf{2}};
    \node[mark=none] at (axis cs:1.5,2){\textsf{1}};
    \end{axis}
    \end{tikzpicture}
\caption{The persistent cup-length diagram $\mathbf{dgm}_{\bsigma}^{\smile}(\Xfunc)$ (left) and the persistent cup-length function $\cupprod(\Xfunc)$ (right), where $\Xfunc$ is a filtration of torus in all dimensions. See Ex.~\ref{ex:n_torus}.} 
\label{fig:ex_dgm}
\end{figure}
\end{example}

\subsection{Supplementary materials for the algorithms}\label{sec:supp_algorithm}

\begin{algorithm}[H]
\SetKwData{Left}{left}\SetKwData{This}{this}\SetKwData{Up}{up}
\SetKwInOut{Input}{Input}\SetKwInOut{Output}{Output}
\Input{A column reduced square matrix $R$.}
\Output{The row operation matrix $U$}

\BlankLine
$m\gets \mathrm{size}(R)$\;
$U\gets I_{m\times m}$\;
$P\gets \mathrm{zeros}(1,m)$\;
\For(\tcp*[f]{$O( m )$}){$j\leq m$}{
    $v\gets j$-th column of $R$\;
    $J\gets \{i:v[i]\neq 0,i\leq j\}$ \tcp*{$O( m )$}
    \If{$J\neq \emptyset$}{
        \For(\tcp*[f]{$O( m )$}){$i\in J[\mathrm{first}:\mathrm{end}-1]$ }{
            $R[i]\gets (R[i]+R[J[\mathrm{end}]])\mod 2$\;
            $U[i,J[\mathrm{end}]]\gets 1$\;
            }
        } 
    }
\Return $U$.
\caption{$\mathrm{RowReduce}(R)$}
\label{alg:row_reduction}
\end{algorithm}

It is clear that the complexity of Alg.~\ref{alg:row_reduction} is $O(m\cdot 2m)=O( m ^2)$.

\iscoboundary*
\begin{proof}[Proof of Prop.~\ref{prop:is_coboundary}]
Take a positive integer $p$. Let $\delta $ denote the coboundary map in $\X$, and let $\delta^p $ denote the dimension-$p$ coboundary map. Then $\delta =\oplus_{p} \delta^p $. Given a $p$-cochain $\sigma$, we can write $\sigma=S^*\cdot y$ for some vector $y\in \Z_2^{m}$. Let $R$ be the column reduced coboundary matrix, let $U$ be the row reduction matrix for $R$, which can be computed by Alg.~\ref{alg:row_reduction}. Then 
\begin{align*}
    \exists \text{ a $(p-1)$-cochain } \tau \text{ s.t. }\sigma=\delta^p \tau &\iff \exists \text{ a $(p-1)$-cochain } \tau \text{ s.t. }\sigma=\delta \tau\\
    &\iff \exists \xi\in \Z_2^{m} \text{ s.t. }y=A  \xi\\
    &\iff \exists \bar{\xi}\in \Z_2^{m} \text{ s.t. }U  y=U R \bar{\xi}\\
    &\iff \{i:(U  y)(i)\neq 0\}\subset
    \mathrm{Pivots}(U R )(=\mathrm{Pivots}(R )).
\end{align*}
\end{proof}


\subparagraph*{Using boundary matrix reduction $\Lambda'=U'A^{a\intercal}V'$.} 
\label{para:boundary_reduction}
The boundary matrix associated to $S$ is given by the anti-transpose $A^{a\intercal}$ of the coboundary matrix $A$. Assume that $R'=A^{a\intercal} V'$
is the reduced matrix of $A^{a\intercal}$ obtained from left-to-right column operations. Let $\mathrm{Pivots}(R ')$ be the set of pivots of $R '$. Using all $i$-th rows for each $i\in \mathrm{Pivots}(R' )$, we do bottom-to-top row reduction on $R$: $U'  R' =\Lambda'$, such that $\Lambda' $ has at most one non-zero element in each row and column, and $U' $ is an upper triangular matrix. Given a $p$-cochain $\sigma=S^*\cdot y$ for some vector $y\in \Z_2^{m}$. Similarly to Prop.~\ref{prop:is_coboundary}, we show that $\sigma$ is a coboundary if and only if $\{i:((V')^{a\intercal} y)(i)\neq 0\}\subset\{j:\text{the j-th column of }R'\neq 0\}.$ Because $A =((V ')^{-1})^{a\intercal}(\Lambda' )^{a\intercal} ((U ')^{-1})^{a\intercal}$,
\begin{align*}
    \exists \text{ a $(p-1)$-cochain } \tau \text{ s.t. }\sigma=\delta^p \tau 
    &\iff \exists \bar{\xi}\in \Z_2^{m} \text{ s.t. }(V') ^{a\intercal} y=(\Lambda') ^{a\intercal}\bar{\xi}\\
    &\iff \{i:(V^{a\intercal}  y)(i)\neq 0\}\subset
    \mathrm{Pivots}((\Lambda') ^{a\intercal})\\
    &\iff \{i:(V^{a\intercal}  y)(i)\neq 0\}\subset
    \{j:\text{the j-th column of }R' \neq 0\}.
\end{align*}

\complexity*
\begin{proof}[Proof of Thm.~\ref{thm:comp-alg3}]
In Alg.~\ref{alg:main}, \textbf{line 9} runs no more than $q_1 \cdot q_{k-1} $ times, due to the definition of $q_1$ and $q_{k-1} $. 
The while loop in \textbf{line 14} runs no more than $\card(\operatorname{b\_time})\leq q_1$ times, and the condition of this while loop involves a matrix multiplication whose complexity is at most $O((m_k)^2)$.
Combined with other comments in Alg.~\ref{alg:main} and the fact that $k $ is a fixed constant,
the total complexity is upper bounded by 
\[O(k )\cdot O( q_1 \cdot q_{k-1} )\cdot O( (m_k )^2\cdot \max\{c_k ,q_1\}) \leq O( (m_k )^2\cdot q_1 \cdot q_{k-1}  \cdot \max\{c_k ,q_1\}).\]

Next, we estimate $q_{k-1} $ and $c_k $ using $q_1$, $m_k$ and $k $. Since each $B_\ell$ consists of $\ell$-fold $*_{\bsigma}$-products of elements in $B_1$, we have $ q_{k-1} = \max_{1\leq \ell\leq k-1}\card(B_\ell)\leq  (q_1 ) ^{k  -1},$ which turns out to be a very coarse bound (see 
Rmk.~\ref{rmk:torus_complexity}). 
On the other hand, $c_k $ as the cost of checking whether a simplex is alive at a given parameter, is at most $m_k $ the number of simplices. Hence, the complexity of Alg.~\ref{alg:main} is upper bounded by $O( (m_k  )^{3}\cdot q_1^k)$. In addition, we have $q_1\leq m_k$, because in the matrix reduction algorithm for computing barcodes, bars are obtained from the pivots of the column reduced coboundary matrix and each column provides at most one pivot. Thus,
$O( (m_k )^2\cdot q_1 \cdot q_{k-1}  \cdot \max\{c_k ,q_1\})\leq O( (m_k  )^{3}\cdot q_1^k)\leq O( (m_k  )^{k +3}).$

Consider the Vietoris-Rips filtration arising from a metric space of $n$ points with the distance matrix $D$. Then \textbf{line 7} of Alg.~\ref{alg:cup_product}, checking whether a simplex $a$ 
(represented by a set of at most $\udim+1$ indices into $[n]$)
is alive at the filtration parameter value $t$, can be done by checking whether $\max( D[a,a])\leq t$, with the constant time complexity $c_\udim=O(\udim^2)$. 
\end{proof}

\subparagraph*{Estimating the parameter $q_{k-1}$ in Alg.~\ref{alg:main}} The inequality $q_{k-1}\leq (m_\udim)^\udim$ is quite coarse in general. For the equality to take place, there need to be $m_k$ many intervals in the barcode, where all $\ell$-fold $*_{\bsigma}$-products of these intervals are non-zero.
Consider a filtration consisting of contractible spaces, where $q_{k-1}$ is always $0$ but $m_k$ can be arbitrarily large. Even in the case when there is a reasonable number of cohomology classes with non-trivial cup products, $q_{k-1}$ can be much smaller than $(m_k)^{k-1}$. See the remark below.

\begin{remark}[$q_{k-1}$ can be much smaller than $(m_k)^{k-1}$] \label{rmk:torus_complexity}
Let $\T^d$ be a $d$-torus with $d>4$, equipped with a cell complex structure such that it has $\binom{ d }{i}$ $i$-dimensional cells for each $i=0,\cdots,d$. The cohomology ring of the $d$-torus has $d$ linearly independent $1$-cocycles, which we denote by $\eta_1,\dots,\eta_d $. In addition, there is a ring isomorphism $\Hfunc^*(\T^d)\cong\Z_2[\eta_1,\dots,\eta_d]/\langle\eta_i^2=0\rangle$, 
see \cite[Ch. 3, Ex. 3.16]{hatcher2000}. Let $\X$ be any triangulation of the $4$-skeleton of $\T^d$. Notice that the number $m_3$ of simplices from dimension $1$ to $4$ in $\X$ is at least $\dim(\Hfunc^{\geq 1}(\X))=O(d^3)$.

For $k=3$, consider the constant filtration $\Xfunc$ of $\X$. 
The set $\bsigma:=
\{\eta_{i_1}\smile\dots\smile\eta_{i_j}\}_{1\leq i_1\leq \dots\leq i_j\leq 3}$ of cochains forms a family of representative cocycles for the filtration $\Hfunc^{\geq1}(\Xfunc)$. By definition, $q_3:=\max_{1\leq \ell\leq 3}\bar{q}_\ell$, where $\bar{q}_\ell$ is the cardinality of $B_\ell:=\left(\caB_{\geq 1}(\Xfunc)\right)^{*_{\bsigma}\ell}$ (see Defn.~\ref{dfn:ell-product}). We then see that
\begin{align*}
\bar{q}_1=\,&\card(B_1)= \binom{ d }{1}+\binom{ d }{2}+\binom{ d }{3}=O( d ^3)\\
\bar{q}_2=\,&\card(B_2)=  \binom{ d }{1}\binom{ d -1}{1}+2\binom{ d }{2}(d-2) =O( d ^3)
\end{align*}
Therefore, $q_2=O(d^3)\leq m_3<(m_3)^2$.  
\end{remark}

\section{Stability of persistent cup-length function}
\label{sec:stability}
\label{sec:Details of the stability theorem}
In \textsection \ref{sec:cont-pers-inv}, we show that the erosion distance $d_{\mathrm{E}}$ (see Defn.~\ref{def:de}) between persistent cup-length functions is stable under the homotopy interleaving $d_{\mathrm{HI}}$ (see Defn.~\ref{def:dhi}) between persistent spaces. Subsequently, for persistent spaces arising from Vietoris-Rips filtrations, we prove that the persistent cup-length function is stable under the Gromov-Hausdorff distance $d_{\mathrm{GH}}$ between metric spaces. We then apply the stability theorem below to compare the $2$-torus $\T^2$ and the wedge sum $\bbS^1\vee\bbS^2\vee\bbS^1$ in Ex.~\ref{ex:T2-wedge}.

Recall that $\Top$ denotes the category of compactly generated weakly Hausdorff topological spaces. 
In this section, we also view the persistent cohomology as a persistent graded algebra (over the base field $K$), in order to refer to both its persistent graded ring structure and its persistent graded vector space structure at the same time. Let $\Alg$ be the category of graded $K$-algebras, and write the (contravariant) cohomology algebra functor as $\Hfunc^*:\Top\to\Alg$.

\homostab*

To prove 
Thm.~\ref{thm:main-stability}, we recall the definitions of the interleaving distance and the homotopy-interleaving distance, together with certain results from persistence theory. 

\paragraph*{Interleaving distance and homotopy-interleaving distance}
\begin{definition}[\cite{bubenik2015metrics}]
\label{dfn:interleaving-dist}
Let $\CC$ be any category. 
Let $\Ffunc,\Gfunc:(\mathbb{R},\leq)\to \CC$ be a pair of persistent objects. $\Ffunc,\Gfunc$ are said to be $\varepsilon$-interleaved if there exists a pair of natural transformations $\varphi=(\varphi_t:\mathbf{F}_t\to \mathbf{G}_{t+\varepsilon})_{t\in\mathbb{R}}$ and $\psi=(\psi_t:\mathbf{G}_t\to \mathbf{F}_{t+\varepsilon})_{t\in\mathbb{R}}$, i.e. the diagrams 
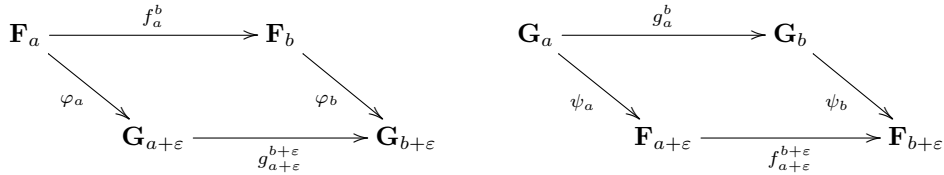
\[
\xymatrix
{
\mathbf{F}_a\ar[dr]_{\varphi_a}\ar[rr]^{f_{a}^{b}} &   & \mathbf{F}_{b}\ar[dr]_{\varphi_b} &  & \mathbf{G}_a\ar[dr]_{\psi_a}\ar[rr]^{g_{a}^{b}} &   & \mathbf{G}_{b} \ar[dr]_{\psi_b}\\
	 & \mathbf{G}_{a+\varepsilon} \ar[rr]_{g_{a+\varepsilon}^{b+\e}} & &
	 \mathbf{G}_{b+\varepsilon}
	 & & \mathbf{F}_{a+\varepsilon} \ar[rr]_{f_{a+\varepsilon}^{b+\e}}
	 & &
	\mathbf{F}_{b+\varepsilon}
	 \\
}
\]
commute for all $a\leq b$ in $\R$; and such that the diagrams
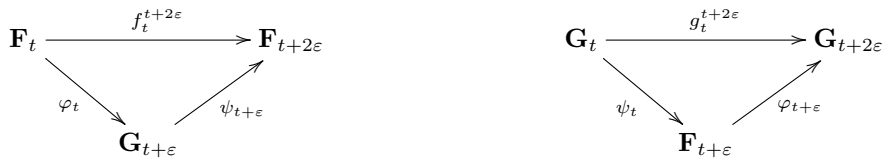
\[
\xymatrix
{
\mathbf{F}_t\ar[dr]_{\varphi_t}\ar[rr]^{f_{t}^{t+2\varepsilon}}  &  & \mathbf{F}_{t+2\varepsilon} & & & \mathbf{G}_t\ar[dr]_{\psi_t}\ar[rr]^{g_{t}^{t+2\varepsilon}}  &  & \mathbf{G}_{t+2\varepsilon} \\
	 & \mathbf{G}_{t+\varepsilon} \ar[ur]_{\psi_{t+\varepsilon}} & & & & & \mathbf{F}_{t+\varepsilon} \ar[ur]_{\varphi_{t+\varepsilon}}\\
}
\]
commute for all $t\in\mathbb{R}$.
The interleaving distance between $\mathbf{F}$ and $\mathbf{G}$ is
\[d_{\mathrm{I}}^{\CC}(\mathbf{F},\mathbf{G}):= \inf\lbrace \varepsilon\geq 0\mid\text{ there is an }\e\text{-interleaving between }\mathbf{F}\text{ and }\mathbf{G}\rbrace.\]
\end{definition}

\begin{theorem}[Bubenik et al.~\cite{bubenik2015metrics}]
\label{thm:Bubenik}
For functors $\Ffunc,\Gfunc:(\R,\leq)\to\CC$ and $\Hfunc:\CC\to \DD$,  
\[ d_{\mathrm{I}}^{\DD}(\Hfunc\Ffunc,\Hfunc\Gfunc)\leq d_{\mathrm{I}}^{\CC}(\Ffunc,\Gfunc).\]
\end{theorem}

From Thm.~\ref{thm:Bubenik}, we obtain the stability of persistent cohomology and truncation functors under the interleaving distance between persistent spaces, summarized as the corollary below.
\begin{corollary}
For any $\Xfunc,\Yfunc:(\R,\leq)\to\Top$,  
\[d_{\mathrm{I}}^{\Ring }(\Hfunc^* \Xfunc,\Hfunc^* \Yfunc)\leq d_{\mathrm{I}}^{\Alg }(\Hfunc^* \Xfunc,\Hfunc^* \Yfunc)\leq d_{\mathrm{I}}^{\Top}(\Xfunc,\Yfunc).\]
Let $(\cdot)^k$ be the $k$-truncation functor on the category of simplicial complexes (viewed as a subcategory of $\Top$) such that $(\cdot)^k(\X)$ is the $\udim$-skeleton of a simplicial complex $\X$, and given a map $f:\X\to \Y,$ $(\cdot)^k(f):=f|_{\X^{\udim}}$. If $\Xfunc,\Yfunc$ are filtrations of simplicial complexes, then 
\[d_{\mathrm{I}}^{\Ring }\left(\Hfunc^* \big(\Xfunc^k\big),\Hfunc^* \big(\Yfunc^k\big)\right)
\leq d_{\mathrm{I}}^{\Alg }\left(\Hfunc^* \big(\Xfunc^k\big),\Hfunc^* \big(\Yfunc^k\big)\right)
\leq d_{\mathrm{I}}^{\Top}\left(\Xfunc^k,\Yfunc^k\right)
\leq d_{\mathrm{I}}^{\Top}(\Xfunc,\Yfunc).\]
\end{corollary}

Following the terminology in \cite{blumberg2017universality}, a pair of persistent spaces $\Xfunc,\Yfunc:(\R,\leq)\to\Top$ are called \textbf{weakly equivalent},
denoted by $\Xfunc\simeq\Yfunc$, if there exists a persistent space $\Zfunc:(\R,\leq)\to\Top$ and a pair of natural transformations
$\varphi:\Zfunc\Rightarrow\Xfunc$ and $\psi:\Zfunc\Rightarrow\Yfunc$
such that for each $t\in\R$, the maps $\varphi_t:\Zfunc_t\to\Xfunc_t$ and $\psi_t:\Zfunc_t\to \Yfunc_t$ are weak homotopy equivalences, i.e., they induce isomorphisms on all homotopy groups. 

\begin{remark}\label{rmk:w.h.e.-cohomolgoy}
Weak homotopy equivalence preserves cohomology algebras. Indeed, let $h:\Y\to \X$ be a weak homotopy equivalence. By \cite[Prop.~4.21]{hatcher2000}, the map $h$ induces a graded linear isomorphism $\Hfunc^*(h):\Hfunc^*(\X)\to\Hfunc^*(\Y)$. On the other hand, the induced map $\Hfunc^*(h)$ preserves the cup product operation. Thus, $\Hfunc^*(h)$ is a graded algebra isomorphism.
\end{remark}

\begin{definition}[The homotopy interleaving distance, \cite{blumberg2017universality}]
\label{def:dhi}
Let $\Xfunc,\Yfunc:(\mathbb{R},\leq)\to \mathbf{Top}$ be a pair of persistent spaces.
The \textit{homotopy interleaving distance of $\Xfunc,\Yfunc$} is
\[d_{\mathrm{HI}}(\Xfunc,\Yfunc)=\inf\left\{d_{\mathrm{I}}^{\Top}(\Xfunc',\Yfunc')\hspace{0.5em}\mid\Xfunc'\simeq \Xfunc\text{ and }\Yfunc'\simeq \Yfunc\right\}.\]
\end{definition}

\subsection{Stability of persistent cohomology rings and persistent invariants}\label{sec:cont-pers-inv}

It follows from \cite[Rmk.~105]{ginot2019multiplicative} that the interleaving distance between persistent cohomology rings is stable under the homotopy-interleaving distance between filtrations, where the proof makes use of the natural $E_{\infty}$-algebra structure (an algebraic structure where the multiplication is associative and commutative up to all higher homotopies) on the cochain complexes. We give a direct proof of Thm.~\ref{thm:cohomology stability} which takes place at the more basic level of the persistent cohomology ring (or algebra).

\begin{theorem}
\label{thm:cohomology stability}
For any pair of persistent spaces $\Xfunc,\Yfunc:(\mathbb{R},\leq)\to \mathbf{Top}$, we have that 
\[d_{\mathrm{I}}^{\Ring }(\Hfunc^* \Xfunc,\Hfunc^* \Yfunc)\leq d_{\mathrm{I}}^{\Alg }(\Hfunc^* \Xfunc,\Hfunc^* \Yfunc)\leq d_{\mathrm{HI}}(\Xfunc,\Yfunc).\]
\end{theorem}

\begin{proof}
Let $\Xfunc',\Yfunc':(\R,\leq)\to\Top$ be any pair of persistent spaces such that $\Xfunc'\simeq \Xfunc$ and  $\Yfunc'\simeq \Yfunc$.
By Thm.~\ref{thm:Bubenik} we have that
\[d_{\mathrm{I}}^{\Alg }(\Hfunc^* \Xfunc',\Hfunc^* \Yfunc')\leq d_{\mathrm{I}}^{\Top}(\Xfunc',\Yfunc').\]
Since $\Xfunc'\simeq \Xfunc$ and  $\Yfunc'\simeq \Yfunc$, there exists persistent spaces $\Zfunc,\Wfunc:(\R,\leq)\to\Top$ and natural transformations $\Xfunc\xLeftarrow{\varphi}\Zfunc\xRightarrow{\psi}\Xfunc'$ and $\Yfunc\xLeftarrow{f}\Wfunc\xRightarrow{g}\Yfunc'$, such that for each $t\in\R$, the maps $\varphi_t:\Zfunc_t\to\Xfunc_t$, $\psi_t:\Zfunc_t\to\Xfunc'_t$, $f_t:\Wfunc_t\to\Yfunc_t$, and $g_t:\Wfunc_t\to\Yfunc'_t$, are weak homotopy equivalences.
Because the cohomology algebra functor $\Hfunc^* $ is invariant of weak homotopy equivalence of spaces (see Rmk.~\ref{rmk:w.h.e.-cohomolgoy}), then for each $t\in\R$, the graded algebra homomorphisms $\Hfunc^*(\varphi_t)$, $\Hfunc^*(\psi_t)$, $\Hfunc^*(f_t)$, and $\Hfunc^*(g_t)$ are actually graded algebra isomorphisms.
In particular, by the functoriality of $\Hfunc^*$, $\Hfunc^*(\varphi):\Hfunc^*\Zfunc\Rightarrow\Hfunc^*(\Xfunc)$, $\Hfunc^*(\psi):\Hfunc^*\Zfunc\Rightarrow\Hfunc^*(\Xfunc)'$, $\Hfunc^*(f):\Hfunc^*\Wfunc\Rightarrow\Hfunc^*(\Yfunc)$, and $\Hfunc^*(g):\Hfunc^*\Wfunc\Rightarrow\Hfunc^*(\Yfunc)'$ are natural isomorphisms of persistent graded algebras.
Thus, the graded algebra isomorphisms $\Hfunc^*(\varphi_t)\circ\Hfunc^*(\psi_t)^{-1}:\Hfunc^*(\Xfunc)'_t\to\Hfunc^*(\X)_t$ and $\Hfunc^*(f_t)\circ\Hfunc^*(g_t)^{-1}:\Hfunc^*(\Yfunc)'_t\to\Hfunc^*(\Yfunc)_t$
now yield natural isomorphisms $\Hfunc^*(\Xfunc)'\cong\Hfunc^*(\Xfunc)$ and $\Hfunc^*(\Yfunc)'\cong\Hfunc^*(\Yfunc)$ of persistent graded algebras.
Hence, 
\[d_{\mathrm{I}}^{\Alg }(\Hfunc^* \Xfunc,\Hfunc^* \Yfunc)=d_{\mathrm{I}}^{\Alg }(\Hfunc^* \Xfunc',\Hfunc^* \Yfunc')\leq d_{\mathrm{I}}^{\Top}(\Xfunc',\Yfunc').\]
Since $\Xfunc',\Yfunc'$ were arbitrary, we obtain
\[d_{\mathrm{I}}^{\Alg }(\Hfunc^*(\Xfunc),\Hfunc^*(\Yfunc))\leq \inf\left\{d_{\mathrm{I}}^{\Top}(\Xfunc',\Yfunc')\hspace{0.5em}\mid\Xfunc'\simeq \Xfunc\text{ and }\Yfunc'\simeq \Yfunc\right\}= d_{\mathrm{HI}}(\Xfunc,\Yfunc).\]
\end{proof}

Next, we recall a notion of distance for comparing a pair of length functions.
\begin{definition}[Erosion distance \cite{patel2018generalized}]
\label{def:de}
Let $\Jfunc_1,\Jfunc_2: \mathbf{Int}\to (\mathbb{N},\geq)$ be two functors. 
$\Jfunc_1,\Jfunc_2$ are said to be \textit{$\e$-eroded} if $\Jfunc_1([a,b])\geq \Jfunc_2([a-\e,b+\e])$ and $\Jfunc_2([a,b])\geq \Jfunc_1([a-\e,b+\e])$, for all $[a,b]\in\Int$.
The \textit{erosion distance of $\Jfunc_1,\Jfunc_2$} is 
\[d_{\mathrm{E}}(\Jfunc_1,\Jfunc_2):=\inf \lbrace \e\geq0
\mid \Jfunc_1,\Jfunc_2\text{ are }\e\text{-eroded}\},
\]
with the convention that $d_{\mathrm E}(\Jfunc_1,\Jfunc_2)=\infty$ if an $\varepsilon$ satisfying the condition above does not exists.
\end{definition}

\begin{theorem}
\label{thm:stability}
Let $\CC$ be a category with images and  $J:\ob(\CC)\to\N$ an inj-surj invariant in $\CC$.
Let $\mathbf{F},\mathbf{G}:(\mathbb{R},\leq)\to \CC$ and let $J(\Ffunc),J(\Gfunc):\Int\to(\N,\geq)$ be their associated persistent invariants.
Then:
\[d_\mathrm{E}(J(\mathbf{F}),J(\mathbf{G}))\leq d_{\mathrm{I}}^{\CC}(\mathbf{F},\mathbf{G}).\]
\end{theorem}

\begin{proof}
Denote by $f_a^b:\mathbf{F}_a\to \mathbf{F}_b$ and $g_a^b:\mathbf{G}_a\to \mathbf{G}_b$, $a\leq b$, the associated morphisms from $\Ffunc$ to $\Gfunc$.
Assume that $\mathbf{F},\mathbf{G}$ are $\e$-interleaved.
Then, there exist two $\R$-indexed families of morphisms $\varphi_t:\mathbf{F}_t\to \mathbf{G}_{t+\e}$ and $\psi_t:\mathbf{G}_t\to  \mathbf{F}_{t+\e}$, which are natural for all $t\in\R$, such that $\psi_{t+\e}\circ \varphi_t=f_t^{t+2\e}$ and $\varphi_{t+\e}\circ \psi_t=g_t^{t+2\e}$, for all $t\in\R$.
Let $[a,b]\in\Int$. We claim that $J(g_{a-\e}^{b+\e})\leq J(f_a^b)$. If we show this, then similarly we can show the symmetric inequality, and therefore obtain that $J(\mathbf{F}),J(\mathbf{G})$ are $\e$-eroded.
Indeed, the claim is true because
\begin{align*}
    J(g_{a-\e}^{b+\e})&=J(g_{b-\e}^{b+\e}\circ g_{a-\e}^{b-\e})\\
    &=J(\varphi_{b}\circ \psi_{b-\e}\circ g_{a-\e}^{b-\e})\\
    &\leq J(\psi_{b-\e}\circ g_{a-\e}^{b-\e})\hspace{1em}(\text{by  Lem.~\ref{lem:gf}})\\
    &=J(f_{a}^{b}\circ \psi_{a-\e} )\hspace{1em}(\text{by naturality of }\psi)\\
    &\leq J(f_{a}^{b}) \hspace{1em}(\text{by  Lem.~\ref{lem:gf}}).
\end{align*}
\end{proof}

We now prove Thm.~\ref{thm:main-stability}:
\begin{proof}[Proof of Thm.~\ref{thm:main-stability}]
\label{pf:main-stability}
Let $\Xfunc,\Yfunc:(\R,\leq)\to\Top$ be two persistent spaces. Recall from Prop.~\ref{prop:injective-surjective-length} that the length invariant $R\mapsto \len(R)$ is an injective-surjective-invariant of persistent graded rings.
Then, Eqn.~(\ref{eq:dE-dHI}) follows from
\[ d_{\mathrm{E}}(\cupprod(\Xfunc),\cupprod(\mathbf{Y}))=d_\mathrm{E}(\len(\Hfunc^* (\Xfunc)),\len(\Hfunc^* (\Yfunc))\leq d_{\mathrm{I}}^{\Ring }(\Hfunc^* (\Xfunc),\Hfunc^* (\Yfunc))\leq d_{\mathrm{HI}}(\Xfunc,\Yfunc),\]
where the two inequalities follow from Thm.~\ref{thm:cohomology stability} and Thm.~\ref{thm:stability}, respectively.

Given two compact metric spaces $X$ and $Y$, Eqn.~(\ref{eq:dE-dGH}) follows from Eqn.~(\ref{eq:dE-dHI}) and \cite[Thm.~1.6]{blumberg2017universality}:
\[d_{\mathrm{HI}}(\mathbf{VR}(X),\mathbf{VR}(Y))\leq 2\cdot d_{\mathrm{GH}}(X,Y).\]
\end{proof}

\subsection{Estimating the Gromov-Hausdorff distance via  persistent cup-length functions}
\label{sec:erosion_T2_wedge}
We apply Eqn.~(\ref{eq:dE-dGH}) in Thm.~\ref{thm:main-stability} to estimate the Gromov-Hausdorff distance between $\T^2$ and $\bbS^1\vee \bbS^2\vee\bbS^1$, where the metrics on the two spaces are defined as follows. For the sake of simplicity and without loss of generality, we consider $\Int$ to be the set of closed intervals in this section. 

\begin{example}[$\VR(\T^2)$ v.s. $\VR(\bbS^1\vee\bbS^2\vee\bbS^1)$]
\label{ex:T2-wedge}
Let the $2$-torus $\T^2=\bbS^1\times \bbS^1$ be the $\ell_\infty$-product of two unit geodesic circles.
Let $\bbS^2$ be the unit $2$-sphere, equipped with the geodesic distance, and denote by $\bbS^1\vee \bbS^2\vee \bbS^1$ the wedge sum equipped with the gluing metric. Using the characterization of Vietoris-Rips complex of $\bbS^1$ given by \cite{adamaszek2017vietoris}, we obtain the persistent cup-length function of $\VR(\bbS^1)$. Combined with Prop.~\ref{prop:prod-coprod-pers-cup}, we obtain the persistent cup-length function of $\VR(\T^2)$ (see Fig.~\ref{fig:torus_sphere_VR}): for any interval $[a,b]$,
\begin{align*}
    \cupprod\left(\VR(\T^2)\right)([a,b])=
      \begin{cases}
      2, &\mbox{ if $[a,b] \subset \left(\frac{l}{2l+1}2\pi ,\frac{l+1}{2l+3}2\pi \right)$ for some $ l=0,1,\dots$}\\
      0, &\mbox{ otherwise.}
      \end{cases}
\end{align*}

For the persistent cup-length function of $\VR(\bbS^1\vee \bbS^2\vee \bbS^1)$, recall \cite[Prop.~3.7]{adamaszek2020homotopy}: the Vietoris-Rips complex of a metric gluing is the wedge sum of Vietoris-Rips complexes. Applying Prop.~\ref{prop:prod-coprod-pers-cup}, we have for any interval $[a,b]$,
\begin{align*}
    \cupprod\left(\VR(\bbS^1\vee \bbS^2\vee \bbS^1)\right)([a,b])
    =\max \left\{ \cupprod(\VR(\bbS^1))([a,b]),\cupprod(\VR(\bbS^2))([a,b])\right\} 
\end{align*}
We now compute $\cupprod(\VR(\bbS^2))$. 
For any $r\geq \pi=\diam(\bbS^2)$, $\VR_r(\bbS^2)$ is contractible. For any $r\in (0,\zeta)$, where $\zeta:=\arccos(-\tfrac{1}3)\approx 0.61\pi$, it follows from \cite[Cor.~7.1]{lim2020vietoris} that $\VR_r(\bbS^2)$ is homotopy equivalent to $\bbS^2$. Thus, $\cupprod(\VR(\bbS^2))([a,b])=1,\forall [a,b]\subset(0,\zeta)$, implying
\[\cupprod(\VR(\bbS^1\vee \bbS^2\vee \bbS^1))([a,b])=1,\forall [a,b]\subset(0,\zeta).\]
Due to the lack of knowledge of the homotopy type of $\VR_r(\bbS^2)$ for $r$ close to $\pi$, we are not able to compute the complete function $\cupprod(\VR(\bbS^2))$, nor $\cupprod(\VR(\bbS^1\vee \bbS^2\vee \bbS^1))$. However, we are still able to evaluate the erosion distance of $\cupprod(\VR(\bbS^1\vee \bbS^2\vee \bbS^1))$ and $\cupprod(\VR(\T^2))$, as in Prop.~\ref{prop:erosion-comp}. See page~\pageref{pf:erosion-comp} in \S \ref{sec:Details of the stability theorem} for the proof of Prop.~\ref{prop:erosion-comp}.

\end{example}

\begin{restatable}{proposition}{torusvswedge} 
\label{prop:erosion-comp}
For the $2$-torus $\T^2$ and the wedge sum space $\bbS^1\vee \bbS^2\vee \bbS^1$, 
\[\frac{\pi}{3}=d_{\mathrm{E}}\left(\cupprod(\VR(\T^2)),\cupprod(\VR(\bbS^1\vee\bbS^2\vee\bbS^1))\right)\leq 2\cdot d_{\mathrm{GH}}\left(\T^2,\bbS^1\vee\bbS^2\vee\bbS^1\right).\]
\end{restatable}

\begin{proof}[Proof of Prop.~\ref{prop:erosion-comp}]
\label{pf:erosion-comp}
For the simplicity of notation, we denote 
\[\Jfunc_{\times}:=\cupprod(\VR(\T^2))\text{ and }\Jfunc_{\vee}:=\cupprod(\VR(\bbS^1\vee\bbS^2\vee\bbS^1)).\] 
And for an interval $I=[a,b]$, we denote $I^{\epsilon}:=[a-\epsilon,b+\epsilon]$.

Suppose that $\Jfunc_{\times}$ and $\Jfunc_{\vee}$ are $\epsilon$-eroded, which means $\Jfunc_{\times}(I)\geq \Jfunc_{\vee}(I^\epsilon)$ and $\Jfunc_{\vee}(I)\geq \Jfunc_{\times}(I^\epsilon)$, for all $I\in\Int$. We take $I_0:=[\tfrac{\pi}{3}-\delta,\tfrac{\pi}{3}+\delta]$ for $\delta$ sufficiently small, so that the associated point of $I_0$ in the upper-diagonal half plane is very close to the point $(\tfrac{\pi}{3},\tfrac{\pi}{3})$. Then, we have 
\[\Jfunc_{\vee}(I_0)= 1\leq \Jfunc_{\times}(I_0^t)=2\text{, for any }t<\tfrac{\pi}{3}-2\delta.\]
Therefore, in order for the inequality $\Jfunc_{\vee}(I_0)\geq \Jfunc_{\times}(I_0^\epsilon)$, it must be true that $\epsilon\geq \tfrac{\pi}{3},$ implying that $d_{\mathrm{E}}(\Jfunc_{\times},\Jfunc_{\vee})\geq\tfrac{\pi}{3}.$

Next, we prove the inverse inequality $d_{\mathrm{E}}(\Jfunc_{\times},\Jfunc_{\vee})\leq\tfrac{\pi}{3}$. Fix an arbitrary $\epsilon>\tfrac{\pi}{3}$. We claim that $\Jfunc_{\vee}(I^\epsilon)=0$ for all $I\in \Int$. As before, let $\zeta:=\arccos\left(-\frac{1}{3}\right)\leq\tfrac{2\pi}{3}$. Notice that the longest possible bar in the barcode for $\VR(\bbS^1\vee\bbS^2\vee\bbS^1)$ is $(0,\pi)$, and any bar $I'$ in the barcode, except for the bar $(0,\zeta)$, is a sub-interval of $[\zeta,\pi)$. Thus, the length of $I'$ is less than or equal to $(\pi-\zeta)<\zeta<2\epsilon$. For any $I\in \Int$, because the interval $I^\epsilon$ has length larger than $2\epsilon$, it cannot be contained in a bar from the barcode. Thus, $\Jfunc_{\vee}(I^\epsilon)=0$. We can directly check that a similar claim holds for $\Jfunc_{\times}$ as well, i.e. $\Jfunc_{\times}(I^\epsilon)=0$ for all $I\in \Int$. Therefore, for any $I\in \Int$,
\[ \Jfunc_{\times}(I)\geq \Jfunc_{\vee}(I^\epsilon)=0\text{  and  }\Jfunc_{\vee}(I)\geq \Jfunc_{\times}(I^\epsilon)=0.\]
In other words, $\Jfunc_\times$ and $\Jfunc_\vee$ are $\epsilon$-eroded, for any $\epsilon>\tfrac{\pi}{3}$. Thus, $d_{\mathrm{E}}(\Jfunc_{\times},\Jfunc_{\vee})\leq\tfrac{\pi}{3}.$
\end{proof}

\begin{remark} \label{rmk:inter-torus-wedge}
Denote by $\Hfunc _*\left( \cdot\right)$ the persistent homology functor in all dimensions. Then,
\begin{itemize}
    \item $\Hfunc _*\left( \VR(\T^2) \right)\big|_{(0,\zeta)}\cong \Hfunc _*\left( \VR(\bbS^1\vee \bbS^2\vee \bbS^1)\right)\big|_{(0,\zeta)}$, and
    \item $\Hfunc _*\left( \VR(\T^2) \right)\big|_{(\pi,\infty)}= \Hfunc _*\left( \VR(\bbS^1\vee \bbS^2\vee \bbS^1)\right)\big|_{(\pi,\infty)}$ is trivial
\end{itemize}
Thus, the interleaving distance $\di$ between persistent homology of $\T^2$ and $\bbS^1\vee \bbS^2\vee \bbS^1$ in any dimension $p$ satisfies
\[ \di\left( \Hfunc _p\left( \VR(\T^2) \right), \Hfunc _p\left( \VR(\bbS^1\vee \bbS^2\vee \bbS^1)\right) \right)\leq \tfrac{\pi-\zeta}{2}<\tfrac{\pi}{3}.\]

By providing a better bound for the Gromov-Hausdorff distance than the one given by persistent homology, the persistent cup-length function demonstrates its  strength in terms of discriminating spaces and capturing additional important topological information.
\end{remark}

\end{appendices}
\end{document}